\input epsf
\documentclass[reqno]{amsart}
\usepackage{graphicx}
\usepackage{amscd}
\usepackage{amssymb}
\usepackage{amstext}
\usepackage{amsmath}

\theoremstyle{my}

\theoremstyle{my}
\newtheorem{thm}{Theorem}[section]
\newtheorem{Theorem}[thm]{Theorem}
\newtheorem*{Theorem*}{Theorem}
\newtheorem{Corollary}[thm]{Corollary}

\newtheorem*{corollary*}{Corollary}
\newtheorem{Lemma}[thm]{Lemma}

\newtheorem{Proposition}[thm]{Proposition}

\newtheorem*{conjecture*}{Conjecture}
\newtheorem{Problem}[thm]{Problem}
\newtheorem*{question*}{Question}
\newtheorem{defn}[thm]{Definition}
\newtheorem{Claim}[thm]{Claim}

\theoremstyle{definition}
\newtheorem{Definition}[thm]{Definition}

\newtheorem*{definitions*}{Definitions}
\newtheorem{rem}[thm]{Remark}
\newtheorem*{rem*}{Remark}
\newtheorem{Remark}[thm]{Remark}

\newtheorem*{remark*}{Remark}

\newtheorem*{remarks*}{Remarks}
\newtheorem*{example*}{Example}
\newtheorem{Example}[thm]{Example}
\newtheorem*{examples*}{Examples}

\newtheorem*{convention*}{Convention}
\newtheorem*{conventions*}{Conventions}
\newtheorem*{Note*}{Note}
\newtheorem*{exercise*}{Exercise}
\newtheorem*{bibliographical-note*}{Bibliographical note}
\newtheorem{cond}[thm]{Conditions}

\theoremstyle{definition}

\newcommand{\R}{\mathbb{R}}
\newcommand{\Z}{\mathbb{Z}}
\newcommand{\Q}{\mathbb{Q}}
\newcommand{\C}{\mathbb{C}}

\renewcommand{\P}{\mathbb{P}}

\title[Lagrangian Floer theory on compact toric manifolds]
 {Lagrangian Floer theory on compact toric manifolds: survey} 

\author{Kenji Fukaya, Yong-Geun Oh, Hiroshi Ohta, Kaoru Ono}

\begin{document}

\begin{abstract}
This is a survey of a series of papers \cite{toric1,toric2,toric3}.
We discuss the calculation of the Floer cohomology
of Lagrangian submanifold which is a 
$T^n$ orbit in a compact toric manifold.
Applications to symplectic topology and to mirror symmetry
are also discussed.
\end{abstract}

\maketitle

\tableofcontents

\section{Introduction}
\label{intro}

This is a survey of a series of papers \cite{toric1,toric2,toric3} we 
have written about Lagrangian Floer theory of toric manifolds 
and their mirror symmetry.
Our main purpose is to perform systematic computation of the 
Lagrangian Floer cohomology of the $T^n$ orbit in 
toric manifolds together 
with various operations introduced in \cite{fooo-book} 
Section 3.8
and apply them to mirror symmetry between toric A model and 
Landau-Ginzburg B model and to 
symplectic topology of toric manifolds.
\par
Let $X$ be a compact toric manifold with complex dimension $n$ 
and $L(\text{\bf u})$ a $T^n$ orbit. 
(Here $\text{\bf u}$ is an element of the interior of the moment polytope 
which parametrizes the $T^n$ orbit. See Section \ref{toric} Formula (\ref{L(u)def}).)
We show that 
the number (counted with multiplicity) 
of the pair $(L(\text{\bf u}),b)$ (where $b$ is an element of  
$H^1(L(\text{\bf u});\Lambda_0)/H^1(L(\text{\bf u});2\pi\sqrt{-1}\Z)$)
for which Floer cohomology $HF((L(\text{\bf u}),b),(L(\text{\bf u}),b);\Lambda)$ is nonzero
is equal to the Betti number of $X$.
(Theorem \ref{Jacqcmain1}.)

Such a pair $(L(\text{\bf u}),b)$  corresponds one to one to a critical 
point of certain function $\frak{PO}$, the potential function, 
where $\text{\bf u}$ (the position of $L(\text{\bf u})$) is the valuation of the 
coordinate of the critical points.
Given $X$ the valuation of the critical points of $\frak{PO}$ 
can be calculated by solving explicitly calculable algebraic equations finitely 
many times.
(We illustrate these examples in sections \ref{exa1} and \ref{exa2}. 
We use the result of Cho-Oh \cite{cho-oh} for this calculation.)

This identification is induced by an isomorphism 
between quantum cohomology $QH(X;\Lambda_0)$ of $X$ and
the Jacobian ring $\text{\rm Jac}(\frak{PO})$ of the potential 
function $\frak{PO}$, 
which goes back to Givental \cite{givental1,givental2} and 
Batyrev \cite{batyrev:qcrtm92,B2} in the case when 
$X$ is Fano.
We remark that the rank of $QH(X;\Lambda_0)$ is 
the Betti number of $X$ and the 
rank of $\text{\rm Jac}(\frak{PO})$ is the number of critical 
points of $\frak{PO}$ counted with multiplicity.

The isomorphism $QH(X;\Lambda_0) \cong \text{\rm Jac}(\frak{PO})$
is a ring isomorphism.
In the case  $QH(X;\Lambda)$ is semi-simple, the ring 
$QH(X;\Lambda)$ splits to the product of the copies of the field $\Lambda$ 
and each of the factors corresponds to a 
critical point of $\frak{PO}$. (Proposition \ref{localization}.)

Thus we associate a non-displaceable Lagrangian submanifold $L(\text{\bf u})$
to each of the direct factor of $QH(X;\Lambda)$.
Entov-Polterovich \cite{entov, EP0,entov-pol06,EP:rigid} 
and others \cite{ostrober2,usher:specnumber} associated 
a Calabi quasi-homomorphism 
to each of the  direct factors of $QH(X;\Lambda)$ and 
also a non-displaceable Lagrangian submanifold $L(\text{\bf u})$ to 
such Calabi quasi-homomorphism.
The non-displaceable Lagrangian submanifold associated by 
the theory of Entov-Polterovich coincides with one 
associated by Lagrangian Floer theory, as we prove in \cite{toric4}.
(Our construction and proof are very different from Entov-Polterovich's however.)

The ring isomorphism  $QH(X;\Lambda_0) \cong \text{\rm Jac}(\frak{PO})$
is generalized to the case when we consider big quantum cohomology 
in the left hand side and 
the potential function in Lagrangian Floer theory with bulk 
deformation in the right hand side.
Moreover it intertwines the pairings which is the
Poincar\'e duality pairing in the left hand side and is 
(a version of) residue pairing in the right hand side.
This implies the coincidence of two Frobenius manifold structures.
One is the Frobenius manifold structure induced by  big quantum cohomology which is due to 
Dubrovin \cite{dub}, and the other is one associated to the isolated singularity by 
Saito \cite{Sai83,Msaito}.
This isomorphism is regarded as a version of mirror symmetry 
between Toric A model and Landau-Ginzburg B model.
It is closely related to the story of Hori-Vafa \cite{hori-vafa} and also of 
Givental.
\par
The mirror symmetry between toric manifold and singularity theory have 
been studied by many mathematicians. 
Besides those already mentioned above, here is a list of some of them (this list is not exhaustive).
\par
In this paper we focus on the case in which we study the $A$ model (symplectic geometry 
and pseudo-holomorphic curve) on  toric manifold and the
$B$ model (deformation theory and complex geometry) on singularity theory side.
The papers \cite{Aur07,auroux,barani, chan-leung, cho-oh, gros, gros2, grospand, iritani1, iritani2, iritani3, 
osttyo, taka, 
woodwards}
also deal with that case.
\par
There have been more works in the other side of the story namely 
$B$ model in toric side and $A$ model in singularity theory side.
\cite{abouz1,AKO04,FLTZ,seidel:TS, Ued06,ueyama} are some of the papers on this side.
\par\medskip
\noindent{\bf Acknowledgements}: \hskip0.3cm
KF is supported partially by JSPS Grant-in-Aid for Scientific Research
No. 18104001 and Global COE Program G08,
YO by US NSF grant \# 0904197,
 HO by JSPS Grant-in-Aid
for Scientific Research No. 19340017,
KO by JSPS Grant-in-Aid for
Scientific Research No. 21244002.

\section{Preliminary}
\subsection{Notations and terminologies}\label{notterm}
The 
{\it universal Novikov ring }$\Lambda_0$ is the set of all formal sums
\begin{equation}\label{novformula}
\sum_{i=0}^{\infty} a_i T^{\lambda_i}
\end{equation}
where $a_i \in \C$ and $\lambda_i \in \R_{\ge 0}$ such that
$\lim_{i\to \infty} \lambda_i = \infty$,
and $T$ is a formal parameter.
We allow $\lambda_i \in \R$ in (\ref{novformula})
(namely negative $\lambda_i$)
to define $\Lambda$ which we call {\it universal Novikov field}. It is a field of fraction of $\Lambda_0$.
We require $\lambda_i > 0$ in  (\ref{novformula}) to define $\Lambda_+$,
which is the maximal ideal of $\Lambda_0$.

We define a valuation $\frak v_T$ on $\Lambda$ by

\begin{equation}\label{valuation}
\frak v_T\left(\sum_{i=0}^{\infty} a_i T^{\lambda_i}\right)
=
\inf\{ \lambda_i \mid a_i \ne 0\}.
\end{equation}
(Here we assume $\lambda_i \ne \lambda_j$ for $i\ne j$.)
$\Lambda, \Lambda_0, \Lambda_+$  are complete with respect to $\frak v_T$ and 
$(\Lambda_0,\Lambda_+)$ is a valuational ring with valuation $\frak v_T$.

\begin{rem}\label{noteone}
In \cite{fooo-book} a slightly different Novikov ring $\Lambda_{0,\mathrm{nov}}$ 
which contains another formal parameter $e$ is used. 
The role of $e$ is to adjust all the operators appearing in the story so that they  have 
well-defined degree. ($e$ has degree $2$.) 
In  \cite{toric1,toric2,toric3} 
and this paper we use $\Lambda_0$ since ring theoretical properties of 
$\Lambda_0$ is better than one of $\Lambda_{0,\mathrm{nov}}$.
As a drawback only the parities of various operators are well-defined.
\end{rem}
\par
Let $Z_1,\dots,Z_m$ be variables. We define the 
{\it strongly convergent power series ring} 
\begin{equation}
\Lambda_0\langle\!\langle Z_1,\dots,Z_m\rangle\!\rangle
\nonumber\end{equation}
as the set of all formal sums
$$
\sum_{k_1=0}^{\infty}\cdots
\sum_{k_m=0}^{\infty}
C_{k_1\dots k_m} Z_1^{k_1}\cdots Z_m^{k_m}
$$
where $C_{k_1\dots k_m} \in \Lambda_0$ such that
$$
\lim_{k_1+\dots+k_m\to \infty} \frak v_T(C_{k_1\dots k_m}) = +\infty.
$$
We define {\it strongly convergent Laurent power series ring} 
\begin{equation}
\Lambda_0\langle\!\langle Z_1,Z_1^{-1},\dots,Z_m,Z_m^{-1}\rangle\!\rangle
\nonumber\end{equation}
as the set of all formal sums
$$
\sum_{k_1\in \Z}\cdots
\sum_{k_m\in \Z}
C_{k_1\dots k_m} Z_1^{k_1}\cdots Z_m^{k_m}
$$
where $C_{k_1\dots k_m} \in \Lambda_0$ such that
$$
\lim_{\vert k_1\vert+\dots+\vert k_m\vert\to \infty} \frak v_T(C_{k_1\dots k_m}) = +\infty.
$$
See \cite{BGR} about those rings.
\par
We also define
\begin{equation}
\Lambda\langle\!\langle Z_1,Z_1^{-1},\dots,Z_m,Z_m^{-1}\rangle\!\rangle
=
\Lambda_0\langle\!\langle Z_1,Z_1^{-1},\dots,Z_m,Z_m^{-1}\rangle\!\rangle
\otimes_{\Lambda_0}\Lambda.
\nonumber\end{equation}
The definition of $\Lambda\langle\!\langle Z_1,\dots,Z_m\rangle\!\rangle$ is similar.
\par
Let $C$ be a graded free $\Lambda_0$ module.
The valuation $\frak v_T$ induces a norm on $C$ in an obvious way, 
by which $C$ is complete.
We define its degree shift $C[1]$ by $C[1]^k = C^{k+1}$. 
The shifted degree $\deg'$ is defined by
$$
\deg' x = \deg x -1.
$$
We
put
\begin{equation}\label{BC}
B_kC = \underbrace{C\otimes \dots \otimes C}_{\text{$k$ times}}.
\end{equation}
Let
$
\widehat BC = \widehat{\bigoplus}_{k=0}^{\infty}B_kC
$
be the completed direct
sum of them. 
Let $\frak S_k$ be the symmetric group of order $k!$.
It acts on $B_kC$ by
\begin{equation}
\sigma \cdot (x_1\otimes \dots \otimes x_k)
= (-1)^* x_{\sigma(1)} \otimes \dots \otimes x_{\sigma(k)}
\end{equation}
where $* = \sum_{i<j : \sigma(i) > \sigma(j)} \deg x_i \deg  x_j$.
We define $E_kC$ as the subset of $\frak S_k$ invariant 
elements of  $B_kC$ and put
$
\widehat EC = \widehat{\bigoplus}_{k=0}^{\infty}E_kC
$
its completed direct
sum.
\par
On $BC$ we define a coalgebra structure $\Delta : BC \to (BC)^{\otimes 2}$ by 
\begin{equation}\label{coalg}
\Delta(x_1\otimes \dots \otimes x_k)
= \sum_{i=0}^k 
(x_1\otimes \dots \otimes x_i) \otimes (x_{i+1}\otimes \dots \otimes x_k).
\end{equation}
(Note the summand in the case $i=0$ is 
$1 \otimes (x_1\otimes \dots \otimes x_k).$)
$\Delta$ is coassociative.
\par
We can define $\Delta : EC \to (EC)^{\otimes 2}$ by restriction.
It is coassociative and graded cocommutative.

We also consider a map $\Delta^{k-1}: BC \to (BC)^{\otimes k}$
$$
\Delta^{k-1} = (\Delta \otimes  \underbrace{id \otimes \dots
\otimes id}_{k-2}) \circ (\Delta \otimes  \underbrace{id \otimes
\dots \otimes id}_{k-3}) \circ \dots \circ \Delta.
$$
For an indecomposable element $\text{\bf x} \in BC$, it can be
expressed as
\begin{equation}\label{Deltasymbol}
\Delta^{k-1}(\text{\bf x}) = \sum_c \text{\bf x}^{k;1}_c \otimes
\dots \otimes \text{\bf x}^{k;k}_c
\end{equation}
where $c$ runs over some index set.
We use the same notation for $EC$.

\par
\subsection{Moduli spaces of pseudo-holomorphic disks.}
\label{subsec:moduli}
Lagrangian Floer theory is based on the moduli space of pseudo-holomorphic disks.
We recall its definition below.
See \cite{fooo-book} subsection 2.1.2 for detail.
\par
Let $X = (X,\omega)$ be a symplectic manifold and 
$L$ its Lagrangian submanifold.
We pick a compatible almost complex structure $J$ on $X$.
Let $\beta \in H_2(X,L;\Z)$.
\par
The moduli space
$\mathcal M_{k+1;\ell}^{\text{\rm main}}(\beta)$ is the
compactified moduli space of the genus zero bordered holomorphic maps
$u : (\Sigma,\partial\Sigma) \to (X,L)$,
in class $\beta  \in H_2(X,L(u);\Z)$  with $k+1$ boundary marked points and $\ell$ interior
marked points.
This means the following:
\begin{cond}\label{condphd}
\begin{enumerate}
\item
$\Sigma$ is a connected union of disks and spheres, which we call {\it
 (irreducible) components}.
We assume the intersection of two different irreducible components is 
either one point or empty. The intersection of two disk components is if nonempty, 
a boundary point of both of the components.
The intersection of a disk and a sphere component is an interior point of 
the disk component.
We assume that intersection of  three different components is empty.
We also require $\Sigma$ to be simply connected.
A point which belongs to two different components is called a 
{\it singular point}.
\item 
$u : \Sigma \to X$ is a continuous map which is $J$-holomorphic on each of the components.
$u(\partial \Sigma) \subset L$. Here $\partial \Sigma$ is the union of the boundary 
of disk components.
\item
There are $k+1$ points $z_0,\dots,z_k$ on  $\partial \Sigma$.
(We call them {\it boundary marked points}.)
They are mutually distinct. None of them are singular point.
We require the order of $k+1$ boundary marked points to respect
the counter-clockwise cyclic order of the boundary of
$\Sigma$.
\item
There are $\ell$ points $z^+_1,\dots,z^+_{\ell}$ on  $\Sigma \setminus \partial\Sigma$.
(We call them {\it interior marked points}.)
They are mutually distinct. None of them are singular point.
\item
For each of the components $\Sigma_a$ of $\Sigma$, one of the 
following conditions hold : 
\begin{enumerate}
\item $u$ is not a constant map on $\Sigma_a$.
\item $\Sigma_a$ is a disk component. 
We have $2n_{\text{\rm int}} + n_{\text{\rm bdry}} \ge 3$.
Here $n_{\text{\rm int}}$ is the sum of the numbers of the interior marked 
points and the interior singular points.
$n_{\text{\rm bdry}}$ is the sum of the numbers of the boundary marked 
points and the boundary singular points.
\item $\Sigma_a$ is a sphere component. The sum of the numbers of 
the marked points and the singular points on $\Sigma_a$ is $\ge 3$.
\end{enumerate}
\end{enumerate}
\end{cond}
The condition 5) is called the {\it stability condition}. It is equivalent to the 
condition that the automorphism group of this element is a finite group.
\par
In case $\ell =0$ we write $\mathcal M_{k+1}^{\text{\rm main}}(\beta)$
in place of $\mathcal M_{k+1;0}^{\text{\rm main}}(\beta)$.
\par
We define the evaluation maps
\begin{equation}\label{evaluationmap}
\text{\rm ev}: \mathcal M_{k+1;\ell}^{\text{\rm main}}(\beta)
\to X^{\ell} \times L^{k+1}
\end{equation}
where we put
$$
\text{\rm ev} = (\text{\rm ev}^+,\text{\rm ev})
= (\text{\rm ev}^+_1,\dots,\text{\rm ev}^+_{\ell};\text{\rm ev}_0,\dots,\text{\rm ev}_k),
$$
as follows:
$$
{\rm ev}_i(\Sigma,u) = u(z_i)
$$
where $z_i$ is the $i$-th boundary marked point as in 3).
$$
{\rm ev}^+_i(\Sigma,u) = u(z^+_i)
$$
where $z_i^+$ is the interior marked point as in 4).
\par
Our moduli spaces $\mathcal M_{k+1;\ell}^{\text{\rm main}}(\beta)$
have Kuranishi structure in the sense of \cite{FO} section 5 and \cite{fooo-book} section A1.
\par
Its boundary is described by using fiber product. 
For example, in case $\ell = 0$ we have the equality
\begin{equation}\label{descboundary}
\partial\mathcal M_{k+1}^{\text{\rm main}}(\beta)
= \bigcup_{k_1+k_2=k+1}\bigcup_{\beta_1+\beta_2=\beta}\bigcup_{i=1}^{k_2}
\mathcal M_{k_1+1}^{\text{\rm main}}(\beta_1) \,\,{}_{{\rm ev}_0}\times_{{\rm ev}_i}\,\,
\mathcal M_{k_2+1}^{\text{\rm main}}(\beta_2).
\end{equation}
as spaces with Kuranishi structures. (\cite{fooo-book} subsection 7.1.1.)
\section{A quick review of Lagrangian Floer theory}
\label{Floertheory}

Let $X = (X,\omega)$ be a symplectic manifold and 
$L$ its Lagrangian submanifold. We assume $L$ is oriented and 
 spin. (Actually relative spinness in the sense of \cite{fooo-book}
Definition 1.6 is enough.)
 
 In \cite{fooo-book} Theorem A, we defined a structure of gapped unital filtered $A_{\infty}$ 
 algebra $\{\frak m_k \mid k=0,1,\dots\}$ on the cohomology group $H(L;\Lambda_0)$ of $L$ with 
$\Lambda_0$ coefficient. 

Namely there exists a sequence of operators
$$
\frak m_k : B_kH(L;\Lambda_0)[1] \to H(L;\Lambda_0)[1]
$$
of odd degree\footnote{See Remark \ref{noteone}. Only the parity of the degree is well-defined 
in Floer cohomology over $\Lambda_0$.} (for $k\ge 0$).
\begin{Theorem}\label{ainfalgdef}
\begin{enumerate}
\item 
\begin{equation}\label{Ainfty}
\sum_{k_1+k_2=k+1}\sum_{i=1}^{k_2}
(-1)^{*}\frak m_{k_2}(x_1,\ldots,\frak m_{k_1}(x_i,\ldots,x_{i+k_1-1}),\ldots,x_k)
= 0,
\end{equation}
where $* =\deg' x_1+\dots+\deg' x_{i-1}$.
\item
$\frak m_0(1) \equiv 0 \mod \Lambda_+$.
\item {\rm (Unitality)} $\text{\bf e} = \text{\rm PD}[L] \in H^0(L;\Lambda_0)$ is the strict unit. 
(Here $\text{\rm PD} : H_k(L) \to H^{n-k}(L)$ is the Poincar\'e duality.) Namely
$$\mathfrak m_{k+1}(x_1, \cdots,{\bf e}, \cdots, x_k) = 0  \qquad \text{for $k \geq 2$ or $k=0$}.
$$
and
$$
\mathfrak m_2({\bf e}, x) = (-1)^{\deg x}\mathfrak m_2(x,{\bf e}) = x.
$$
\item {\rm ($G$-gappedness)}
There exists an additive discrete submonoid $G = \{\lambda_i \mid i=0,1,2,\dots\}$ 
$(\lambda_0= 0 < \lambda_1< \lambda_2 < \cdots, \lim_{i\to\infty}\lambda_i = \infty)$
of $\R_{\ge 0}$  
such that our structure is $G$-gapped. Namely $\frak m_k$ is written as 
$$
\frak m_k = \sum_{i=0}^{\infty} T^{\lambda_i} \frak m_{k,i} 
$$
where $
\frak m_{k,i} : B_kH(L;\C)[1] \to H(L;\C)[1]
$
is $\C$-linear.
\item
$\frak m_{2,0}$ coincides with cup product up to sign.
\end{enumerate}
\end{Theorem}
The triple $(C,\{\frak m_k\},\text{\bf e})$ that satisfies 1)-4) of Theorem \ref{ainfalgdef} 
(with $H(L;\Lambda_0)$ being replaced by $C$)
is called a {\it $G$-gapped unital filtered $A_{\infty}$ algebra.}
\par
The operator $\frak m_k$ is constructed by using 
the moduli spaces $\mathcal M_{k+1}^{\text{\rm main}}(\beta)$ as follows.
(Here we use de Rham cohomology, following \cite{toric1,toric2, toric3,f090,f091}.
In \cite{fooo-book} singular homology is used. Morse homology 
version is in \cite{fooo-can}.)
\par
Let $h_1,\dots,h_k$ be differential forms on $L$.
We define a differential form $\frak m_{k,\beta}(h_1,\dots,h_k)$ on $L$ 
as follows:
\begin{equation}\label{defmkbyint}
\frak m_{k,\beta}(h_1,\dots,h_k)
= {\rm ev}_{0 !} ({\rm ev}_1,\dots,{\rm ev}_k)^*(h_1 \times \dots \times h_k)
\end{equation}
for $(k,\beta) \ne (1,0)$.
We use evaluation maps (\ref{evaluationmap}) in (\ref{defmkbyint}).
We put 
$$\frak m_{1,0}(h) = (-1)^{n+\deg h + 1} dh,
$$ 
where $d$ is the de Rham differential. (See \cite{fooo-book} Remark 3.5.8.)
\par
Here we regard $h_1 \times \dots \times h_k$ as a differential 
form on $L^k$. Then the pull back $({\rm ev}_1,\dots,{\rm ev}_k)^*$ 
defines a differential form on $\mathcal M_{k+1}^{\rm{main}}(\beta)$.
The symbol ${\rm ev}_{0 !}$ denotes the integration along the fiber associated to the map 
${\rm ev}_0 : \mathcal M_{k+1}^{\rm{main}}(\beta) \to L$.
We remark that $\mathcal M_{k+1}^{\rm{main}}(\beta)$ 
itself is not necessarily transversal. So it may have 
wrong dimension. However we can use general theory of 
Kuranishi structure to obtain a multisection $\frak s$ (\cite{FO} section 5, \cite{fooo-book} section A1) 
so that the perturbed moduli space 
$ \mathcal M_{k+1}^{\rm{main}}(\beta)^{\frak s}$ 
(that is the zero point set of the multisection $\frak s$) has a 
virtual fundamental chain (over $\Q$). 
However still after perturbation, the map 
${\rm ev}_0 : \mathcal M_{k+1}^{\rm{main}}(\beta)^{\frak s} \to L$ 
may not be a submersion on the perturbed moduli space $ \mathcal M_{k+1}^{\rm{main}}(\beta)^{\frak s}$ .
So we take a continuous family of perturbations written as 
$\{{\frak s}_w\}_{w\in W}$ parametrized by a certain smooth manifold 
$W$ so that 
$$
{\rm ev}^W_0 : \bigcup_{w\in W}\left(\mathcal M_{k+1}^{\rm{main}}(\beta)^{\frak s_w} 
\times \{w\} \right)\to L
$$ 
is a submersion.\footnote{Actually the parameter space $W$ is defined only locally. 
See \cite{toric2} section 12.} So we can justify (\ref{defmkbyint}) as 
$$
\frak m_{k,\beta}(h_1,\dots,h_k)
= {\rm ev}^W_{0 !}\left(({\rm ev}_1,\dots,{\rm ev}_k)^*(h_1 \times \dots \times h_k) \wedge \omega_W\right).
$$
Here $\omega_W$ is a smooth form of degree $\dim W$ on $W$ that 
has compact support and satisfies
$\int_W \omega_W =1$. 
We pull it back to $\bigcup_{w\in W}\left(\mathcal M_{k+1}^{\rm{main}}(\beta)^{\frak s_w} 
\times \{w\} \right)$ in an obvious way.
The fiberwise evaluation map
${\rm ev}^W_0$ is ${\rm ev}_0$ on $\mathcal M_{k+1}^{\rm{main}}(\beta)^{\frak s_w}\times \{w\}$.
\par
We omit the detail of this construction and refer \cite{toric2} section 12 or 
\cite{f090} section 13.
In the toric case, which is the case of our main interest in this article, 
this construction can be simplified in most of the cases. Namely 
$
{\rm ev}_0 :
\mathcal M_{k+1}^{\rm{main}}(\beta)^{\frak s
} \to L
$
itself can be taken to be a submersion (without using continuous family).
See Section \ref{calcu}.
\par
We now put
$$
\frak m_k = \sum_{\beta\in H_2(X,L;\Z)} T^{(\beta \cap [\omega])/2\pi}\frak m_{k,\beta}.
$$
\par
We can use various properties of the moduli space to check Theorem \ref{ainfalgdef}.
In fact, for example, Theorem \ref{ainfalgdef} 1) is a consequece of Formula (\ref{descboundary}) and 
Theorem \ref{ainfalgdef} 4) is a consequence of 
Gromov compactness.
\par
Thus we obtain a structure of $G$-gapped unital filtered $A_{\infty}$ algebra 
on de Rham complex of $L$. Then it induces one on 
{\it cohomology} $H(X,L;\Lambda_0)$, by a purely algebraic result. 
(\cite{fooo-book} Theorem 5.4.2.)
\par
The filtered $A_{\infty}$ algebra $(H(X,L;\Lambda_0),\{\frak m_k \mid k=0,1,\dots\})$
is independent of the choices (such as compatible almost complex structures and 
perturbations etc.)  up to an isomorphism of  a gapped 
unital filtered $A_{\infty}$ algebra, (that is gapped unital filtered $A_{\infty}$ homomorphism which has 
an inverse).
We omit the precise definition of this notion and refer readers to \cite{fooo-book} Definition 3.2.29 and Proposition 5.4.5.
\par
Let $(C,\{\frak m_k \mid k=0,1,\dots\},\text{\bf e})$ be a 
unital filtered $A_{\infty}$ algebra. We define its 
{\it weak Maurer-Cartan scheme} $\widehat{\mathcal M}_{\text{\rm weak}}(C)$ as the set of 
solutions of the equation
\begin{equation}\label{WKMCeq}
\sum_{k=0}^{\infty}\frak m_k(b,\cdots,b) \equiv  0 \mod \Lambda_0\text{\bf e},
\end{equation}
for $b \in C^{odd}$, with $b \equiv 0 \mod \Lambda_+$.
(Here and hereafter $\text{\bf e}$ denotes the unit.)
\par
For $b \in C^{odd}$, with $b \equiv 0 \mod \Lambda_+$, we define $\frak m^b_k$ by
$$
\frak m^b_k(x_1,\dots,x_k)
= \sum_{m_0=0}^{\infty}\cdots\sum_{m_k=0}^{\infty}
\frak m_k(\underbrace{b,\dots,b}_{m_0},x_1,\underbrace{b,\dots,b}_{m_1},\dots,
x_k,\underbrace{b,\dots,b}_{m_k}).
$$
The right hand side converges in $\frak v_T$ topology.
We can show that 
$(C,\{\frak m_k^b \mid k=0,1,\dots\},\text{\bf e})$ is a filtered $A_{\infty}$
algebra.
\par
In our geometric situation, where $C = H(L;\Lambda_0)$, we can 
remove the assumption  $b \equiv 0 \mod \Lambda_+$ 
using a trick due to Cho \cite{Cho} and 
can define $\frak m^b_k$ for any $b \in H^{odd}(L;\Lambda_0)$.
(See \cite{toric1} section 12 for toric case and 
\cite{f091} section 5 for the general case.)
Moreover the left hand side of (\ref{WKMCeq}) makes sense for any $b \in H^{odd}(L;\Lambda_0)$.
In case we need to distinguish it from the case $b \in H^{odd}(L;\Lambda_+)$, 
we denote the former by $\widehat{\mathcal M}_{\text{\rm weak}}(H(L;\Lambda_0);\Lambda_0)$.
\par
It is easy to see that
$\frak m_0^b(1)$ coincides with the left hand side of  (\ref{WKMCeq}).
Therefore if $b \in \widehat{\mathcal M}_{\text{\rm weak}}(C)$ then 
$\frak m_0^b(1) = c\text{\bf e}$ for some $c \in \Lambda_+$.
It follows that
$$
(\frak m_1^b\circ \frak m_1^b)(x) 
= -c\left(\frak m_2^b(\text{\bf e},x) + (-1)^{\deg'x}\frak m_2^b(x,\text{\bf e})\right) = 0.
$$
Here we use Properties \ref{ainfalgdef} 1) in the first equality and 
Properties \ref{ainfalgdef} 3) in the second equality.
Now we define
\begin{Definition}
Let $b \in H^{odd}(L;\Lambda_0)$. We define Floer  cohomology by:
$$
HF((L,b),(L,b);\Lambda_0) = \frac{\text{\rm Ker}(\frak m_1^b)}{\text{\rm Im}(\frak m_1^b)}.
$$
\end{Definition}
$HF((L,b),(L,b);\Lambda)$ is defined by taking $\otimes_{\Lambda_0}\Lambda$.
\par
It is proved in \cite{fooo-book} Proposition 3.7.75 and the discussion right after that (general case, singular homology version) 
\cite{toric2} section 8 (toric case, de Rham homology version) that 
$HF((L,b),(L,b);\Lambda) \ne 0$ implies that $L$ is Hamiltonian non-displaceable.\footnote{We need to take 
$\Lambda$ (not $\Lambda_0$) for the coefficient ring for this statement. 
Actually $HF((L,b),(L,b);\Lambda_0)= 0$ never occurs when Floer cohomology is defined.}
Namely for any Hamiltonian diffeomorphism $F : X \to X$ we have
$F(L) \cap L \ne \emptyset$.
\par
Let $b \in \widehat{\mathcal M}_{\text{\rm weak}}(C)$.
Then there exists $\frak{PO}(b) \in \Lambda_+$ such that
$$
\sum_{k=0}^{\infty}\frak m_k(b,\dots,b) = \frak{PO}(b)\text{\bf e}.
$$
\begin{Definition}\label{Potentialfunc}
We call 
$\frak{PO} : \widehat{\mathcal M}_{\text{\rm weak}}(C) \to \Lambda_+$, 
the {\it potential function}.
\end{Definition}
In the geometric situation we have
$\frak{PO} : \widehat{\mathcal M}_{\text{\rm weak}}(H(L;\Lambda_0);\Lambda_0) \to \Lambda_+$. 
\section{A quick review of toric manifold}
\label{toric}

In this section we review a very small portion of the theory of toric variety.
We explain only the points we use in this article. See for example \cite{fulton}
for an account of toric variety.
\par
Let $(X,\omega,J)$ be a K\"ahler manifold, where $J$ is its complex structure 
and $\omega$ is its K\"ahler form. 
Let $n$ be the complex dimension of $X$. We assume $n$ dimensional real torus $T^n
= (S^1)^n$ acts effectively on $X$ such that $J$ and $\omega$ are preserved by the 
action. We call such $(X,\omega,J)$ a {\it K\"ahler toric manifold}
if the $T^n$ action has a moment map in the sense we describe below.
Hereafter we simply say $(X,\omega,J)$ (or $X$) is a toric manifold.
\par
Let $(X,\omega,J)$ be as above. We say a map
$\pi = (\pi_1,\dots,\pi_n) : X \to \R^n$ is a {\it moment map} if the following holds.
We consider the $i$-th factor $S^1_i$ of $T^n$. (Here $i=1,\dots,n$.)
Then $\pi_i : X\to \R$ is the moment map of the action of  $S^1_i$. 
In other words, we have the following identity  of $\pi_i$
$$
d\pi_i(X) = \omega(X,\tilde{\frak t}),
$$ 
where $\tilde{\frak t}$ is the Killing vector field
associated to the action of the circle
 $S^1_i$ on $X$.
\par
Let $\text{\bf u} \in \text{Int}P$.
Then the inverse image $\pi^{-1}(\text{\bf u})$ is a 
Lagrangian submanifold which is an orbit of the $T^n$ action.
We put
\begin{equation}\label{L(u)def}
L(\text{\bf u}) = \pi^{-1}(\text{\bf u}).
\end{equation}
This is a Lagrangian torus.
The main purpose of this article is to study 
Lagrangian Floer cohomology for such $L(\text{\bf u})$.
\par
\par
It is well-known that $P = \pi(X)$ is a convex polytope.
We can find a finitely many affine functions $\ell_j: \R^n \to \R$ ($j=1,\dots,m$) 
such that
\begin{equation}
P 
= \{ \text{\bf u} \in \R^n \mid \ell_j(\text{\bf u}) \ge 0, \quad \forall j=1,\dots,m\}.
\end{equation}
We put $\partial_j P = \{\text{\bf u} \in P \mid \ell_j(\text{\bf u})= 0\} $ 
and $D_j = \pi^{-1}(\partial_j P)$. 
($\dim_{\R} \partial_j P = n-1$.)
$D_1 \cup \dots \cup D_m$ is called the {\it toric divisor}.
\par
Moreover we may choose $\ell_j$ so that the following holds.
\begin{cond}\label{condellj}
\begin{enumerate}
\item
We put
$$
d\ell_j = \vec v_j = (v_{j,1},\dots,v_{j,n}) \in \R^n.
$$
Then $v_{j,i} \in \Z$.
\item
Let $p$ be a vertex of $P$. Then the number of faces $\partial_j P$ which contain 
$p$ is $n$. Let $\partial_{j_1} P,\dots, \partial_{j_n} P $ be those faces. Then
$\vec v_{j_1},\dots, \vec v_{j_n}$ (which is contained in $\Z^n$ by item 1)) is 
a basis of $\Z^n$.
\end{enumerate}
\end{cond}
\par
The affine function $\ell_j$ has the following geometric interpretation.
Let $\text{\bf u} \in \text{\rm Int}P$.
There exists $m$ elements 
$\beta_j \in H_2(X,L(\text{\bf u});\Z)$ such that 
\begin{equation}\label{charbetaj}
\beta_j \cap D_{j'} =
\begin{cases}
1  & j = j' \\
0  & j \ne j'.
\end{cases}
\end{equation}
Then we have
\begin{equation}\label{elljandarea}
2\pi\ell_j(\text{\bf u}) = \int_{\beta_j} \omega.
\end{equation}
The existence of such $\ell_j$ and the property above is proved in 
\cite{Gu} Theorem 4.5.
(See \cite{toric1} section 2 also.)
\begin{Example}\label{Pn1}
We consider the complex projective space $\C P^n$.
Using homogeneous coordinate $[x_0:x_1:\dots:x_n]$ we define 
$T^n$ action by
$$
(t_1,\dots,t_n)\cdot [x_0:\dots:x_n]
= [x_0:e^{2\pi\sqrt{-1}t_1}x_1:\dots:e^{2\pi\sqrt{-1}t_n}x_n].
$$
(Here we identify $\R/\Z \cong S^1$.)
The moment map $\pi = (\pi_1,\dots,\pi_n)$ is given by 
$$
\pi_i([x_0:\dots:x_n])
= \frac{\vert x_i\vert^2}{\vert x_0\vert^2 + 
\cdots + \vert x_n\vert^2}.
$$
Its moment polytope $P_0$ is a simplex that is:
$$
P_0 = \{(u_1,\dots,u_n) \mid 0 \le u_i, i=1,\dots,n,\,\,\, 
\sum_{i=0}^n u_i \le 1\}.
$$
We have
\begin{equation}\label{ellforPn}
\ell_i(u_1,\dots,u_n) = 
\begin{cases}
u_i  & i\ne 0 \\
1 - \sum_{j=0}^n u_j &i=0.
\end{cases}
\end{equation}
\end{Example}
\begin{Example}\label{P2blowupone1}
We consider $\C P^2$ as above. 
For $1>\alpha >0$, let us consider 
$$P(\alpha) 
= P_0 \setminus \{(u_1,u_2) \in P_0 \mid u_2 > 1-\alpha\}
= \{(u_1,u_2) \in P_0 \mid u_2 \le 1-\alpha\}.
$$
The inverse image $\pi^{-1}(\{(u_1,u_2) \in P_0 \mid u_2 > 1-\alpha\})$
is a ball of radius $\sqrt{\alpha/2}$ centered at $[0:1:0]$.
The boundary of 
$\pi^{-1}(P(\alpha))$ has an induced contact form which is identified with the 
standard contact form of $S^3$. We identify two points on 
$\partial\pi^{-1}(P(\alpha))$ if they lie on the same orbit of Reeb flow.
After this identification we obtain from $\pi^{-1}(P(\alpha))$ 
a symplectic manifold which we write 
$X(\alpha) = \C P^2\# \overline{\C P}^2(\alpha)$.
\par
It is well-known (see for example \cite{msintroduction} section 6.2) and can be proved from the above 
description that $X(\alpha)$ is a 
blow up of $\C P^2$ with K\"ahler form $\omega$ such that 
the symplectic area of the exceptional divisor is $\alpha$.
\par
The $T^2$ action on $\C P^2$ induces a $T^2$ action on 
 $X(\alpha)$ so that it becomes a 
toric manifold. The moment polytope is $P(\alpha)$.
\par
There are $4$ faces of $P(\alpha)$ and 
4 affine functions $\ell_i$ ($i=0,1,2,3$). Three of them are $\ell_0$, 
$\ell_1$, $\ell_2$ as in $(\ref{ellforPn})$. The fourth one is given by
\begin{equation}\label{ell3}
\ell_3(u_1,u_2) = 1 - \alpha - u_2.
\end{equation}
\end{Example}
\begin{Example}\label{P2blowuptwo1}
We can blow up again and may regard a two points blow up of 
$\C P^2$ as a toric manifold. For $\alpha,\alpha' >0$, 
with $\alpha +\alpha' <1$ we consider the polytope
$$
P(\alpha,\alpha') 
= \{(u_1,u_2) \in P_0 \mid u_2 \le 1-\alpha, \,\, 
u_1+u_2 \ge \alpha'\}.
$$
There exists a toric manifold $X(\alpha,\alpha')$ that is a 
two points blow up of $\C P^2$ and whose 
moment polytope is $P(\alpha,\alpha')$.
\par
$P(\alpha,\alpha') $ has 5 faces. There are 5 affine functions 
$\ell_0,\dots,\ell_4$ associated to each of the faces. 
$\ell_0,\ell_1,\ell_2$ are as in $(\ref{ellforPn})$ and 
$\ell_3$ is as in $(\ref{ell3})$. $\ell_4$ is given by
\begin{equation}\label{ell4}
\ell_4(u_1,u_2) = u_1 + u_2 -\alpha'.
\end{equation}
\end{Example}
\par\bigskip
\section{Floer cohomology and potential function of the $T^n$ orbits}
\label{calcu}

In this section we give a description of Floer cohomology of the $T^n$ 
orbit $L(\text{\bf u})$ of the toric manifold $X$. Here $\text{\bf u} \in 
\text{Int} P$ and $P$ is the moment polytope of $X$.

In this toric case the calculation of the Floer cohomology becomes significantly simpler.
This is because in this case the calculation of Floer cohomology 
is reduced to the calculation of the potential function.
Moreover the leading order term of the potential function is calculated by 
the work of Cho-Oh \cite{cho-oh}.
We will explain those points in this section.

We first fix a basis of $H^1(L(\text{\bf u});\Z)$ as follows.
In Section \ref{toric} we fix a splitting $T^n = (S^1)^n$ and the 
associated coordinate $(t_1,\dots,t_n) \in (\R/\Z)^n$.
Let $\text{\bf e}_i \in H^1(T^n;\Z)$ be the element 
represented by $dt_i$ in de Rham cohomology, where 
$t_i$ is the coordinate of the $i$-th factor of $(S^1)^n$.
(Here we identify $S^1$ with $\R/\Z$.)
The elements $\text{\bf e}_i$, $i=1,\dots,n$ form a basis of
$H^1(T^n;\Z) \cong \Z^n$.
Since the $T^n$ action on $L(\text{\bf u})$ is free and transitive, 
we may identify $H^1(T^n;\Z) = H^1(L(\text{\bf u});\Z)$.
Hence we have a basis $\text{\bf e}_i$, $i=1,\dots,n$ 
of $H^1(L(\text{\bf u});\Z)$.

Let $b \in H^1(L(\text{\bf u});\Lambda_0)$.
We can write 
$b = \sum_{i=1}^n x_i \text{\bf e}_i$.
Hence we take $(x_1,\dots,x_n)$ as a coordinate of 
$H^1(L(\text{\bf u});\Lambda_0)$.
We also put $y^{\text{\bf u}}_i = e^{x_i}$. 
\begin{Remark}
The expression $e^{x_i}$ determines an element of $\Lambda_0$ 
in case $x_i \in \Lambda_0$ as follows.
We write $x_i = x_{i,0} + x_{i,+}$ where 
$x_{i,0} \in \C$ and $x_{i,+} \in \Lambda_+$.
Then we put
$$
y^{\text{\bf u}}_i = e^{x_i} = e^{x_{i,0}} \sum_{k=0}^{\infty} x_{i,+} ^k/k!.
$$
Note $e^{x_{i,0}} \in \C$ is defined as usual. 
The sum $\sum_{k=0}^{\infty} x_{i,+} ^k/k!$ converges 
in $\frak v_T$-topology.
\end{Remark}
Now we consider a toric manifold $X$ with its moment polytope $P$. 
We consider affine functions $\ell_j$ ($j=1,\dots,m$).
We define $v_{j,i} \in \Z$ as in Properties \ref{condellj} 1).
We define 
\begin{equation}\label{defzj}
z_j = T^{\ell_j(\text{\bf u})}(y_1^{\text{\bf u}})^{v_{j,1}}\dots (y_n^{\text{\bf u}})^{v_{j,n}}.
\end{equation}
\begin{Theorem}\label{calcPO}
\begin{enumerate}
\item
$H^1(L(\text{\bf u});\Lambda_0)$ is contained in 
$\widehat{\mathcal M}_{\text{\rm weak}}(H(L;\Lambda_0);\Lambda_0)$.
\item
Let $b = \sum x_i\text{\bf x}_i \in H^1(L(\text{\bf u});\Lambda_0)$.
Then we have
\begin{equation}\label{POmainformula}
\frak{PO}(b) 
= z_1+\dots+z_m + \sum_{k=1}^{N} T^{\rho_k}P_k(z_1,\dots,z_m).
\end{equation}
Here $N \in \Z_{\ge 0}$ or $N=\infty$. The numbers $\rho_k >0$ are 
positive and real.
In case $N=\infty$, the sequence of numbers $\rho_k$  goes to $\infty$ as $k$ goes to $\infty$.
$P_k(z_1,\dots,z_m)$ are monomials of $z_1,\dots,z_m$ of degree $\ge 2$ with 
$\Lambda_0$ coefficient.
We remark that $z_j$ is defined from $y_i^{\text{\bf u}}
= e^{x_j}$ by 
$(\ref{defzj})$.
\item If $X$ is Fano then $P_k$ are all zero.
\item The monomials $P_k$ and the numbers 
$\rho_k$ are independent of $\text{\bf u}$
and depends only on $X$.
\end{enumerate}
\end{Theorem}
Item 1) is \cite{toric1} Proposition 4.3 plus the last line of \cite{toric1} section 4.
\par
Item 2) is \cite{toric1} Theorem 4.6 in the form (slightly) improved in 
\cite{toric2} Theorem 3.4.
In \cite{toric1,toric2} this formula is written using $y^{\text{\bf u}}_i$
in place of $z_j$. But it is easy to see that they are the same by the 
identification (\ref{POmainformula}).
We use the result of Cho-Oh \cite{cho-oh} to calculate 
the term $z_1+\cdots+z_m$ in the right hand side of 
(\ref{POmainformula}).
\par
Item 3) is \cite{toric1} Theorem 4.5.
\par
Item 4)  follows from \cite{toric1} Lemma 11.7.
\begin{proof}[Sketch of the proof]
The linear terms $z_j$ in (\ref{POmainformula}) come  from the contribution 
(that is $\frak m_{k,\beta_j}(b,\cdots,b)$)
of $\mathcal M_{1}^{\text{\rm main}}(\beta_j)$ to 
$\frak m_k(b,\cdots,b)$, where $\beta_j \in H_2(X,L(\text{\bf u});\Z)$ is as in 
(\ref{charbetaj}).
Its coefficient $1$ is the degree of the map
\begin{equation}\label{ev0map}
{\rm ev}_0 : \mathcal M_{1}^{\text{\rm main}}(\beta_j) \to L(\text{\bf u}),
\end{equation}
which is calculated by \cite{cho-oh}. 
\par
The term $T^{\rho_k}P_k$ is a contribution of $\mathcal M_{1}^{\text{\rm main}}(\beta)$
for some $\beta$. 
We will assume $\beta \ne 0$ in the rest of the argument.
\par
We can use a $T^n$ equivariant multisection to
define virtual fundamental chain.
To see this we first observe that the $T^n$ action on $\mathcal M_{1}^{\text{\rm main}}(\beta)$
is free. This is because $T^n$ action on $L(\text{\bf u})$ is free and 
(\ref{ev0map}) is $T^n$ equivariant.
Therefore to find a transversal multisection we can proceed as follows. We first take the  quotient with respect to $T^n$ 
action, next find transversal multisection on the quotient space and then lift it.
\par
\par
Let $\frak s$ be a $T^n$ equivariant multisection which is transversal to $0$.
Then $T^n$ acts freely on its zero set $\mathcal M_{1}^{\text{\rm main}}(\beta)^{\frak s}$.
Therefore the dimension of  $\mathcal M_{1}^{\text{\rm main}}(\beta)^{\frak s}$ is not smaller than 
$n$ if it is nonempty. We can show
$$
\dim \mathcal M_{1}^{\text{\rm main}}(\beta)^{\frak s}
= n + \mu(\beta) - 2
$$
where $\mu : H_2(X,L(\text{\bf u});\Z) \to \Z$ is the Maslov index.
It implies that 
$\mu(\beta) \ge 2$ if $\mathcal M_{1}^{\text{\rm main}}(\beta)^{\frak s} \ne \emptyset$.
\par
This is the key point of the proof. 
\begin{Remark}
In case $X$ is Fano, $\mu(\beta) \ge 2$ automatically holds if 
$\mathcal M_{1}^{\text{\rm main}}(\beta) \ne \emptyset$.
But in non-Fano case this holds only after taking $T^n$ equivariant perturbation.
\end{Remark}
Moreover $T^n$ equivariance implies that 
${\rm ev}_0 : \mathcal M_1^{\text{\rm main}}(\beta) \to L(u)$ 
is a submersion if $\mathcal M_1^{\text{\rm main}}(\beta) 
\ne \emptyset$.
Therefore we may use this $T^n$ equivariant $\frak s$ to define $\frak m_{k,\beta}$.
Namely we do not need to use a continuous family of multisections in this case.
\par
Now if $\deg b = 1$ then 
$$
\deg \frak m_{k,\beta}(b,\dots,b) = 2-\mu(\beta) \le 0.
$$
Namely $\frak m_{k,\beta}(b,\dots,b)$ is either $0$ or is proportional to the unit.
This proves item 1).
\par
To study $\frak m_{k,\beta}(b,\dots,b)$ for $\beta \ne \beta_j$, 
we again use the classification of $J$ holomorphic disks in \cite{cho-oh} 
to find that if $\mathcal M_{1}^{\text{\rm main}}(\beta)$ is nonempty 
the homology class $\beta$ is decomposed to a sum of $\beta_j$'s 
($j = 1,\dots,m$)
and sphere bubbles.
Therefore 
$$
\beta = \beta_{j_1} + \dots + \beta_{j_{e}} +\alpha_1 + \dots + \alpha_f
$$
where $b_{j_k}$ is one of $b_j$'s and $\alpha_i \in H_2(X;\Z)$ is represented by 
$J$-holomorphic sphere.
We put
$$
c_{\beta} = \text{\rm deg} [{\rm ev}_0 : \mathcal M_{1}^{\text{\rm main}}(\beta)^{\frak s} \to L(\text{\bf u})].
$$
Here the right hand side is the mapping degree of the map ${\rm ev}_0$.
It is well-defined since in case $\mu(\beta) =2 $  
the boundary of $\mathcal M_{1}^{\text{\rm main}}(\beta)^{\frak s}$ is empty.
(This is because $\mathcal M_{1}^{\text{\rm main}}(\beta')^{\frak s}$ is empty 
if $\mu(\beta') \le 0$, $\beta'\ne 0$.)
\par
Then we can show that
$$
\sum_{k=0}^{\infty}\frak m_{k,\beta}(b,\dots,b) 
= c_{\beta}T^{\sum_{i=1}^f (\alpha_i\cap \omega)/2\pi} z_{j_1}\dots z_{j_e}.
$$
Item 2) follows from this formula.
\par
Item 3) follows from the fact that in the Fano case,
$\mathcal M_{1}^{\text{\rm main}}(\beta) \ne \emptyset$ 
and $\mu(\beta) = 2$ imply $\beta = \beta_j$ for some $j$.
\par
Item 4) follows from the fact that $c_{\beta}$ is independent of $\text{\bf u}$.
\end{proof}
\begin{Remark}
In the general situation, the filtered $A_{\infty}$ structure 
associated to a Lagrangian submanifold is well-defined only 
up to isomorphism. In particular potential function 
$\frak{PO}$ is well-defined only up to a coordinate change.
(Namely it may depend on the choice of perturbation etc.)
However in our toric case we can use a  $T^n$ equivariant 
perturbation $\frak s$ and then $\frak{PO}$ is well-defined as a function 
on $H^1(L(\text{\bf u});\Lambda_0)$ without ambiguity.
This is a consequence of well-definedness of $c_{\beta}$
and is \cite{toric1} Lemma 11.7.
\end{Remark}
We have the following useful criterion which reduces computation of Floer cohomology
to the critical point theory of potential function.
\begin{Theorem}\label{CritisHFne0}
Let $b = \sum x_i\text{\bf e}_i \in H^1(L(\text{\bf u});\Lambda_0)$.
Then the following three conditions are equivalent.
\begin{enumerate}
\item For each of $i=1,\dots,n$ we have:
$$
\left.\frac{\partial \frak{PO}}{\partial x_i}\right\vert_{b} = 0.
$$
\item
$$
HF((L(\text{\bf u}),b),(L(\text{\bf u}),b);\Lambda_0) \cong
H(T^n;\Lambda_0).
$$
\item
$$
HF((L(\text{\bf u}),b),(L(\text{\bf u}),b);\Lambda) \ne 0.
$$
\end{enumerate}
\end{Theorem}
\begin{proof}[Sketch of the proof]
By definition
\begin{equation}\label{defPO}
\frak{PO}(b) \text{\bf e}
= \sum_{k=0}^{\infty}\frak m_k(b,\dots,b).
\end{equation}
We differentiate (\ref{defPO}) by $x_i$. 
Then using $\partial b/\partial x_i = \text{\bf e}_i$ we obtain:
\begin{equation}\label{form21}
\left.\frac{\partial \frak{PO}}{\partial x_i}\right\vert_{b} \text{\bf e}
= \sum_{k_1=0}^{\infty}\sum_{k_2=0}^{\infty}\frak m_{k_1+k_2+1}(\underbrace{b,\dots,b}_{k_1}, \text{\bf e}_i, \underbrace{b,\dots,b}_{k_2})
= \frak m_1^{b}(\text{\bf e}_i).
\end{equation}
Here the second equality is the definition of $\frak m_1^{b}$.
\par
Now we assume item 2). Then we have $\frak m_1^{b}(\text{\bf e}_i)=0$.
Therefore (\ref{form21}) implies item 1).
\par
We next assume item 1). Then  (\ref{form21}) implies 
$\frak m_1^{b}(\text{\bf e}_i)=0$.
We use it together with the fact that $\text{\bf e}_i$ generates $H(L(\text{\bf u})
;\Lambda_0)$ by cup product, and $A_{\infty}$ formula to prove 
that $\frak m_1^{b} = 0$.
(See \cite{toric1} proof of Lemma 13.1.)
Item 2) follows.
\par
The equivalence between item 2) and item 3) is proved in 
\cite{toric1} Remark 13.9.
\end{proof}
To apply Theorems \ref{calcPO} and \ref{CritisHFne0} for the 
calculation of Floer cohomology of $T^n$, we need some algebraic 
discussion, which is in order.
\par
Let $y_1,\dots,y_n$ be $n$ formal variables. We consider 
the ring $\Lambda[y_1,\dots,y_n,y_1^{-1},\dots,y_n^{-1}]$ 
of Laurent polynomials of $n$ variables with $\Lambda$ coefficient.
We write it as $\Lambda[y,y^{-1}]$ for simplicity.
\par
Let $\text{\bf u} = (u_1,\dots,u_n) \in P$. We put
\begin{equation}
y_i^{\text{\bf u}} = T^{-u_i} y_i \in \Lambda[y,y^{-1}].
\end{equation}
By an easy computation we have
\begin{equation}
T^{\ell_j(\text{\bf u})}(y_1^{\text{\bf u}})^{v_{j,1}}\dots (y_n^{\text{\bf u}})^{v_{j,n}}
= 
T^{\ell_j(\text{\bf u}')}(y_1^{\text{\bf u}'})^{v_{j,1}}\dots (y_n^{\text{\bf u}'})^{v_{j,n}}.
\end{equation}
for $\text{\bf u}, \text{\bf u}' \in P$.
Therefore (\ref{defzj}) defines an elements 
$z_j \in \Lambda[y,y^{-1}]$ in a way independent of 
$\text{\bf u} \in P$.
\par
We next introduce a family of valuations $\frak v_T^{\text{\bf u}}$ on $\Lambda[y,y^{-1}]$ 
parametrized by $\text{\bf u} \in P$.
\par
Let $F \in \Lambda[y,y^{-1}]$. 
Then for each $\text{\bf u} \in 
\text{\rm Int} \,\,P$ there exists $F^{\text{\bf u}}_{i_1 \dots i_n} \in \Lambda$ for $i_1,\dots,i_n \in \Z^n$ 
such that
$$
F = \sum_{i_1,\dots,i_n \in \Z^n} F^{\text{\bf u}}_{i_1 \dots i_n} (y_1^{\text{\bf u}})^{i_1}\cdots 
(y_n^{\text{\bf u}})^{i_n}.
$$
Here only finitely many of $F^{\text{\bf u}}_{i_1 \dots i_n}$ are nonzero. So the 
right hand side is actually a finite sum.
\begin{Definition}\label{normvTu}
$$
\frak v_T^{\text{\bf u}}(F) 
= \inf \{ \frak v_T(F^{\text{\bf u}}_{i_1 \dots i_n}) \mid F^{\text{\bf u}}_{i_1 \dots i_n} \ne 0\},
$$
if $F \ne 0$ and $\frak v_T^{\text{\bf u}}(0) = +\infty$.
\par
$\frak v_T^{\text{\bf u}}$ defines a valuation on $\Lambda[y,y^{-1}]$.
\par
We denote the completion of $\Lambda[y,y^{-1}]$ with respect to $\frak v_T^{\text{\bf u}}$
by $\Lambda^{\text{\bf u}}\langle\!\langle y,y^{-1}\rangle\!\rangle$.
\end{Definition}
By definition we have
$$
\frak v_T^{\text{\bf u}}(z_j) = \ell_j(\text{\bf u}) \ge 0
$$
for $\text{\bf u} \in P$. The following lemma is its immediate consequence.
\begin{Lemma}
The right hand side of $(\ref{POmainformula})$ converges 
with respect to $\frak v_T^{\text{\bf u}}$ for any $\text{\bf u} \in P$.
\end{Lemma}
We remark that according to the general theory described in section 
\ref{Floertheory}, the potential function $\frak{PO}$ associated to a 
Lagrangian submanifold $L(\text{\bf u})$ is a $\Lambda_+$ 
valued function on $\widehat{\mathcal M}_{\rm{weak}}(L(\text{\bf u});\Lambda_0)$.
By Theorem \ref{calcPO} (1), we have 
the inclusion $H^1(L(\text{\bf u});\Lambda_0)
\subset \widehat{\mathcal M}_{\rm{weak}}(L(\text{\bf u});\Lambda_0)$.
Since $x_1,\dots,x_n \in \Lambda_0$ forms a coordinate of 
$H^1(L(\text{\bf u});\Lambda_0)$ with respect to the basis 
$\text{\bf e}_i$, we may regard $\frak{PO}$ restricted to 
$H^1(L(\text{\bf u});\Lambda_0)$
as a function on  $(x_1,\dots,x_n) \in \Lambda_0^n 
\cong H^1(L(\text{\bf u});\Lambda_0)$.
\par
Then by Theorem \ref{calcPO} 2) we have
$$
\frak{PO}(x_1,\dots,x_n) 
= \frak{PO}(x'_1,\dots,x'_n) 
$$
if $x_i - x'_i \in 2\pi\sqrt{-1}\Z$ for each $i$.
In other words, we may regard $\frak{PO}$ as a function of 
$y_i^{\text{\bf u}} = e^{x_i}$. 
Note $x_i \in \Lambda_0$ implies that 
$y_i^{\text{\bf u}} - 1 \in \Lambda_+$.
We next extend the domain of $\frak{PO}$ by using Theorem \ref{calcPO} 2).
\par
We put $\lambda_j = \ell_j(\text{\bf 0})$. 
Then it is easy to see from definition that
\begin{equation}\label{zjandyi}
z_j = T^{\lambda_j}
y_1^{v_{j,1}}\dots y_n^{v_{j,n}}.
\end{equation}
\begin{Lemma}\label{convPO}
Let $(\frak y_1,\dots,\frak y_n) \in (\Lambda\setminus\{0\})^n$.
We assume
\begin{equation}\label{inP}
(\frak v_T(\frak y_1),\dots,\frak v_T(\frak y_n)) \in P.
\end{equation}
We put 
$
\frak z_j = T^{\lambda_j}
\frak y_1^{v_{j,1}}\dots \frak y_n^{v_{j,n}}$.
Then
$$
\frak  z_1+\dots+\frak z_m + \sum_{k=1}^{N} T^{\rho_k}P_k(\frak z_1,\dots,\frak z_m) \in \Lambda_+
$$
converges as $N\to \infty$ with respect to the valuation $\frak v_T$.
\end{Lemma}
\begin{proof}
(\ref{inP}) implies $\frak v_T(\frak z_j) = \ell_j(\text{\bf u}) \ge 0$.
The lemma then follows easily from $\lim_{k\to\infty}\rho_k =\infty$ 
in the statement of Theorem  \ref{calcPO} (\ref{POmainformula}).
\end{proof}
We define
\begin{equation}\label{A(P)}
\frak A(P) = \{(\frak y_1,\dots,\frak y_n) \in  (\Lambda\setminus\{0\})^n
\mid  (\frak v_T(\frak y_1),\dots,\frak v_T(\frak y_n)) \in P\}.
\end{equation}
By Lemma \ref{convPO} we may regard $\frak{PO}$ as a function
$$
\frak{PO} : \frak A(P) \to \Lambda_+.
$$
We remark that $\frak A(P)$ is not a manifold. So we can not define 
differentiation of 
$\frak{PO}$ in the sense of usual calculus. 
Instead we will define it as follows.
We remark that $z_j$ and $P_k(z_1,\dots,z_m)$
are Laurent monomials of $y_1,\dots,y_n$ with $\Lambda_0$ 
coefficient. So we can differentiate it by $y_i$ in an obvious way.
Moreover 
$$
y_i \frac{\partial}{\partial y_i}  P_k(z_1,\dots,z_m)
$$
is again a monomial of $z_1,\dots,z_m$ with $\Lambda_0$ 
coefficient. Therefore for 
$\frak y = (\frak y_1,\dots,\frak y_n) \in \frak A(P)$ the 
limit
$$
\lim_{N\to\infty} \left(
\frak y_i\frac{\partial z_1}{\partial y_i}(\frak y)+\dots+
\frak y_i\frac{\partial z_m}{\partial y_i}(\frak y) + \sum_{k=1}^{N} T^{\rho_k}
\frak y_i \frac{\partial P_k}{\partial y_i}  (\frak z_1,\dots,\frak z_m)\right)
$$
converges. (Here we put $
\frak z_j = T^{\lambda_j}
\frak y_1^{v_{j,1}}\dots \frak y_n^{v_{j,n}}$.)
We write its limit as 
$$
\frak y_i \frac{\partial \frak{PO}}{\partial y_i}(\frak y).
$$
Thus we have defined 
$$
y_i \frac{\partial \frak{PO}}{\partial y_i} : \frak A(P) \to \Lambda_+.
$$
We now have the following:
\begin{Theorem}\label{POandHF}
For $\text{\bf u} \in \text{\rm Int}\,P$ the following two conditions are 
equivalent.
\par\smallskip
\begin{enumerate}
\item
There exists $b \in H^1(L(\text{\bf u});\Lambda_0)$ such that
$$
HF((L(\text{\bf u}),b),(L(\text{\bf u}),b);\Lambda_0) \cong
H(T^n;\Lambda_0).
$$
\item
There exists $\frak y =(\frak y_1,\dots,\frak y_n) \in \frak A(P)$
such that
\begin{equation}\label{eqcritial}
\frak y_i \frac{\partial \frak{PO}}{\partial y_i}(\frak y) = 0
\end{equation}
for $i=1,\dots,n$ and that
$$
(\frak v_T(\frak y_1),\dots,\frak v_T(\frak y_n)) 
= \text{\bf u}.
$$
\end{enumerate}
\end{Theorem}
\begin{Definition}
We say that $L(\text{\bf u})$ is a {\it strongly balanced} if the Condition 1) (= Condition 2)) in Theorem \ref{POandHF} is satisfied.
\end{Definition}
\begin{proof}
2) $\Longrightarrow 1)$:
Let $\frak y$ be as in 2). We put
$y_i^{\text{\bf u}} = T^{-u_i}\frak y_i$.
Then $\frak v_T(y_i^{\text{\bf u}}) = 0$.
Therefore there exist $y_{i,0}^{\text{\bf u}} \in \C$
and $y_{i,+}^{\text{\bf u}} \in \Lambda_+$ such that
$y_i^{\text{\bf u}} =y_{i,0}^{\text{\bf u}} + y_{i,+}^{\text{\bf u}}$.
We put $x_{i,0} = \log(y_{i,0}^{\text{\bf u}})$ and 
$$
x_{i,+} = \log (1+(y_{i,0}^{\text{\bf u}})^{-1}y_{i,+}^{\text{\bf u}})).
$$
Note $(y_{i,0}^{\text{\bf u}})^{-1}y_{i,+}^{\text{\bf u}}\in \Lambda_+$.
Therefore we can define the right hand side by the Taylor expansion 
of $\log(1+z)$.
\par
We put $x_i = x_{i,0} + x_{i,+}$ and $b = \sum_{i=1}^m x_i{\text{\bf e}}_1$.
Then using Theorem \ref{CritisHFne0} it is easy to see that 1) is 
satisfied.
\par
1) $\Longrightarrow 2)$:
Let $b = \sum x_i \text{\bf e}_i$ be as in 1). We put
$\frak y_i = T^{u_i}e^{x_i}$. It is easy to see that 
$\frak y = (\frak y_1,\dots,\frak y_n)$ satisfies
$
\frak y_i \frac{\partial \frak{PO}}{\partial y_i}(\frak y) = 0
$.
\end{proof}
\begin{Remark}
It is easy to see that $y^{\text{\bf 0}}_i = y_i$, where $\text{\bf 0} \in \R^n$ is the 
origin.
Note that the moment polytope $P$ is well-defined only up to parallel translation.
Namely we can replace it by $P + \text{\bf u}$ for any $\text{\bf u}\in \R^n$,  
then
$P+\text{\bf u}$ corresponds to the same toric manifold as $P$.
\par
Thus the choice $y^{\text{\bf 0}}_i = y_i$ is quite ad-hoc, and we may take 
any $y^{\text{\bf u}}_i$ in place of $y_i$ in our story.
In fact the ring $\Lambda[y,y^{-1}]$ can be canonically identified 
with the Laurent polynomial rings over $y^{\text{\bf u}}_i$ ($i=1,\dots,n$) 
using $y^{\text{\bf u}}_i \in \Lambda[y,y^{-1}]$.
\par
On the other hand, the valuation $\frak v_T^{\text{\bf u}}$ and the completion
$\Lambda^{\text{\bf u}}\langle\!\langle y,y^{-1}\rangle\!\rangle$ is 
canonically associated to the Lagrangian submanifold $L(\text{\bf u})$.
\par
The variables $y^{\text{\bf u}}_i$ also is defined in a way independent of the 
choice of the origin of the affine space in which $P$ is embedded.
\par
In some reference such as \cite{Aur07,hori-vafa} `renormalization'  
is discussed. 
It seems that this process depends on the choice of the origin 
in the affine space $\R^n$. Namely it is related to the homothetic 
transformation $y_i \mapsto Cy_i$ where $C \to \infty$.
\par
As we mentioned above the choice of $\text{\bf 0}$ is not 
intrinsic. More canonical way seems to be as follows.
We consider each of $\text{\bf u}_0$ such that 
$HF((L(\text{\bf u}_0),b),(L(\text{\bf u}_0),b);\Lambda) \ne 0$ for some $b$.
We then replace $P$ by $P-\text{\bf u}_0$, so 
this orbit  $L(\text{\bf u}_0)$ becomes $L(\text{\bf 0})$.
We now use $y_i \mapsto Cy_i$ to `renormalize'.
\par
Thus there exists a `renormalization' for each such $\text{\bf u}_0$.
This process of `renormalization' seems to be related to the 
study of leading term equation, which we discuss in section \ref{bulkFloer}.
\end{Remark}
\par\bigskip

\section{Examples 1}
\label{exa1}

\begin{Example}\label{Pn2}
We first consider the case of $\C P^n$.
We use (\ref{ellforPn}) and Theorem \ref{calcPO} 2), 3) to obtain
$$
\frak{PO} =z_1+\dots+z_n +  z_0
= y_1 + \dots +y_n + T(y_1\cdots y_n)^{-1}.
$$
Therefore the equation (\ref{eqcritial}) becomes
$$
0 = y_i \frac{\partial \frak{PO}}{\partial y_i} = y_i - T (y_1\cdots y_n)^{-1}.
$$
The solutions are 
$$
y_1 = \dots = y_n = T^{1/(n+1)} \exp(2\pi\sqrt{-1}k/(n+1))
$$
where $k = 0,1,\dots,n$.
The valuation of $y_i$ are $1/(n+1)$.
Thus $\text{\bf u}_0 = (1/(n+1),\dots,1/(n+1))$ is the unique 
strongly balanced fiber.
\end{Example}
\begin{Example}\label{P2blowupone2}
We next consider $X(\alpha)$, one point blow up of $\C P^2$ 
as in Example \ref{P2blowupone1}. Using the discussion in 
Example \ref{P2blowupone1} and Theorem \ref{calcPO} 2), 3) 
we obtain
$$
\frak{PO} 
= y_1 + y_2 + T(y_1y_2)^{-1} + T^{1-\alpha} y^{-1}_2.
$$
The equation (\ref{eqcritial}) becomes
$$
1 - T y_1^{-2}y_2^{-1} = 0, 
\qquad
1 - T y_1^{-1}y_2^{-2} -  T^{1-\alpha} y^{-2}_2= 0.
$$
By eliminating $y_2 = Ty_1^{-2}$ we obtain
\begin{equation}\label{3rdorder}
y_1^4 + T^{\alpha}y_1^3 - T^{\alpha + 1} = 0.
\end{equation}
We put $u_1 = \frak v_T(y_1)$.
\par\smallskip
\noindent 
(Case 1) $u_1 < \alpha$.
\par
We take $\frak v_T$ of (\ref{3rdorder}) and obtain $4u_1 = \alpha + 1$.
Namely $u_1 = (\alpha+1)/4$. 
$u_1 < \alpha$ then implies 
$\alpha > 1/3$. 
\par
Conversely if $\alpha > 1/3$ and $u_1 = (\alpha+1)/4$ we put 
$y_1 = T^{u_1} y$ then (\ref{3rdorder}) becomes
$$
y^4 + T^{(3\alpha - 1)/4}y^3 -1 = 0.
$$
Since $(3\alpha - 1)/4>0$, this equation has 4 simple roots $y$ which are 
congruent to $\pm 1, \pm \sqrt{-1}$ modulo $\Lambda_+$, respectively.
\par\smallskip\noindent
(Case 2) $u_1 > \alpha$.
\par
We take $\frak v_T$ of (\ref{3rdorder}) and have $3u_1 + \alpha = \alpha + 1$.
Namely $u_1 = 1/3$. $u_1 < \alpha$ then implies 
$\alpha < 1/3$. 
\par
Conversely if $\alpha < 1/3$ and $u_1 = 1/3$ we put 
$y_1 = T^{1/3} y$ then (\ref{3rdorder}) becomes
$$
T^{1/3-\alpha}y^4 +y^3 -1 = 0.
$$
This equation has 3 simple roots $y$ which are 
congruent to $1, e^{2\pi\sqrt{-1}/3}, e^{4\pi\sqrt{-1}/3}$ modulo $\Lambda_+$, respectively.
\par\smallskip\noindent
(Case 3) $u_1 = \alpha$.
\par
We put 
$y_1 = T^{u_1} y$. Then $\frak v_T(y)  = 0$ and we have
\begin{equation}\label{3ody}
y^3(1 + y) - T^{1-3\alpha} = 0.
\end{equation}
\par\smallskip
\noindent
(Case 3-1) $\alpha = 1/3$.
\par
In this case there exists exactly 4 roots $y \in \C$ of (\ref{3ody}).
\par\smallskip
\noindent
(Case 3-2) $\alpha \ne 1/3$.
\par
By (\ref{3ody}) $\alpha < 1/3$. Then $\frak v_T(1+y) = 1 - 3\alpha$.
We put $y = T^{1 - 3\alpha}w - 1$. Then $\frak v_T(w) = 0$.
Then (\ref{3ody}) becomes
$$
(1-T^{1 - 3\alpha}w)^3w +1 = 0.
$$
There is one root of this equation with $w \equiv -1$ modulo $\Lambda_+$.
Three other roots do not satisfy $\frak v_T(w) = 0$.
Thus there exists one solution in this case such that $u_1 = \frak v_T(y) = 0$. 
\par\smallskip
In sum we have the following.
\par
If $\alpha < 1/3$ there exists one solution with 
$u_1 = \frak v_T(y_1) = \alpha$ and three solutions with $u_1 = 1/3$.
Note $u_2 = \frak v_T(y_2) = 1 - 2u_1$. Therefore 
$L(\alpha, 1-2\alpha)$  and $L(1/3,1/3)$ are the strongly balanced fibers.
\par
If $\alpha \ge 1/3$ we have 4 solutions with 
$u_1 = (\alpha+1)/4$, $u_2 = (1-\alpha)/2$.
Namely there is exactly one strong balanced fiber 
$L((\alpha+1)/4,(1-\alpha)/2)$.
\end{Example}
In this section we discuss the Fano case only, where we can 
explicitly calculate $\frak{PO}$. The non-Fano case 
will be discussed in section \ref{exa2}.
\par\smallskip
In the case of Example \ref{Pn2} and \ref{P2blowupone2},
McDuff \cite{mcd}  proved that all the 
$T^n$ orbits where Floer cohomology vanish for all choices of $b$, 
are displaceable by Hamiltonian diffeomorphism.
\par
However there is an example of toric surface and its $T^2$ orbit, such that 
one can not displace it from itself by the method of \cite{mcd} 
but all the known versions of Floer cohomology over $\Lambda$ vanish for this $T^2$ orbit.
(See \cite{mcd} Lemma 4.4.) We do not know whether  they are displaceable or not.
\par\bigskip
\section{Open-closed Gromov-Witten theory and operators $\frak q$}
\label{operatorq}

In this section, we discuss the operator $\frak q$ introduced in \cite{fooo-book}
section 3.8.
Let $(X,\omega)$ be a symplectic manifold and $L$ its 
Lagrangian submanifold as in section \ref{Floertheory}.
Let $h_1,\dots,h_k$ be differential forms on $L$ and 
$g_1,\dots,g_{\ell}$ differential forms on $X$.
Let $\beta \in H_2(X,L;\Z)$.
We define
\begin{equation}\label{operatorqdef}
\aligned
&\frak q_{\ell,k,\beta}(g_1,\dots,g_{\ell};h_1,\dots,h_k)
\\
&= \frac{1}{\ell !}{\rm ev}_{0!}\left(({\rm ev}^+_1,\dots,{\rm ev}^+_{\ell},{\rm ev}_1,\dots,{\rm ev}_k\right)^*
(g_1 \times \dots \times g_{\ell} \times h_1 \times \dots \times h_k).
\endaligned
\end{equation}
We also put 
$$
\frak q_{0;1;0}(h) = (-1)^n dh.
$$
We remark that 
$g_1 \times \dots \times g_{\ell} \times h_1 \times \dots \times h_k$ 
is a differential form on $X^{\ell} \times L^{k}$ and 
its pull back is a differential form on 
$\mathcal M_{k+1;\ell}^{\text{\rm main}}(\beta)$.
The map ${\rm ev}_{0!}$ is integration along fiber by the map
${\rm ev}_0 : \mathcal M_{k+1;\ell}^{\text{\rm main}}(\beta) \to L$.
More precisely we use a continuous family of perturbations in the same way as 
we defined $\frak m_k$ in section \ref{Floertheory}.
\par
We then put
$$
\frak q_{\ell,k}
= \sum_{\beta \in H_2(X,L;\Z)}
T^{(\beta\cap \omega)/2\pi}\frak q_{\ell,k,\beta}.
$$
It defines a map
$$
\frak q_{\ell;k} :
E_{\ell}(\Omega(X)[2]\otimes \Lambda_0) \otimes B_{k}(\Omega(L)[1]\otimes \Lambda_0)  \to \Omega(L)[1]\otimes \Lambda_0 .
$$
This operator has the following properties.
We omit the suffix $\ell,k$ in $\frak q_{\ell;k}$ 
and write $\frak q$ in the formula below.
We use the convention (\ref{Deltasymbol}) introduced at the end of subsection \ref{notterm}.
\begin{Theorem}\label{qqainfalgdef}
\begin{enumerate}
\item
Let $\text{\bf x} \in B_k(\Omega(L)[1]\otimes \Lambda_0)$,
$\text{\bf y} \in E_{\ell}(\Omega(X)[2]\otimes \Lambda_0)$. 
Suppose $\text{\bf y}$ is a linear combination of the 
elements of the form
$y_1\otimes \dots \otimes y_{\ell}$  
where each of $y_i$ are closed forms. We then have the following:
\begin{equation}\label{qmaineq0}
0 =
\sum_{c_1,c_2}
(-1)^*
\frak q(\text{\bf y}^{2;1}_{c_1} ;
\text{\bf x}^{3;1}_{c_2} \otimes
\frak q(\text{\bf y}^{2;2}_{c_1} ; \text{\bf x}^{3;2}_{c_2})
\otimes \text{\bf x}^{3;3}_{c_2})
\end{equation}
where
$
* = \deg'\text{\bf x}^{3;1}_{c_2} +
\deg'\text{\bf x}^{3;1}_{c_2} \deg \text{\bf y}^{2;2}_{c_1}
+\deg \text{\bf y}^{2;1}_{c_1}.
$
\item
If $\text{\bf y} = 1 \in E_0(\Omega(X)[2]\otimes \Lambda_0) 
= \Lambda_0$ then 
\begin{equation}\label{qreducedtom}
\frak q_{0,k}(1,\text{\bf x}) = \frak m_k(\text{\bf x}).
\end{equation}
\item Let $\text{\bf e} = \text{\rm PD}([L])$ be the Poincar\'e dual to the fundamental
class of $L$. Let $\text{\bf x}_i \in B(\Omega(L)[1]\otimes \Lambda_0)$ and we put
$\text{\bf x} = \text{\bf x}_1 \otimes \text{\bf e} \otimes \text{\bf x}_2
\in B(\Omega(L)[1]\otimes \Lambda_0)$. Then
\begin{subequations}\label{qproptall}
\begin{equation}\label{unital0}
\frak q(\text{\bf y};\text{\bf x}) = 0
\end{equation}
except the following case:
\begin{equation}\label{unital20}
\frak q(1;\text{\bf e} \otimes x) =
(-1)^{\deg x}\frak q(1;x \otimes \text{\bf e}) = x,
\end{equation}
\end{subequations} 
where $x \in \Omega(L)[1]\otimes \Lambda_0
= B_1(\Omega(L)[1]\otimes \Lambda_0)$.
\item
There exists a discrete submonoid
$G = \{\lambda_i \mid i=0,1,2,\dots \}$ 
such that 
$$
\frak q_{\ell,k} 
= \sum_{i=1}^{\infty} T^{\lambda_i} \frak q_{\ell,k,i} 
$$
where 
$\frak q_{\ell,k,i} : E_{\ell}(\Omega(X)[2]) \otimes B_{k}(\Omega(L)[1])  \to \Omega(L)[1]$.
\item
Let $i : L \to X$ be the inclusion and $y \in \Omega(X) \otimes \Lambda_0$.
Then
$$
\frak q_{1,0}(y,1) \equiv i^*(y) \mod \Omega(L) \otimes\Lambda_+.
$$
\end{enumerate}
\end{Theorem}
\begin{Remark}
Formula (\ref{qmaineq0}) above implies that the operator 
$\frak q$ (after modifying the sign appropriately) define a 
homomorphism
$
E\mathcal A[2] \to HH^*(L;\Lambda)
$
to the Hochschild cohomology of de Rham cohomology ring of $L$.
See \cite{fooo-book}  Section 7.4.
\end{Remark}
This is de Rham version of \cite{fooo-book} Theorem 3.8.32.
Namely item 1) is \cite{fooo-book} (3.8.33), 
Item 2) is  \cite{fooo-book} 
Theorem 3.8.32 (3). Item 3) is  \cite{fooo-book} (3.8.34.2). 
Item 4) follows immediately from definition. 
Item 5) follows from  \cite{fooo-book} (3.8.34).
\par
Let $\frak b \in \Omega^{even}(X) \otimes \Lambda_+$
and $b \in \Omega^{odd}(L) \otimes \Lambda_+$.
Suppose $d\frak b = 0$.
We put $\text{\bf b} = (\frak b,b)$ and
define 
$$
\frak m_k^\text{\bf b} : B_k(\Omega(L)[1]\otimes \Lambda_0)
\to \Omega(L)[1]\otimes \Lambda_0
$$ 
by
\begin{equation}\label{mbulkdeformed}
\aligned
&\frak m_k^\text{\bf b}(x_1,\dots,x_k) \\
&= \sum_{\ell=0}^{\infty}
\sum_{m_0=0}^{\infty}\cdots\sum_{m_k=0}^{\infty}
\frak q_{\ell,k}(
\frak b^{\ell};\underbrace{b,\dots,b}_{m_0},x_1,\underbrace{b,\dots,b}_{m_1},\dots,
x_k,\underbrace{b,\dots,b}_{m_k}).
\endaligned
\end{equation}
It is easy to see that 
$\{\frak m_k^\text{\bf b}\mid k=0,1,2,\dots\}$ defines 
a unital and gapped filtered $A_{\infty}$ structure.
\par
We define $\widehat{\mathcal M}_{\text{\rm def,weak}}(L)$ as the set of 
all $\text{\bf b} = (\frak b,b)$ such that
\begin{equation}
\frak m_0^\text{\bf b}(1) = c \text{\bf e}.
\end{equation}
Here $\text{\bf e} = 1 \in \Omega^0(L)$.
\par
If $\text{\bf b}\in \widehat{\mathcal M}_{\text{\rm def,weak}}(L)$ then 
we have
$$
\frak m_1^\text{\bf b} \circ \frak m_1^\text{\bf b} =0.
$$
\begin{Definition}
For $\text{\bf b}\in \widehat{\mathcal M}_{\text{\rm def,weak}}(L)$,
we define {\it Floer cohomology with bulk deformation}
by
\begin{equation}\label{HFbulk}
HF((L,\text{\bf b}),(L,\text{\bf b});\Lambda_0)
\cong
\frac{\text{\rm Ker}\,\, \frak m_1^\text{\bf b} }
{\text{\rm Im} \,\,\frak m_1^\text{\bf b} }.
\end{equation}
$HF((L,\text{\bf b}),(L,\text{\bf b});\Lambda)$
is defined by taking $\otimes_{\Lambda_0}\Lambda$.
\par
We define the potential function 
$\frak{PO} : \widehat{\mathcal M}_{\text{\rm def,weak}}(L) 
\to \Lambda_+$ by the equation
\begin{equation}\label{PObulk}
\frak{PO} \text{\bf e} = \frak m_0^{\text{\bf b}}(1).
\end{equation}
We also put
$\frak{PO}^{\frak b}(b) = \frak{PO}(\frak b,b)$.
\end{Definition}
If $HF((L,\text{\bf b}),(L,\text{\bf b});\Lambda) \ne 0$ then 
$L$ is non-displaceable. 
This is \cite{toric2} Proposition 3.15 which is proved 
in \cite{toric2} section 8.
\par\bigskip
\section{Floer cohomology with bulk deformation in the toric case}
\label{bulkFloer}

Now we apply the construction explained in the last section to 
the case of toric manifolds.
In this section we use cycles (submanifolds) rather than differential forms 
to represent the (co)homology classes of ambient manifold $X$, by a reason 
we will mention in Remark \ref{remmaslovdash2}.
\par
Let $D_1,\dots,D_m$ be the irreducible components of toric divisors.
Let $J = \{j_1,\dots,j_k\}\subseteq \{1,\dots,m\}$.
If $D_J = D_{j_1} \cap \dots \cap D_{j_k}$ is non-empty, 
it is a (real) codimension $2k$ submanifold of $X$.
We include the case $J=\emptyset$. In that case $D_J = X$.
We denote by $\mathcal A$ the free abelian group 
generated by $D_J$. 
We put cohomology degree on it. 
Namley $\deg D_J = 2k$ if codimension of $D_J$ is $2k$.
We define $\mathcal A(\Lambda_0) = \mathcal A \otimes \Lambda_0$.
\par
There is an obvious homomorphism
\begin{equation}
\mathcal A \to H^*(X;\Z)
\end{equation}
which is surjective but not injective.
We write the generator of $\mathcal A$ as 
$\text{\bf p}_i$, ($i=0,\dots,B$), where $\text{\bf p}_0 = X$ and 
$\text{\bf p}_i = D_i$ for $i=1,\dots,m$ are degree $2$ classes. 
For $I = (i_1,\dots,i_{\ell}) \in \{1,\dots,B\}^{\ell}$
we put
$$
\text{\bf p}_I = \text{\bf p}_{i_1}\otimes \dots \otimes \text{\bf p}_{i_{\ell}}
, \quad
[\text{\bf p}_I] 
= \frac{1}{\ell !}\sum_{\sigma\in \frak S_{\ell}}
\text{\bf p}_{i_{\sigma(1)}}\otimes \dots \otimes \text{\bf p}_{i_{\sigma({\ell})}}
\in E_{{\ell}}\mathcal A[2].
$$
Here $\frak S_{\ell}$ is the symmetric group of order $\ell !$.
\par
Let $\text{\bf u} \in \text{\rm Int}\,P$, $\beta \in H_2(X,L(\text{\bf u});\Z)$
and $I\in \{1,\dots,B\}^{\ell}$.
We define:
\begin{equation}
\mathcal M_{k+1,\ell}^{\text{\rm main}}(\beta,\text{\bf p}_I) 
=
\mathcal M_{k+1,\ell}^{\text{\rm main}}(\beta) 
{}_{({\rm ev}^+_1,\dots,{\rm ev}_{\ell}^+)}\times_{X^{\ell}} \text{\bf p}_I.
\end{equation}
Note ${\rm ev}^+_i$ are evaluation maps at interior marked points.
We then still have evaluation maps at boundary marked points:
\begin{equation}
{\rm ev} = ({\rm ev}_{0},\dots,{\rm ev}_k) :
\mathcal M_{k+1,\ell}^{\text{\rm main}}(\beta,\text{\bf p}_I) 
\to L^{k+1}.
\end{equation}
We use it to define an operator
$$
\frak q_{\ell,k;\beta} : E_{\ell}\mathcal A[2] 
\otimes B_kH(L(\text{\bf u});\C)[1] \to  H(L(\text{\bf u});\C)[1]
$$
as follows. We remark that there is a transitive and free action of 
$T^n$ on 
$L(\text{\bf u})$. We put a $T^n$ invariant metric on $L(\text{\bf u})$.
Harmonic forms with respect to this metric are nothing but the 
$T^n$ invariant differential forms. We identify the cohomology group 
$H(L(\text{\bf u});\C)$ with the set of the 
$T^n$ invariant forms on  $L(\text{\bf u})$ from now on.
\par
Let $h_1,\dots,h_k \in H(L(\text{\bf u});\C)$.
The pull-back
$$
({\rm ev}_1,\dots,{\rm ev}_k)^*(h_1 \times \dots \times h_k)
$$
is a differential form on 
$\mathcal M_{k+1,\ell}^{\text{\rm main}}(\beta,\text{\bf p}_I)$.
We use integration along fiber of the evaluation map 
${\rm ev}_0 : \mathcal M_{k+1,\ell}^{\text{\rm main}}(\beta,\text{\bf p}_I) 
\to L$ and define:
\begin{equation}\label{qbetatoric}
\frak q_{\ell,k;\beta}([\text{\bf p}_I];h_1 \times \dots \times h_k)
=
{\rm ev}_{0!}({\rm ev}_1,\dots,{\rm ev}_k)^*(h_1 \times \dots \times h_k).
\end{equation}
We can perform all the constructions in a $T^n$ equivariant way.
So the right hand side is a $T^n$ equivariant differential form, 
which we identify with an element of cohomology group.
\begin{Remark}\label{virdimge2}
To define integration along the fiber, 
we need the map ${\rm ev}_0 : \mathcal M_{k+1,\ell}^{\text{\rm main}}(\beta,\text{\bf p}_I) \to L$ to be a submersion.
We also need the moduli space to be transversal after taking an 
appropriate perturbation.
\par
We can do so by using multisection in the same way as section 
\ref{calcu} as follows.
We remark that the fiber product moduli space
$\mathcal M_{k+1,\ell}^{\text{\rm main}}(\beta,\text{\bf p}_I)$
has a Kuranishi structure. The group $T^n$ acts on it. 
Moreover the $T^n$ action is free. (This is because ${\rm ev}_0$ is $T^n$ 
equivariant and the $T^n$ action on $L(\text{\bf u})$ is free.)
Thus by the same argument as we explained during the proof of 
Theorem \ref{calcPO}, we can take multisection $\frak s$ which is $T^n$ equivariant 
and transversal to $0$. Then 
${\rm ev}_0 : \mathcal M_{k+1,\ell}^{\text{\rm main}}(\beta,\text{\bf p}_I)^{\frak s}
\to L(\text{\bf u})$ automatically becomes a submersion 
if $\mathcal M_{k+1,\ell}^{\text{\rm main}}(\beta,\text{\bf p}_I)^{\frak s}$
is nonempty.
\par
We can also choose our perturbation so that 
it is invariant under the permutation of the 
interior marked points so descents to $E_{\ell}\mathcal A[2]$.
Therefore the right hand side of (\ref{qbetatoric}) 
depends only on $[\text{\bf p}_I]$ rather than on $\text{\bf p}_I$.
\end{Remark}
We now define 
$$\frak q_{\ell,k}
 : E_{\ell}\mathcal A(\Lambda_0)[2] 
\otimes B_kH(L(\text{\bf u});\Lambda_0)[1] 
\to  H(L(\text{\bf u});\Lambda_0)[1]
$$
by
$$
\frak q_{\ell,k} = \sum_{\beta\in H_2(X;L(\text{\bf u});\Z)}
T^{(\omega\cap\beta)/2\pi}\frak q_{\ell,k;\beta}.
$$
In case we consider elements of 
$E_{\ell}\mathcal A(\Lambda_0)[2]$ which 
contain $\text{\bf p}_0$, the Poincar\'e dual to $[X]$, 
we define  $\frak q_{\ell,k}$ as follows:
\begin{equation}
\frak q_{1,0}(\text{\bf p}_0;1) = \text{\bf e},\quad
\frak q_{1,2}(\text{\bf p}_0;h_1,h_2) = (-1)^{\deg h_1(\deg h_2+1)} h_1 \wedge h_2.
\end{equation}
In all the other cases, $\frak q_{\ell,k}$ is zero if the first factor 
$E_{\ell}\mathcal A(\Lambda_0)[2]$ contains $\text{\bf p}_0$.
\par
Then our $\frak q_{\ell,k}$ satisfies the conclusion of Theorem \ref{qqainfalgdef}.
\par
For $\text{\bf b} = (\frak b,b) \in  \mathcal A(\Lambda_+)
\times H^{odd}(L(\text{\bf u});\Lambda_+)$, 
we define $\frak m_k^\text{\bf b}$ by (\ref{mbulkdeformed}).
It defines a unital gapped filtered $A_{\infty}$ structure 
on $H(L(\text{\bf u});\Lambda_0)$.
\par
Now we define 
$$
\widehat{\mathcal M}_{\text{\rm def,weak}}(L(\text{\bf u});\Lambda_+)
\subset \mathcal A(\Lambda_+) \times H^{odd}(L(\text{\bf u});\Lambda_+)
$$
as the set of all $\text{\bf b} = (\frak b,b) \in  \mathcal A(\Lambda_+)
\times H^{odd}(L(\text{\bf u});\Lambda_+)$
such that 
$\frak m_0^\text{\bf b}(1)\equiv 0 \mod \Lambda_+ \text{\bf e}$.
In other words it is the set of $(\frak b,b)$ such that
\begin{equation}
\sum_{\ell=0}^{\infty}\sum_{k=0}^{\infty}
\frak q_{\ell;k}(\frak b^{\ell};b^{k}) 
\equiv 0 \mod \Lambda_+ \text{\bf e}.
\end{equation}
We define the potential function 
$\frak{PO} : \widehat{\mathcal M}_{\text{\rm def,weak}}(L(\text{\bf u});\Lambda_+)
\to \Lambda_+$ 
by
\begin{equation}
\sum_{\ell=0}^{\infty}\sum_{k=0}^{\infty}
\frak q_{\ell;k}(\frak b^{\ell};b^{k}) 
= \frak{PO}(\frak b;b) \text{\bf e}.
\end{equation}
Using a similar trick as the one used in section \ref{calcu} we can extend the story 
to the cohomology groups with $\Lambda_0$ coefficient.
Namely we obtain a Maurer-Cartan scheme
$$
\widehat{\mathcal M}_{\text{\rm def,weak}}(L(\text{\bf u});\Lambda_0)
\subset \mathcal A(\Lambda_0) \times H^{odd}(L(\text{\bf u});\Lambda_0)
$$
and Floer cohomology parametrized thereover.
We also have a potential function
$$
\frak{PO} : \widehat{\mathcal M}_{\text{\rm def,weak}}(L(\text{\bf u});\Lambda_0)
\to \Lambda_+.
$$ 
\par\smallskip
Most of the stories in section \ref{calcu} can be generalized to the current 
situation.
\begin{Theorem}\label{calcPObulk}
\begin{enumerate}
\item
$\mathcal A(\Lambda_0) \times H^1(L(\text{\bf u});\Lambda_0)$ is contained in 
$\widehat{\mathcal M}_{\text{\rm def,weak}}(H(L;\Lambda_0);\Lambda_0)$.
\item
Let $b = \sum x_i\text{\bf x}_i \in H^1(L(\text{\bf u});\Lambda_0)$ 
and $\frak b \in \mathcal A(\Lambda_+)$.
Then we have
\begin{equation}\label{POmainformulabulk}
\frak{PO}(\frak b,b) 
= z_1+\dots+z_m 
+\sum_{k=1}^{N} T^{\rho_k}P_k(\frak b;z_1,\dots,z_m).
\end{equation}
Here $N \in \Z_{\ge 0}$ or $N=\infty$.   The numbers $\rho_k$ are all positive 
and real. In case $N = \infty$, the sequence
of numbers $\rho_k$ goes to $\infty$ as $k$ goes to $\infty$.
$P_k(\frak b;z_1,\dots,z_m)$ are monomials of $z_1,\dots,z_m$ of degree $\ge 2$ with 
$\Lambda_0$ coefficient. (Here degree means that of monomials of $z_{i}$.)
We remark that $z_j$ is defined from $y_i^{\text{\bf u}}
= e^{x_j}$ by 
$(\ref{defzj})$.
\item
Let $b = \sum x_i\text{\bf x}_i \in H^1(L(\text{\bf u});\Lambda_0)$ 
and $\frak b \in \mathcal A(\Lambda_0)$.
\begin{equation}\label{POmainformulabulk2}
\aligned
\frak{PO}(\frak b,b) 
= \frak c_1z_1+\dots+\frak c_mz_m 
&+ 
P_0(\frak b;z_1,\dots,z_m)\\ 
& + \sum_{k=1}^{N} T^{\rho_k}P_k(\frak b;z_1,\dots,z_m).
\endaligned\end{equation}
$P_0(\frak b;z_1,\dots,z_m)$ is a formal power series of $z_1,
z_2,\dots,z_m$ with $\Lambda_0$ coefficient such that each 
term has degree $\ge 2$.
The numbers $\frak c_j$ are defined as follows. 
Let $\frak b = \sum w_j \text{\bf p}_j$.
We put $w_j \equiv \overline w_j \mod \Lambda_+$ and 
$\overline w_j \in \C$. Then $\frak c_j = e^{\overline w_j} \in \C \setminus \{0\}$.
Other notations are the same as in $(\ref{POmainformulabulk})$.
\item The monomials $P_k$ and the numbers $\rho_k$ are independent of $\text{\bf u}$
and depends only on $X$ and $\frak b$.
\end{enumerate}
\end{Theorem}
Item 1) is \cite{toric2} Proposition 3.1. (In \cite{toric2} Proposition 3.1 it is assumed that $\frak b \in
\mathcal A(\Lambda_+)$.
It holds also for $\frak b \in \mathcal A(\Lambda_0)$. 
See  \cite{toric2} section 11.)
\par
Item 2) is \cite{toric2} Theorem 3.4. 
\par
Item 3) follows from \cite{toric2} sections 8 and 11. (Formulas (9.3), (11.1) etc.)
\par
Item 4) follows from \cite{toric2} Lemma 6.8.
\par
The proof of Theorem \ref{calcPObulk}
is similar to that of Theorem \ref{calcPO}.
We here mention only a few points.
Let
$I = (i_1,\dots,i_{\ell}) \in \{1,\dots,B\}^{\ell}$.
We put
$\text{\bf p}_I = \text{\bf p}_{i_1} \otimes \cdots \otimes \text{\bf p}_{i_{\ell}}$.
We have
\begin{equation}\label{virdimbulkmod}
\dim \mathcal M_{1,\ell}^{\text{\rm main}}(\beta;\text{\bf p}_I)
= n-2+\mu(\beta) - \sum_{i=1}^{\ell} (\deg \text{\bf p}_i - 2).
\end{equation}
Here $\dim$ is the virtual dimension that is the dimension in the 
sense of Kuranishi structure.
As we explained in Remark \ref{virdimge2} 
the perturbed moduli space 
$\mathcal M_{1,\ell}^{\text{\rm main}}(\beta;\text{\bf p}_I)^{\frak s}$
is empty if (\ref{virdimbulkmod}) $< n$.
\begin{Remark}\label{remmaslovdash2}
This is the reason why we use cycles $\text{\bf p}_i$ rather than 
differential forms on $X$ to represent cohomology classes 
of $X$. This point is crucial to prove item 1) in Theorem \ref{calcPObulk}.
\end{Remark}
In the  case (\ref{virdimbulkmod}) $= n$ we define
\begin{equation}\label{defcbeta}
c(\beta;I)
= \deg [{\rm ev}_0 : \mathcal M_{1,\ell}^{\text{\rm main}}(\beta;\text{\bf p}_I)^{\frak s}
\to L(\text{\bf u})] \in \Q.
\end{equation}
Here and hereafter $\mathcal M_{1,\ell}^{\text{\rm main}}(\beta;\text{\bf p}_I)^{\frak s}$ denotes 
the perturbation of the moduli space $\mathcal M_{1,\ell}^{\text{\rm main}}(\beta;\text{\bf p}_I)$.
Namely it is the zero set of the multisection $\frak s$. 
This zero set has a triangulation and each simplex of maximal 
degree comes with a weight $\in \Q$. Thus it has a virtual fundamental cycle.
See \cite{fooo-book} Section A1.
\par
The number (\ref{defcbeta})  is well-defined. Namely it is independent of the perturbation 
$\frak s$ as far as it is $T^n$ equivariant.
It is also independent of $\text{\bf u}$.
(\cite{toric2} Lemma 6.8.)
The potential function is calculated by using $c(\beta;I)$ as follows.
Let $\vec \ell = \ell_1,\dots,\ell_B \in \Z_{\ge 0}^B$.
We put 
$$
I(\vec \ell) = (\underbrace{1,\dots,1}_{\ell_1},\dots,\underbrace{B,\dots,B}_{\ell_B})
\in \{1,\dots,B\}^{\sum_{i=1}^B \ell_i},
$$
and
$$
c(\beta;\vec \ell) = c(\beta;I(\vec \ell)).
$$
Let $\frak b = \sum_{i=0}^B w_i \text{\bf p}_i$.
\par
We define $\partial_i(\beta) \in \Z$ by 
$$
\partial_i(\beta) = \langle \partial \beta, \text{\bf e}_i^*\rangle,
$$
and put
$$
(y^{\text{\bf u}})^{\partial \beta}
= (y_1^{\text{\bf u}})^{\partial_1 \beta}
\cdots (y_n^{\text{\bf u}})^{\partial_n \beta}
= T^{-\langle \partial \beta, \text{\bf u}\rangle}y_1^{\partial_1 \beta}
\cdots y_n^{\partial_n \beta}.
$$
Now we have
\begin{equation}\label{formula49}
\aligned
\frak{PO}(\frak b,b)
=w_0 +\sum_{\beta \in H_2(X,L(\text{\bf u});\Z)}
&\sum_{\ell_1=0}^{\infty}\cdots\sum_{\ell_B=0}^{\infty}
\\
&\frac{T^{(\beta\cap\omega)/2\pi}}{\ell_1!\cdots \ell_B!}
c(\beta;\vec \ell)
w_1^{\ell_1}\cdots w_B^{\ell_B}(y^{\text{\bf u}})^{\partial \beta}.
\endaligned
\end{equation}
For the proof of (\ref{formula49}) see \cite{toric2} section 9.
\par
(\ref{POmainformulabulk}) follows from (\ref{formula49})
and $c(\beta_j;(0,\dots,0)) = 1$.
This follows from \cite{cho-oh}.
(See \cite{toric2} section 7.)
\par\medskip
Theorem \ref{CritisHFne0} is generalized to our situation 
without change. Namely we have the following theorem.
Hereafter we put $\frak{PO}^{\frak b}(b) = \frak{PO}(\frak b,b)$.
\begin{Theorem}\label{CritisHFne0bulk} 
Let $b = \sum x_i\text{\bf e}_i \in H^1(L(\text{\bf u});\Lambda_0)$ 
and $\frak b \in \mathcal A(\Lambda_0)$.
Then the following three conditions are equivalent.
\begin{enumerate}
\item For each of $i=1,\dots,n$ we have:
$$
\left.\frac{\partial \frak{PO}^{\frak b}}{\partial x_i}\right\vert_{b} = 0
$$
\item
$$
HF((L(\text{\bf u}),(\frak b,b)),(L(\text{\bf u}),(\frak b,b));\Lambda_0) \cong
H(T^n;\Lambda_0).
$$
\item
$$
HF((L(\text{\bf u}),(\frak b,b)),(L(\text{\bf u}),(\frak b,b));\Lambda) \ne 0.
$$
\end{enumerate}
\end{Theorem}
The proof is the same as the proof of Theorem \ref{CritisHFne0} 
except some technical points, which we omit and refer \cite{toric2}.
\par
The discussion in section \ref{calcu} on the domain of the function $\frak{PO}$ as 
a function of $y_i$ is also generalized. 
\par
We put:
\begin{equation}\label{A(P0)}
\frak A(\overset{\circ}P) = \{(\frak y_1,\dots,\frak y_n) \in \Lambda^n
\mid  (\frak v_T(\frak y_1),\dots,\frak v_T(\frak y_n)) \in \text{\rm Int}\,P\}.
\end{equation}
We remark that by Theorem \ref{calcPObulk}
$\frak{PO}^{\frak b}$ may be regarded as a function 
of $y_1,\dots,y_n$.
\begin{Lemma}\label{convPObulk}
Let $(\frak y_1,\dots,\frak y_n) \in \frak A(\overset{\circ}P)$.
We put 
$
\frak z_j = T^{\lambda_j}
\frak y_1^{v_{j,1}}\dots \frak y_n^{v_{j,n}}$.
Then
$$
\frak  z_1+\dots+\frak z_m + 
P_0(
\frak b;\frak z_1,\dots,\frak z_m) + \sum_{k=1}^{N} T^{\rho_k}P_k(
\frak b;\frak z_1,\dots,\frak z_m) \in \Lambda_+
$$
converges as $N\to \infty$ with respect to the valutaion $\frak v_T$.
\par
In case $\frak b \in \mathcal A(\Lambda_+)$  where 
the term $P_0(
\frak b;\frak z_1,\dots,\frak z_m)$ is absent, we may relax the 
assumption to $(\frak y_1,\dots,\frak y_n) \in \frak A(P)$.
\end{Lemma}
Thus we may regard $\frak{PO}^{\frak b}$ as a function 
either $: \frak A(\overset{\circ}P) \to \Lambda_+$
or $: \frak A(P) \to \Lambda_0$.
\par
We can define 
$$
\frak y_i\frac{\partial \frak{PO}^{\frak b}}{\partial y_i}
$$
in the same way as section \ref{calcu}. It defines either a function
$: \frak A(\overset{\circ}P) \to \Lambda_+$
or $: \frak A(P) \to \Lambda_0$.
Theorem \ref{POandHF} can be generalized as follows:
\begin{Theorem}\label{POandHFbulk}
For $\text{\bf u} \in \text{\rm Int}\,P$, 
$\frak b \in \mathcal A(\Lambda_0)$, the following two conditions are 
equivalent.
\par\smallskip
\begin{enumerate}
\item
There exists $b \in H^1(L(\text{\bf u});\Lambda_0)$ such that
$$
HF((L(\text{\bf u}),(\frak b,b)),(L(\text{\bf u}),(\frak b,b));\Lambda_0) \cong
H(T^n;\Lambda_0).
$$
\item
There exists $\frak y =(\frak y_1,\dots,\frak y_n) \in \frak A(P)$
such that
\begin{equation}\label{eqcritialbulk}
\frak y_i \frac{\partial \frak{PO}^{\frak b}}{\partial y_i}(\frak y) = 0
\end{equation}
for $i=1,\dots,n$ and that
$$
(\frak v_T(\frak y_1),\dots,\frak v_T(\frak y_n)) 
= \text{\bf u}.
$$
\end{enumerate}
\end{Theorem}
This is \cite{toric2} Theorem 3.12.
\par\bigskip
\section{Leading term equation}
\label{LTE}

Theorem \ref{POandHFbulk} provides a means of determining the Floer
cohomology in terms of the potential function. The main obstacle to
directly apply the theorem in practice is that we do not know how to
calculate the extra terms $P_k(\mathfrak b;z_1,\cdots, z_n)$ unless $X$
is Fano and $\mathfrak b$ has degree 2. (There has been some computation
carried out in this direction for the nef case. See e.g. 
\cite{chan-lau}.)
\par
Fortunately to determine all the $T^n$ 
orbits $L(\text{\bf u})$ for which some Floer cohomology
with bulk does not vanish, we do not need to calculate those terms.
We will explain it in this section.
\par
In this and the next sections we fix $\frak b$ and $\text{\bf u}$ and 
consider $\frak{PO}^{\frak b}$ as a function of variables 
$y_i^{\text{\bf u}}$.
In this section we write $\overline y_i$ instead of $y_i^{\text{\bf u}}$.
We remark that $\frak v_T^{\text{\bf u}}(\overline y_i) = 0$  and
$$
z_j = T^{\ell_j(\text{\bf u})} \overline y_1^{v_{j,1}}\cdots \overline y_1^{v_{j,n}}.
$$
\begin{Definition}\label{PO0def}
We denote the sum of linear terms $z_j$'s in  $\frak{PO}^{\frak b}$ by 
$$
\frak{PO}^{\frak b}_0 = {\frak c}_1z_1+\cdots+{\frak c}_mz_m
= \sum_{j=1}^m T^{\ell_j(\text{\bf u})} {\frak c}_j\overline y_1^{v_{j,1}}\cdots \overline y_n^{v_{j,n}}
$$
and call it the {\it leading order potential function}.
Here $\frak c_j$ is defined as in Theorem \ref{calcPObulk} 3).
\end{Definition}
Note this function appears frequently in the literature 
(see \cite{givental1, hori-vafa, iritani1}), is denoted as $W$, and is called the (Landau-Ginzburg) superpotential.
\begin{Remark}
Note in our situation of toric manifold, superpotential 
in physics literature is basically the same as our 
potential function. However in other situation 
they may be different. 
For example in the case of Calabi-Yau $3$ fold $X$ and 
its special Lagrangian submanifold $L$, 
our potential function is identically $0$. 
(In other words, if $b$ is a weak bounding chain then it is a bounding chain 
automatically.)
On the other hand, 
the physisists' superpotential coincides with the invariant 
introduced in \cite{fu0912}.
\end{Remark}
We remark that the leading order potential function $\frak{PO}^{\frak b}_0$ is explicitly 
read off from the moment polytope $P$ and $\text{\bf u}$.
The leading term equation we will define below depends only 
on leading order potential function and so is also 
explicitly calculable.
\par
We renumber the values $\ell_i(\text{\bf u})$ 
according to its order.
Namely we take
$j(l,r) \in \{1,\dots,m\}$ for $l=1,\dots, K_0$, $r= 1,\dots,a(l')$ 
with the following conditions.
\begin{cond}
\begin{enumerate}
\item $\{ j(l,r) \mid l=1,\dots, K_0, r= 1,\dots,a(l) \} 
= \{1,\dots,m\}$.
\item $a(1) +\dots + a(K_0) = m$.
\item $\ell_{j(l,r)}(\text{\bf u}) = \ell_{j(l,r')}(\text{\bf u})$ for $1 \le r,r' \le a(l)$.
\item $\ell_{j(l,r)}(\text{\bf u}) < \ell_{j(l',r')}(\text{\bf u})$ if $l < l'$.
\end{enumerate}
\end{cond}
We put
\begin{equation}
S_l = \ell_{j(l,r)}(\text{\bf u}).
\end{equation}
This is independent of $r$. Set
\begin{equation}
\vec v_{l,r} = \vec v_{j(l,r)} = (v_{j(l,r),1},\dots,v_{j(l,r),n}) \in \Z^n.
\end{equation}
It is an element of the dual vector space of $\mathbb A(\Q) = \Q^n$, which we denote 
by 
$\mathbb A(\Q)^*$. 
Here $\mathbb  A(\R) = \mathbb A(\Q) \otimes \R$ is the $\R$ vector space associated to the affine 
space which contains the moment polytope $P$.
Let $\mathbb  A_l^{\perp}$ is a vector subspace of $\mathbb  A(\Q)^*$
generated by 
$
\{ \vec v_{l',r} \mid l'\le l, r = 1,\dots, a(l')\}.
$
We denote by $K \le K_0$ the smallest integer such that  
$\mathbb A_K^{\perp} = \mathbb A(\Q)^*$.
We have a filtration
\begin{equation}
0 \subset \mathbb  A_1^{\perp} \subset 
\mathbb  A_2^{\perp} \subset \dots \subset
\mathbb  A_K^{\perp} = \mathbb (\Q)^*.
\end{equation}
We put
\begin{equation}
d(l) = \dim \mathbb  A_l^{\perp} - \dim \mathbb  A_{l-1}^{\perp}.
\end{equation}
We have
\begin{equation}
d(1) + \cdots + d(K) = n = \dim \mathbb  A(\Q)^*.
\end{equation}
Note $\mathbb  A \cong \Z^n \subset \mathbb  A(\Q) = \Q^n$.
So $\Z^n \subset \mathbb  A(\Q)^*$ is determined canonically.
(We remark that $\Z^n \subset \mathbb  A(\Q)^*$ 
is generated by $\vec v_j$, $j=1,\dots,m$.)
Let $\{\text{\bf e}^*_i\mid i=1,\dots,n\}$ be the standard 
basis of  $\Z^n \subset \mathbb  A(\Q)^*$.
We take $\text{\bf e}^*_{l,s}$ for $l=1,\dots,K$, $s=1,\dots,d(l)$ 
satisfying the following conditions.
\begin{cond}\label{condels}
\begin{enumerate}
\item $\{ \text{\bf e}^*_{l',s} \mid l' \le l, \,\, s=1,\dots,d(l')\}$
is a $\Q$ basis of $\mathbb A_l^{\perp}.$
\item $\vec v_{l,r}$ is contained in the $\Z$ module generated by 
$\{ \text{\bf e}^*_{l',s} \mid l' \le l, \,\, s=1,\dots,d(l')\}$.
\end{enumerate}
\end{cond}
We define $b_{l',s;j} \in \Q$ by
$$
\text{\bf e}^*_{l',s} = \sum_{j=1}^m b_{l',s;j} \text{\bf e}^*_j
$$
and put
$$
\overline y_{l',s} = \prod_{j=1}^m \overline y_j^{b_{l',s;j}}
= \exp\left({\sum_{j=1}^m} b_{l',s;j}x_j\right).
$$
(Note $\overline y_{i} = e^{x_i}$.)
Since $b_{l',s;i}$ may not be an integer, 
$\overline y_{l',s}$ may not be contained in the 
Laurent polynomial ring $\Lambda_0[\overline y,\overline y^{-1}]$
of the variables $\overline y_j$ ($j=1,\dots,m$). But it is contained in the finite extention of it.
Let $\Lambda_0[y_{**},y_{**}^{-1}]$ be the 
Laurent polynimial ring of the variables 
$\overline y_{l,s}$, $l=1,\dots,K$, $s=1,\dots,d(l)$.
\par
By Condition \ref{condels} 2), 
$$
z_{j(l,r)} = T^{S_l}\overline y_1^{v_{j(l,r),1}} \cdots\overline y_n^{v_{j(l,r),n}}
$$
is contained in $\Lambda_0[\overline y,\overline y^{-1}]$.
Moreover it is contained in the Laurent polynomial ring 
of the variables $\overline y_{l',s}$, $l'=1,\dots,l$, $s=1,\dots,d(l')$.
\par
We define $c_{l,r;l',s} \in \Z$ by
\begin{equation}
z_{j(l,r)} = T^{S_l} \prod_{l'\le l}\prod_{s\le d(l')} \overline y_{l',s}^{c_{l,r;l',s}}.
\end{equation}
In other words
$$
\vec v_{l,r} = \sum_{l'\le l}\sum_{s\le d(l')} c_{l,r;l',s} e^*_{l',s}.
$$
We put
\begin{equation}
\left(\frak{PO}_0^{\frak b}\right)_l
= \sum_{r=1}^{a(l)} {\frak c}_{j(l,r)}z_{j(l,r)} = 
\sum_{r=1}^{a(l)} {\frak c}_{j(l,r)}  \prod_{l'\le l}\prod_{s\le d(l')} \overline y_{l',s}^{c_{l,r;l',s}}.
\end{equation}
The numbers ${\frak c}_{j(l,r)} \in \{ c \in \Lambda_0 \mid \frak v_T(c) = 0\}$ are 
defined in Definition \ref{PO0def}.
\par
We remark
$\left(\frak{PO}_0^{\frak b}\right)_l$ is a 
Laurent polynomial of variables 
$\overline y_{l',s}$, $l' \le l$, $s = 1,\dots,d(l')$
with coefficient in complex number.
\par
\begin{Definition}
The {\it leading term equation} is a system of $n$ equations 
of $n$ variables $\overline y_{l,s}$ with complex number coefficient. We define it by
\begin{equation}\label{LTEq}
\left\{
\aligned
\overline y_{1,s} \frac{\partial\left(\frak{PO}_0^{\frak b}\right)_1}{\partial \overline y_{1,s}}
&= 0   \qquad s =1,\dots,d(1), \\
\overline y_{2,s} \frac{\partial\left(\frak{PO}_0^{\frak b}\right)_2}{\partial \overline y_{2,s}}
&= 0   \qquad s =1,\dots,d(2), \\
&\cdots\\
\overline y_{l,s} \frac{\partial\left(\frak{PO}_0^{\frak b}\right)_l}{\partial \overline y_{l,s}}
&= 0   \qquad s =1,\dots,d(l), \\
&\cdots\\
\overline y_{K,s} \frac{\partial\left(\frak{PO}_0^{\frak b}\right)_K}{\partial \overline y_{K,s}}
&= 0   \qquad s =1,\dots,d(K).
\endaligned
\right.
\end{equation}
\end{Definition}
Note the first equation in (\ref{LTEq}) contains $\overline y_{1,s}$ $s =1,\dots,d(1)$,
the second equation in (\ref{LTEq}) contains $\overline y_{1,s}$ $s =1,\dots,d(1)$ and 
$\overline y_{2,s}$ $s =1,\dots,d(2)$ etc.
\par
If $\frak b - \frak b' \in \mathcal A^2(\Lambda_+) \oplus
\bigoplus_{k\ne 2} \mathcal A^k(\Lambda_0)$
then $\left(\frak{PO}_0^{\frak b}\right)_l 
= \left(\frak{PO}_0^{\frak b'}\right)_l$.
So the leading term equation is the same for such $\frak b$ and $\frak b'$.
\par
One of the main results of \cite{toric2} is as follows.
\begin{Theorem}\label{eliminatebulk}
Let $\text{\bf u} \in \text{\rm Int}\,\, P$ and $\frak b \in \mathcal A(\Lambda_0)$.
Then the following two conditions are equivalent.
\begin{enumerate}
\item
The leading term equation $(\ref{LTEq})$ has a solution 
$\overline y_{l,s} \in \C \setminus \{0\}$.
\item
There exists $b \in H^1(L(\text{\bf u});\Lambda_0)$ and 
$\frak b' \in \mathcal A(\Lambda_0)$ with 
$\frak b - \frak b' \in \mathcal A^2(\Lambda_+)$ 
such that
$$
HF((L(\text{\bf u}),(\frak b,b)),(L(\text{\bf u}),(\frak b,b));\Lambda_0) \cong
H(T^n;\Lambda_0).
$$
\end{enumerate}
\end{Theorem}
This is \cite{toric2} Theorem 4.7 and Proposition 11.3.
We omit the proof and refer \cite{toric2}.
\begin{Definition}
We say that $L(\text{\bf u})$ is {\it strongly bulk balanced} if 
there exists $\frak b\in \mathcal A(\Lambda_0)$ and $b\in H^1(L(\text{\bf u});
\Lambda_0)$ such that 
$$
HF((L(\text{\bf u}),(\frak b,b)),(L(\text{\bf u}),(\frak b,b));\Lambda_0) \cong
H(T^n;\Lambda_0).
$$
\end{Definition}
See \cite{toric2} Definition 3.13 for a related definition.
\par
Theorem \ref{eliminatebulk} gives a way to locate strongly bulk 
balanced $L(\bf u)$ in terms of the leading term equation.
\par\bigskip
\section{Examples 2}
\label{exa2}

\begin{Example}\label{ex:Hilexa}
We consider Hirzebruch surface $F_n$, $n\ge 2$. We take its
K\"ahler form so that the moment polytope is
$$
P = \left\{ (u_1,u_2) \mid 0 \le u_1,u_2, \,\, u_1+nu_2 \le n, \,\, u_2 \le 1-\alpha\right\},
$$
$0 < \alpha <1$. The leading order potential function is
$$
\mathfrak{PO}_0 = y_1 + y_2 + T^{n} y_1^{-1}y_2^{-n}  + T^{1-\alpha}y_2^{-1}.
$$
We put
$$\aligned
\ell_1(u_1,u_2) = u_1,\qquad
&\ell_2(u_1,u_2) = u_2, \\
\ell_3(u_1,u_2) = n-u_1-nu_2,\qquad
&\ell_4(u_1,u_2) = 1-\alpha - u_2.
\endaligned$$
We put $S_1(u_1,u_2) = \inf\{ \ell_j(u_1,u_2) \mid j=1,2,3,4\}$.
\par
Suppose the first of the leading term equation (\ref{LTEq}) 
has a nonzero solution.
Then it is easy to see that $d(1) \ge 2$.
Namely
$$
\#\{ j \mid S_1(u_1,u_2)= \ell_j(u_1,u_2)\} \ge 2.
$$
This is satisfied on the 5 line seguments $l_1,\dots,l_5$, where
$$\aligned
&l_1 : u_1=u_2 \le (1-\alpha)/2,
\quad
l_2 : u_1=1-\alpha-u_2 \le (1-\alpha)/2, \\
&
l_3 : u_1=n-(n+1)u_2 \ge n-(n+1)(1-\alpha)/2 \\
&
l_4 : u_1=n-1+\alpha-(n-1)u_2 \ge n-(n-1)(1-\alpha)/2, \\
&
l_5 : u_2 =  (1-\alpha)/2, (1-\alpha)/2\le u_1\le n-(n-1)(1-\alpha)/2.
\endaligned$$
Note 
$$
v_1 = (1,0), \, v_2 = (0,1), \, v_3 = (-1,-n), \, v_4 = (0,-1).
$$
Let $\text{\bf u} = (u_1,u_2) \in l_5$. Then
$
\mathbb A_1^{\perp}(1) 
$
is $\Q \cdot (0,1)$ and 
\begin{equation}\label{ltehir1}
(\mathfrak{PO}^\text{\bf u}_0)_1 = \overline y_2 + \overline y^{-1}_2
\end{equation}
(Here $\frak b = 0$ and so we do not write  $\frak b$ in the 
above notation. We put $\overline y_i = y_i^{\text{\bf u}}$)
\par\medskip
\epsfbox{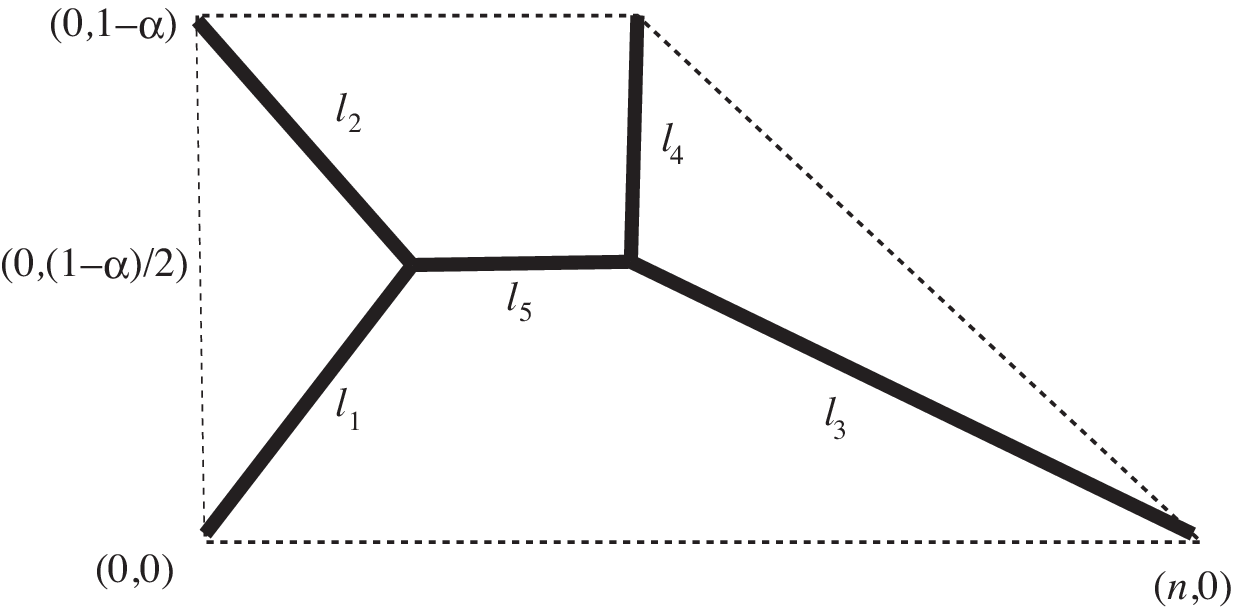}
\par\smallskip
\centerline{\bf Figure 10.1}
\par\bigskip
We also have
\begin{equation}\label{ltehir2}
(\mathfrak{PO}^\text{\bf u}_0)_2 =
\begin{cases}
\overline y_1  &\text{If $u_1 < (1+\alpha)n/4$} \\
\overline y_1^{-1}\overline y_2^{-n}  &\text{If $u_1 > (1+\alpha)n/4$}  \\
\overline y_1+\overline y_1^{-1}\overline y_2^{-n} &\text{If $u_1= (1+\alpha)n/4$} 
\end{cases}
\end{equation}
(\ref{ltehir1}) gives the first leading term equation 
$
1 -\overline  y_2^{-2} = 0
$ whose solutions are $\overline y_2 = \pm 1$.
\par
Then (\ref{ltehir2}) gives the second of the leading 
term equation which are
$1=0$, $-(\pm 1)^{-n}\overline y_1^{-2} = 0$, $1-(\pm 1)^{-n}\overline y_1^{-2}= 0$, 
where $u_1 < (1+\alpha)n/4$, 
$u_1 > (1+\alpha)n/4$ and $u_1= (1+\alpha)n/4$,
respectively.
\par
The solution $\overline y_1 \ne 0$ exists only in the case
$u_1= (1+\alpha)n/4$. In that case the solutions of leading term equations 
are 
$(1,\pm 1)$  and 
$(-1,\pm (-1)^{n/2})$. 
Thus $L((1+\alpha)n/4,(1-\alpha)/2)$ is 
strongly bulk balanced.
\par
We can check that there are no other strongly bulk balanced $T^2$ orbit.
(This follows from Theorem \ref{critHQ} also.)
\par
See \cite{toric1} Example 8.2 where the same conclusion is proved 
by basically the same but a slightly different calculation.
\end{Example}
\begin{Remark}
For the case of Example \ref{ex:Hilexa} we can actually 
prove that $L((1+\alpha)n/4,(1-\alpha)/2)$ is 
strongly balanced. Namely some Floer cohomology without bulk deformation 
is non-zero. This follows from \cite{toric1} Theorem 10.4.
\end{Remark}
\begin{Example}\label{p22ptbu}
(\cite{toric2} section 5, \cite{toric1} Example 10.17.)
We consider two points blow up $X(\alpha,\alpha')$ of $\C P^2$.
(Example \ref{P2blowuptwo1}.)
We consider the case $\alpha > 1/3$, $\alpha' = (1-\alpha)/2$.
The moment polytope is
$$
P = \{(u_1,u_2) \mid 0\le u_1 \le 1,\,\, 0\le u_2 \le 1-\alpha, \,\,
(1-\alpha)/2\le u_1+u_2 \le 1\}.
$$
We consider 
\begin{equation}\label{utP2}
\text{\bf u}(t) = (t,(1-\alpha)/2), \qquad
t \in ((1-\alpha)/2,(1+\alpha)/4).
\end{equation}
We have
$$
\frak{PO} = T^{(1-\alpha)/2} (\overline y_2+\overline y_2^{-1}) + T^t(\overline y_1+\overline y_1\overline y_2) 
+ T^{(1+\alpha)/2 + t}(\overline y_1\overline y_2)^{-1} 
$$
where 
$$
(1-\alpha)/2 < t < (1+\alpha)/2 + t.
$$
Therefore
$$
(\frak{PO})_1 = \overline y_2+\overline y_2^{-1}, 
\qquad
(\frak{PO})_2 = \overline y_1+\overline y_1\overline y_2.
$$
Thus the leading term equation is 
$$
1 - \overline y_2^{-2} = 0, 
\qquad
1+\overline y_2 = 0.
$$
This has a solution $\overline y_2 = -1$ ($\overline y_1$ is any number $\in \C \setminus \{0\}$.)
\par
Theorem \ref{eliminatebulk} implies that 
all of $L(\text{\bf u}(t))$ as in (\ref{utP2}) are strongly bulk balanced.
In particular they are non-displaceable.
\par\medskip
\epsfbox{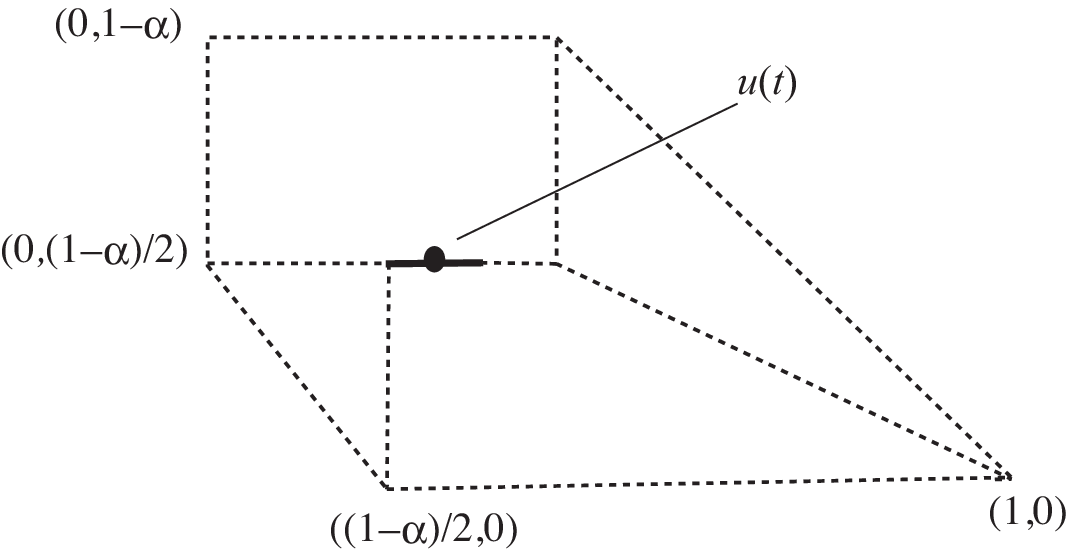}
\par\smallskip
\centerline{\bf Figure 10.2}
\par\bigskip
\end{Example}
\begin{Remark}
In the toric case, for each 
given $\frak b$, the number of $L(\text{\bf u})$ with 
nontrivial Floer cohomology for a pair $(\frak b,b)$  for 
some $b\in H^1(L(\text{\bf u});\Lambda_0)$ is finite.
(It is smaller than the Betti number of $X$ by Theorem \ref{Jacqcmain1}.)
So to obtain infinitely many $L(\text{\bf u})$ with 
nontrivial Floer cohomology we need to include bulk 
deformations.
\end{Remark}
In the examples we discussed in this section, we do not need to 
change the variables from $y_j$ to $y_{l,s}$.
An example where we need this change of variables is 
given in \cite{toric1} Example 10.10.
\par
In Example \ref{p22ptbu} we obtain a continuum of non-displaceable 
Lagrangian torus in certain two points blow up of $\C P^2$.
(\cite{toric2}).
We can also use bulk deformation to obtain a continuum of Lagrangian 
tori in $S^2 \times S^2$. 
They are not of the type of $T^2$ orbit 
but is obtained from the $T^2$ orbit of singular Hirzebruch surface $F_2(0)$ by 
deforming the singularity, that is of orbifold of $A_2$-type.
(\cite{fooo10}.)
Closely related construction is in  \cite{nnu1, nnu2}
\par\bigskip

\section{Quantum cohomology and Jacobian ring}
\label{QCJac}

\subsection{Jacobian ring over Novikov ring}
\label{Jacsub}

In this section we discuss the isomorphism between the Jacobian ring of 
$\frak{PO}^{\frak b}$ and the quantum cohomology ring of $X$ 
deformed by $\frak b$.
We start with defining Jacobian ring precisely.
\par
Usually Jacobian ring is studied in the case of (Laurent) polynomial or 
holomorphic function germ.
Our function $\frak{PO}^{\frak b}$ is neither a Laurent polynomial 
and nor a holomorphic function.
So we first define a function space in which $\frak{PO}^{\frak b}$
is contained.
\par
We consider the Laurent polynomial ring 
$\Lambda[y,y^{-1}]$ of $n$ variables with $\Lambda$ 
coefficients.
We defined a valuation $\frak v_T^{\text{\bf u}}$ for each 
$\text{\bf u} \in \R^n$ in section \ref{calcu} 
Definition \ref{normvTu}.
Let $P$ be a compact subset of $\R^n$. 
(We use the case when $P$ is a convex polytope only in this 
article.)
\begin{Definition}
For $F \in \Lambda[y,y^{-1}]$ we define
$$
\frak v_T^{P}(F)
= \inf \{ \frak v_T^{\text{\bf u}}(F) \mid \text{\bf u}\in P\}.
$$
This is not a valuation but is a norm. Therefore it defines a metric 
on $\Lambda[y,y^{-1}]$ by $d_P(F_1,F_2) = e^{-\frak v_T^{P}(F_1-F_2)}$.
We denote the completion of $\Lambda [y,y^{-1}]
$ with respect to $d_P$ 
by $\Lambda^P\langle\!\langle y,y^{-1}\rangle\!\rangle$.
It is a normed ring.
\par
We define $\Lambda_0^P\langle\!\langle y,y^{-1}\rangle\!\rangle$ as the set of all 
$F\in \Lambda^P\langle\!\langle y,y^{-1}\rangle\!\rangle$ such that $\frak v_T^P(F) \ge 0$.
\end{Definition}
Let $P$ be a moment polytope of our toric manifold $X$.
We take $\ell_j$ ($j=1,\dots,m$) as in Condition \ref{condellj} and put
$$
P_{\epsilon} = \{\text{\bf u} \in \R^n \mid \ell_j(\text{\bf u}) \ge \epsilon,
\,\,\, j=1,\dots,m  \} 
$$
for $\epsilon >0$.
\begin{Definition}
We define a metric $d_{\overset{\circ}P}$ on  $\Lambda[y,y^{-1}]$
by
$$
d_{\overset{\circ}P}(F_1,F_2)
= 
\sum_{k=1}^{\infty} 2^{-k}\min(d_{P_{1/k}}(F_1,F_2),1).
$$
Let $\Lambda^{\overset{\circ}P}\langle\!\langle y,y^{-1}\rangle\!\rangle$ 
be the completion of $\Lambda[y,y^{-1}]$ with respect to 
$d_{\overset{\circ}P}$.
\end{Definition}
It is easy to see that an element of $\Lambda^P\langle\!\langle y,y^{-1}\rangle\!\rangle$ may be regarded as a function 
$: \frak A(P) \to \Lambda$ and 
an element of $\Lambda^{\overset{\circ}P}\langle\!\langle y,y^{-1}\rangle\!\rangle$ may be regarded as a function 
$: \frak A({\overset{\circ}P}) \to \Lambda$.
\begin{Lemma}
If $\frak b\in \mathcal A(\Lambda_0)$ then
\begin{equation}
\frak{PO}^{\frak b} \in \Lambda_0^{\overset{\circ}P}\langle\!\langle y,y^{-1}\rangle\!\rangle,
\qquad
y_i\frac{\partial \frak{PO}^{\frak b}}{\partial y_i} \in \Lambda_0^{\overset{\circ}P}
\langle\!\langle y,y^{-1}\rangle\!\rangle.
\end{equation}If $\frak b\in \mathcal A(\Lambda_+)$ then
\begin{equation}
\frak{PO}^{\frak b} \in \Lambda_0^{P}\langle\!\langle y,y^{-1}\rangle\!\rangle,
\qquad
y_i\frac{\partial \frak{PO}^{\frak b}}{\partial y_i} \in  \Lambda_0^{P}\langle\!\langle y,y^{-1}\rangle\!\rangle.
\end{equation}
\end{Lemma}
We omit the proof, which follows from 
Theorem \ref{calcPObulk}. See \cite{toric3} Lemma 2.6.
Now we define
\begin{Definition}
$$
\text{\rm Jac}(\frak{PO}^{\frak b})
= 
\frac{\Lambda_0^{\overset{\circ}P}\langle\!\langle y,y^{-1}\rangle\!\rangle}
{\text{\rm Clos}_{d_{\overset{\circ}P}}
\left( y_i\frac{\partial \frak{PO}^{\frak b}}{\partial y_i} 
: i=1,\dots,n\right)}.
$$
\end{Definition}
(We may replace  $\Lambda_0^{\overset{\circ}P}\langle\!\langle y,y^{-1}\rangle\!\rangle$ 
by $\Lambda_0^{P}\langle\!\langle y,y^{-1}\rangle\!\rangle$
in the above formula in case $\frak b \in \mathcal A(\Lambda_+)$.)

Here the denominator is the closure of the ideal generated by 
$ y_i\frac{\partial \frak{PO}^{\frak b}}{\partial y_i} 
: i=1,\dots,n$. The closure is taken with respect to the metric 
$d_{\overset{\circ}P}$.
\par\bigskip
\subsection{Big quantum cohomology: a quick review}
\label{bigquatum}

We next review briefly the well established story of 
deformed quantum cup product. 
Let $(X,\omega)$ be a symplectic manifold. For
$\alpha \in H_2(X;\Z)$ let
$\mathcal M_{\ell}(\alpha)$ be the moduli space of
stable maps from genus zero semi-stable curves with $\ell$ marked points and of homology class $\alpha$.
There exists an evaluation map
$$
\text{\rm ev}: \mathcal M_{\ell}(\alpha) \to X^{\ell}.
$$
$ \mathcal M_{\ell}(\alpha)$ has a virtual fundamental cycle
and hence defines a class
$$
\text{\rm ev}_*[\mathcal M_{\ell}(\alpha)] \in H_{*}(X^{\ell};\Q).
$$
(See \cite{FO}.)
Here
$
* = 2n + 2c_1(X)\cap \alpha + 2\ell - 6
$.
Let $Q_1,\dots,Q_{\ell}$ be cycles such that
\begin{equation}\label{degreecondGW}
\sum \text{\rm codim}\, Q_i = 2n + 2c_1(X)\cap \alpha + 2\ell - 6.
\end{equation}
We define Gromov-Witten invariant by
$$
GW_{\ell}(\alpha:Q_1,\dots,Q_{\ell}) =
\text{\rm ev}_*[\mathcal M_{\ell}(\alpha)] \cap (Q_1 \times \dots \times Q_{\ell})
\in \Q.
$$
We put $GW_{\ell}(\alpha:Q_1,\dots,Q_{\ell}) = 0$ when (\ref{degreecondGW})
is not satisfied.
\par
We now define
\begin{equation}\label{sumGW}
GW_{\ell}(Q_1,\dots,Q_{\ell})
= \sum_{\alpha} T^{(\alpha \cap \omega)/2\pi} GW(\alpha:Q_1,\dots,Q_{\ell}).
\end{equation}
The formula 
(\ref{sumGW}) extends to a $\Lambda_0$ module homomorphism
$$
GW_{\ell}: H(X;\Lambda_0)^{\otimes\ell} \to \Lambda_0.
$$
\begin{Definition}\label{def:deformcup}
Let $\frak b \in H(X;\Lambda_0)$ be given. For each given pair $\frak c,\frak d \in H(X;\Lambda_0)$, we define a product
$
\frak c \cup^{\frak b} \frak d \in H(X;\Lambda_0)
$
by the following formula
\begin{equation}\label{defcup}
\langle \frak c \cup^{\frak b} \frak d, \frak e\rangle_{\text{\rm PD}_X}
= \sum_{\ell=0}^{\infty} \frac{1}{\ell!}GW_{\ell+3}(\frak c,\frak d,\frak e,\frak b,\dots,\frak b).
\end{equation}
Here $\langle \cdot,\cdot\rangle_{\text{\rm PD}_X}$ denotes the Poincar\'e duality
pairing.
The right hand side converges
if $\frak b \in H^2(X;\Lambda_+) \oplus \bigoplus_{k>2}
H^k(X;\Lambda_0)$. We can extend it to arbitrary 
$\frak b \in H^*(X;\Lambda_0)$.
(This is well-known. See for example \cite{toric3} section 2.)
\par
$\cup^{\frak b}$ defines a graded commutative and associative ring structure on
$H(X;\Lambda_0)$. We call $\cup^{\frak b}$ the {\it deformed quantum cup product}.
\end{Definition}
\par\bigskip
\subsection{The isomorphism `Jacobian ring = quantum cohomology' and its applications}
\label{JRAC}

\begin{Theorem}\label{Jacqcmain1}
There exists a ring isomorphism
$$
(H(X;\Lambda_0),\cup^{\frak b})
\cong 
\text{\rm Jac}(\frak{PO}^{\frak b}) 
$$
\end{Theorem}

This is \cite{toric3} Theorem 1.1 (1).
We explain some parts of the proof later in this section. We first discuss some applications.
\begin{Definition}
Let
$\text{\rm Crit}(\frak{PO}^{\frak b})$ 
be the set of all $\frak y \in \frak A(\overset{\circ}P)$ such that
$$
\frac{\partial \frak{PO}^{\frak b}}{\partial y_i}(\frak y) = 0
$$
for $i=1,\dots,n$.
An element of $\text{\rm Crit}(\frak{PO}^{\frak b})$  
is said to be a {\it critical point} of $\frak{PO}^{\frak b}$.
\par
A critical point $\frak y$ of $\frak{PO}^{\frak b}$ 
is said to be {\it non-degenerate}
if the matrix
$$
\left[ \frak y_i\frak y_j\frac{\partial^2 \frak{PO}^{\frak b}}{\partial y_i\partial y_j}(\frak y)\right]
_{i,j=1}^{i,j=n}
$$
is invertible, as an $n\times n$ matrix with $\Lambda$ entries.
\par
The function $\frak{PO}^{\frak b}$ is said to be 
a {\it Morse function} if all of its critical points are non-degenerate.
\end{Definition}
We put
$$
\frak M(X;\frak b) 
= \left\{ (\text{\bf u},b) \,\,
\left\vert
\aligned
&\text{\bf u} \in \text{\rm Int}\,\, P, 
b\in H^1(L(\text{\bf u});\Lambda_0)/H^1(L(\text{\bf u});2\pi \sqrt{-1}\Z), 
\\
&HF((L(\text{\bf u}),(\frak b,b)), (L(\text{\bf u}),(\frak b,b));\Lambda_0)
\cong H(T^n;\Lambda_0)
\endaligned
\right\}\right..
$$
Theorem \ref{POandHFbulk} implies the following.
\begin{equation}\label{fraMandcrit}
\# \frak M(X;\frak b) = \# \text{\rm Crit}(\frak{PO}^{\frak b}).
\end{equation}

\begin{Proposition}\label{localization}
There exists a direct product decomposition 
\begin{equation}
\text{\rm Jac}(\frak{PO}^{\frak b}) \otimes_{\Lambda_0}
\Lambda
= \prod_{\frak y \in \text{\rm Crit}(\frak{PO}^{\frak b})}
\text{\rm Jac}(\frak{PO}^{\frak b};\frak y), 
\end{equation}
as a ring.
\par
The factor $\text{\rm Jac}(\frak{PO}^{\frak b};\frak y)$ in the 
right hand side is a local ring.
\par
The ring $\text{\rm Jac}(\frak{PO}^{\frak b};\frak y)$ is 
one dimensional if and only if $\frak y$ is non-degenerate.
\end{Proposition}
This is a standard result in the case, for example, when the function ($\frak{PO}^{\frak b}$
in our case) is a polynomial or a holomorphic function. 
We can prove Proposition \ref{localization} in a similar way to those cases.
It is proved in \cite{toric3} section 5.
\par
Theorem \ref{Jacqcmain1} and Proposition \ref{localization}
imply that $(H(X;\Lambda),\cup^{\frak b})$ is semi-simple 
if and only if $\frak{PO}^{\frak b}$ is a Morse function.
\par
Theorem \ref{Jacqcmain1} together with 
Proposition \ref{localization} and Formula (\ref{fraMandcrit}) 
imply the following:

\begin{Theorem}\label{critHQ}
\begin{enumerate}
\item
If $\frak{PO}^{\frak b}$ is a Morse function then
$$
\text{\rm rank} \,H(X;\Q) =\# \frak M(X;\frak b).
$$
\item
If $\frak{PO}^{\frak b}$ is not a Morse function then
$$
0 < \# \frak M(X;\frak b) < \text{\rm rank}\, H(X;\Q).
$$
\end{enumerate}
\end{Theorem}
This is \cite{toric3} Theorem 1.3. 
Some of the earlier partial results is given in \cite{toric1} Theorems 1.9 and 1.12.
\begin{Remark}
Theorem \ref{critHQ} in particular implies that there exists at least 
one non-displaceable $T^n$ orbit.
This fact also follows from an earlier result by Entov-Polterovich \cite{entov-pol06,EP:rigid}.
\end{Remark}
Another application is the following: 

\begin{Theorem}\label{c1iscrit}{\rm (\cite{toric3} Theorem 1.4.)}
Assume $\frak b \in H^2(X;\Lambda_0)$.
The set of eigenvalues of the map $x \mapsto c_1(X) \cup^{\frak b} x :
H(X;\Lambda) \to  H(X;\Lambda)$ coincides with the set of
critical values of $\frak{PO}^{\frak b}$, with multiplicities counted.
\end{Theorem}
\begin{Remark}
Theorem \ref{c1iscrit} was conjectured by M. Kontsevich.
See also \cite{Aur07}.
\end{Remark}
\begin{proof}
The proof uses the following:
\begin{Lemma}
Let us consider the situation of Theorem \ref{c1iscrit}.
Then, by the isomorphism in Theorem \ref{Jacqcmain1},
the first Chern class $c_1(X) \in H^2(X;\C)$ is sent to 
the equivalence class of $\frak{PO}^{\frak b}$ in 
$\text{\rm Jac}(\frak{PO}^{\frak b})$.
\end{Lemma}
This is \cite{toric3} Proposition 15.1.
\par
Now we consider $x \mapsto c_1(X) \cup^{\frak b} x$.
We use Thoerem \ref{Jacqcmain1} and Proposition \ref{localization}
then it is identified to the direct sum of maps
$$
[F] \mapsto [\frak{PO}^{\frak b}F], \quad 
\text{\rm Jac}(\frak{PO}^{\frak b};\frak y) \to \text{\rm Jac}(\frak{PO}^{\frak b};\frak y).
$$
The eigenvalue of this map is $\frak{PO}^{\frak b}(\frak y)$. 
This implies Theorem \ref{c1iscrit}.
\end{proof}
\par\bigskip
\subsection{Construction of the homomorphism $\frak{ks}_{\frak b}$}
\label{isoconst}

In various applications of Thoerem \ref{Jacqcmain1} it is also important to 
know the way how the isomorphism is defined, which we describe in this subsection. 
\par
Let $\text{\bf p}_i$ be the basis of $\mathcal A$ as in section \ref{bulkFloer}.
We write an element $\frak b \in \mathcal A(\Lambda_0)$ as
$$
\frak b  =\sum_{i=0}^B w_i \text{\bf p}_i.
$$
We put $\frak w_i = e^{w_i}$ for $i=1,\dots,m$. 
(Note $\text{\bf p}_i$, $i=1,\dots,m$ are degree 2 classes.) 
We define $P_{j_0\dots j_B}(y)$ by 
\begin{equation}
\frak{PO}(\frak b;y) 
= \sum_{j_0=0}^{\infty}\cdots\sum_{j_B=0}^{\infty}
P_{j_0\dots j_B}(y)  w_0^{j_0}\frak w_1^{j_1}\cdots \frak w_m^{j_m} 
w_{m+1}^{j_{m+1}}\dots w_B^{j_B}.
\end{equation}
We can show that 
$$
P_{j_0\dots j_B}(y) \in T^{\rho_{j_0\dots j_B}}\Lambda_0^{\overset{\circ}P}\langle\!\langle y,y^{-1}\rangle\!\rangle
$$
with
$$
\lim_{j_0 + \dots + j_B \to \infty} \rho_{j_0\dots j_B} = \infty.
$$
Therefore the right hand side of
\begin{equation}
\aligned
&\frac{\partial}{\partial w_i}\frak{PO}(\frak b;y) \\
&=
\begin{cases}
\displaystyle
\sum_{j_0=0}^{\infty}\cdots\sum_{j_B=0}^{\infty} j_i P_{j_0\dots j_B}(y)  w_0^{j_0}\cdots 
w_i^{j_i-1}\cdots w_B^{j_B}  \quad  i\ne 1,\dots,m\\
\displaystyle
\sum_{j_0=0}^{\infty}\cdots\sum_{j_B=0}^{\infty} j_iP_{j_0\dots j_B}(y)  w_0^{j_0}\cdots 
\frak w_i^{j_i}\cdots w_B^{j_B} 
\quad \quad i = 1,\dots,m
\end{cases}
\endaligned
\end{equation}
makes sense and is contained in 
$\Lambda^{\overset{\circ}P}\langle\!\langle y,y^{-1}\rangle\!\rangle$ for each 
$\frak b\in \mathcal A(\Lambda_0)$.
\par
We define the map
$$
\tilde{\frak{ks}}_{\frak b_0}
: \mathcal A(\Lambda_0) \to \Lambda_0^{\overset{\circ}P}\langle\!\langle y,y^{-1}\rangle\!\rangle
$$
by
\begin{equation}
\tilde{\frak{ks}}_{\frak b_0}(\text{\bf p}_i) 
= \left.\frac{\partial}{\partial w_i}\frak{PO}(\frak b;y)\right\vert_{\frak b = \frak b_0}.
\end{equation}
\begin{Theorem}\label{indepenceKS}
There exists a $\Lambda_0$ module homomorphism ${\frak{ks}}_{\frak b}$ such that
the following diagram commutes:
\begin{equation}
\begin{CD}
\mathcal A(\Lambda_0) @ > {\widetilde{\frak{ks}}_{\frak b}} >>
\Lambda_0^{\overset{\circ}P}\langle\!\langle y,y^{-1}\rangle\!\rangle  \\
@ VVV @ VVV\\
H(X;\Lambda_0) @ > {{\frak{ks}}_{\frak b}} >> \text{\rm Jac}(\frak{PO}^{\frak
b}).
\end{CD}
\end{equation}
\end{Theorem}
The map ${\frak{ks}}_{\frak b}$ is the isomorphism mentioned in Theorem \ref{Jacqcmain1}.
\par
Theorem \ref{indepenceKS} is \cite{toric3} Theorem 7.1.
\begin{proof}[Sketch of the proof]
By definition, we have
\begin{equation}\label{qmap1}
\frak{PO}(\frak b;y) 
= \sum_{k=0}^{\infty}\sum_{\ell=0}^{\infty}\int_{L(\text{\bf u})}\frak q_{\ell;k}(\frak b^{\ell},b^{k}).
\end{equation}
Here $b = \sum_{i=1}^n x_i \text{\bf e}_i$ and $y_i = e^{x_i}$.
Using $\partial \frak b/\partial w_i = \text{\bf p}_i$ we have
\begin{equation}\label{qmap2}
\frac{\partial \frak{PO}(\frak b;y)}{\partial w_i}
= \sum_{k=0}^{\infty}\sum_{\ell_1=0}^{\infty}\sum_{\ell_2=0}^{\infty}
\int_{L(\text{\bf u})}\frak q_{\ell;k}(\frak b^{\ell_1}\text{\bf p}_i\frak b^{\ell_2},b^{k}).
\end{equation}
The homomorphism
\begin{equation}\label{qmap3}
\text{\bf p}_i 
\mapsto  \sum_{k=0}^{\infty}\sum_{\ell_1=0}^{\infty}\sum_{\ell_2=0}^{\infty}
\frak q_{\ell;k}(\frak b^{\ell_1}\text{\bf p}_i\frak b^{\ell_2},b^{k})
\end{equation}
induces a homomorphism
\begin{equation}\label{qmap4}
H(X;\Lambda_0) \to HF((L(\text{\bf u}),(\frak b,b)),(L(\text{\bf u}),(\frak b,b));\Lambda_0).
\end{equation}
This fact was proved in \cite{fooo-book} Theorem 3.8.62 for arbitrary $L \subset X$.
\par
Note that to define (\ref{qmap4}) by (\ref{qmap3}) we fix $\frak b$, $b$ and regard the 
right hand side of  (\ref{qmap3}) as an element of $H(L(\text{\bf u}),\Lambda_0)$. 
When we define $\tilde{\frak {ks}}_{\frak b}$, we regard $b = \sum_{i=1}^n x_i \text{\bf e}_i$, 
as a ($H(L(\text{\bf u}),\Lambda_0)$ valued function of $x_i$.
So the right hand side of (\ref{qmap2}) is a function of $y_i = e^{x_i}$.
\par
In other words we need to study the `family version' of the well-definedness of
 (\ref{qmap4}).
 \par
 We consider the boundary operator
 $$
 a \in H(L(\text{\bf u}),\Lambda_0) \mapsto \frak m_1^{\frak b,b}(a) 
 = 
  \sum_{k_1=0}^{\infty}\sum_{k_2=0}^{\infty}\sum_{\ell=0}^{\infty}
\frak q_{\ell;k}(\frak b^{\ell},b^{k_1} a b^{k_2}).
 $$
 The well-definedness of  (\ref{qmap4}) means the following Claim \ref{claimss}.
 Let $i^*_{\text{\rm qm},(\frak b,b)}(\text{\bf p}_i)$  be the 
 right hand side of (\ref{qmap3}).

\begin{Claim}\label{claimss}
If $\sum_{i=0}^B c_i\text{\bf p}_i$ is  zero in $H(X;\Lambda_0)$, then 
$\sum_{i=0}^B c_i i^*_{\text{\rm qm},(\frak b,b)}(\text{\bf p}_i)$ lies 
in the image of  $\frak m_1^{\frak b,b}$.
\end{Claim}
We can prove the same claim when we regard $b$ as a function of $x_i$.
By the proof of Theorem \ref{CritisHFne0} (especially by Formula (\ref{form21})), 
the image of $\frak m_1^{\frak b,b}$ (where $b$ is regarded as a function of $x_i$) 
is in the Jacobian ideal (the ideal generated by $y_i\partial \frak{PO}^{\frak b}/\partial  y_i$).
\par
Thus the kernel of $\mathcal A(\Lambda_0) \to H(X;\Lambda_0)$ 
is mapped to the  Jacobian ideal  by $\tilde{\frak{ks}}_{\frak b}$.
This implies the theorem.
\end{proof}
Before closing this subsection, we state Theorem \ref{cchangethem} which is a
nonlinear version of Theorem \ref{indepenceKS}.
\par
The potential function with bulk $\frak{PO}^{\frak b}$
is parametrized by $\frak b \in \mathcal A(\Lambda_0)$.
Theorem  \ref{cchangethem}  says that it depends only on the cohomology class $\frak b$ 
up to appropriate change of variables.
$\Lambda_+^{\overset{\circ}P}\langle\!\langle y,y^{-1}\rangle\!\rangle$
denotes the set of elements $R$ of $\Lambda_0^{\overset{\circ}P}\langle\!\langle y,y^{-1}\rangle\!\rangle$
such that $T^{-\epsilon}R \in \Lambda_0^{\overset{\circ}P}\langle\!\langle y,y^{-1}\rangle\!\rangle$
for some $\epsilon > 0$.
\begin{Definition}\label{coordinatechange}
We consider $n$ elements $y_i' \in \Lambda^{\overset{\circ}P}\langle\!\langle y,y^{-1}\rangle\!\rangle $ $(i=1,\dots,n)$.
\begin{enumerate}
\item We say that $y' = (y'_1,\dots,y'_n)$ is a {\it coordinate change
converging on} $\text{\rm Int}\, P$
(or a {\it coordinate change on $\text{\rm Int}\,P$}) if
\begin{equation}\label{leadingcondition}
y'_i \equiv c_i y_i  \mod y_i\Lambda_+^{\overset{\circ}P}\langle\!\langle y,y^{-1}\rangle\!\rangle
\end{equation}
$c_i \in \C \setminus \{0\}$.
\item We say that the coordinate change is {\it strict} if $c_i =1$ for all $i$.
\item We say that the {\it coordinate change
converges on $P$} if
$y_i' \in \Lambda^{P}\langle\!\langle y,y^{-1}\rangle\!\rangle $ $(i=1,\dots,n)$
in addition.
Its strictness is defined in the same way.
We also say that $y'$ is
a {\it coordinate change on $P$}.
\end{enumerate}
\end{Definition}
The set of all coordinate changes forms a group. 
It is regarded as a kind of group of self automorphisms of the 
filtered $A_{\infty}$ algebra associated to $L(\text{\bf u})$.
(The domain of convergence assumed in Definition \ref{coordinatechange} 
requires that it converges not only by the norm $\frak v_T^{\text{\bf u}}$ 
but also by $\frak v_T^{\text{\bf u}'}$  with any $\text{\bf u}'$. 
This is the reason we write ``a kind of'' in the above sentence.)
A closely related group appears in \cite{KS2} and \cite{grospand}.

\begin{Theorem}\label{cchangethem}
Let $\frak b,\frak b' \in \mathcal A(\Lambda_0)$.
We assume that $[\frak b] = [\frak b']
\in H(X;\Lambda_0)$.
\par
Then there exists a coordinate change $y'$ on $\text{\rm Int}\, P$, such that
\begin{equation}\label{POcoordinatechage}
\mathfrak{PO}^{\frak b}(y')
= \mathfrak{PO}^{\frak b'}(y).
\end{equation}
\par
If $\frak b - \frak b' \in \mathcal A(\Lambda_+)$,
then $y'$ can be taken to be strict.
\par
If both $\frak b,\frak b' \in \mathcal A(\Lambda_+)$,
then $y'$ can be taken to be
a strict coordinate change on $P$.
\end{Theorem}

This is \cite{toric3} Theorem 8.7.
\par
We remark that $\Lambda_0^{\overset{\circ}P}\langle\!\langle y,y^{-1}\rangle\!\rangle$
parametrizes the deformation of the potential function. Then the Jacobian ideal corresponds to the 
part induced by the coordinate change. 
Thus Theorem \ref{cchangethem} follows from Theorem \ref{indepenceKS}
by some `integration' (that is solving appropriate ordinary differential equation.)
See \cite{toric3} section 8.
\par\bigskip
\subsection{The homomorphism $\frak{ks}_{\frak b}$ is an isomorphism}
\label{isomproof}

The main geometric input to the proof of Theorem \ref{Jacqcmain1}
is the following:

\begin{Theorem}\label{multiplicative}
The map
$
\frak{ks}_{\frak b}: (H(X;\Lambda_0),\cup^{\frak b})
\to \text{\rm Jac}(\frak{PO}^{\frak b})
$
is a ring homomorphism.
\end{Theorem}
Theorem \ref{multiplicative} is \cite{toric3} Theorem 9.1.
\par
Note this theorem is a version of a result which holds in greater 
generality. Namely there exists a ring homomorphism
\begin{equation}
QH(X;\Lambda_0) \to HH(Fuk(X,\omega)),
\end{equation}
where the right hand side is the Hochschild cohomology of the Fukaya category 
(see \cite{fooo099,anchor} for its definition.)
The existence of such homomorphism was first suggested by \cite{konts:hms} and 
conjectured explicitly by \cite{seidel:HH} etc.
See \cite{toric3} section 31 and the reference therein for some of the related works. 
\par
We remark that $HH(Fuk(X,\omega))$ parametrizes the deformation of 
the Lagrangian Floer theory on $X$. 
The Jacobian ring $\text{\rm Jac}(\frak{PO}^{\frak b})$ parametrizes the deformation of 
a part of the structures, that is the part described by $\frak m^b_0(1)$.   
So there is a natural ring homomorphism 
$HH(Fuk(X,\omega)) \to \text{\rm Jac}(\frak{PO}^{\frak b})$ in the 
toric case. Combining them we obtain the ring homomorphism in Theorem \ref{multiplicative}.
\par
More precise and down-to-earth proof of Theorem \ref{multiplicative} is given as follows.
\par
We recall that the map $
\frak{ks}_{\frak b}: (H(X;\Lambda_0),\cup^{\frak b})
\to \text{\rm Jac}(\frak{PO}^{\frak b})
$ is induced from the map
\begin{equation}\label{qmap3again}
\text{\bf p}_i 
\mapsto  \sum_{k=0}^{\infty}\sum_{\ell_1=0}^{\infty}\sum_{\ell_2=0}^{\infty}
\int_{L(\text{\bf u})}\frak q_{\ell_1 + \ell_2 + 1;k}(\frak b^{\ell_1}\text{\bf p}_i\frak b^{\ell_2},b^{k}) : 
\mathcal A \to \Lambda_0^P\langle\!\langle y,y^{-1} \rangle\!\rangle.
\end{equation}
(See (\ref{qmap3}).)
Note $b = \sum x_i\text{\bf e}_i$ and the right hand side is a function of $x_i$.
It then turns out to be a function of $y_i^{\text{\bf u}} = e^{x_i}$.
Moreover by changing the variables to $y_i$ by the formula 
$y_i = T^{u_i}y_i^{\text{\bf u}}$, the right hand side becomes a function 
of $y_i$ and is an element of $ \Lambda_0^P\langle\!\langle y,y^{-1} \rangle\!\rangle$.
\par
We consider the case $\frak b =\text{\bf 0}$ for simplicity.

We consider the moduli space $\mathcal M_{k+1;2}(\beta)$ of 
$J$-holomorphic disks with $k+1$ boundary and $\ell$ interior marked points,
(See subsection \ref{subsec:moduli}.)
and take a fiber product
$$
\mathcal M^{\text{\rm main}}_{k+1;2}(\beta) {}_{({\rm ev}_1^+,{\rm ev}_2^+)} \times (\text{\bf p}\times \text{\bf p}')
$$
where $\text{\bf p},\text{\bf p}' \in \mathcal A$.
We denote this fiber product by 
$$
\mathcal M^{\text{\rm main}}_{k+1;2}(\beta;\text{\bf p},\text{\bf p}' ).
$$
Let $\mathcal M_{1;2}$ be the moduli space of 
bordered Riemann surface of genus $0$ with two interior and 
one boundary marked points.
This moduli space is a two dimensional disk. We consider two 
points $[\Sigma_1],[\Sigma_2] \in \mathcal M_{1;2}$ as in the figure below.
\par\medskip
\hskip1cm\epsfbox{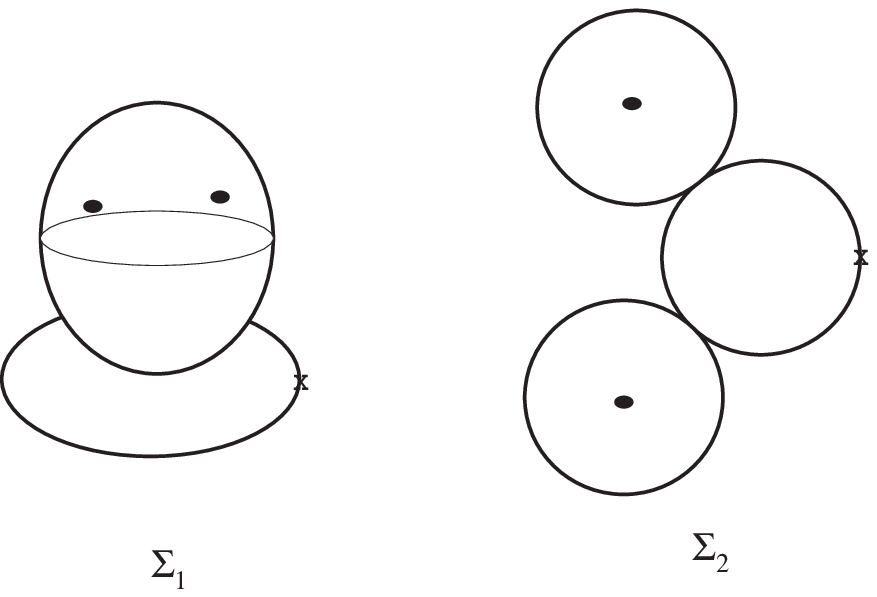}
\par\smallskip
\centerline{\bf Figure 11.1}
\par\bigskip
\par
We have a forgetful map
\begin{equation}
\frak{forget} : 
\mathcal M^{\text{\rm main}}_{k+1;2}(\beta )
\to \mathcal M_{1;2}.
\end{equation}
Namely we put 
$$
\frak{forget} ([\Sigma;z_0,\dots,z_k,z_1^+,z_2^+,u])
= [\Sigma;z_0;z_1^+,z_2^+].
$$
It induces a map 
$$
\frak{forget} : \mathcal M^{\text{\rm main}}_{k+1;2}(\beta;\text{\bf p},\text{\bf p}')
\to \mathcal M_{1;2}.
$$
For $i=1,2$, we denote by
$$
\mathcal M^{\text{\rm main}}_{k+1;2}(\beta;\text{\bf p},\text{\bf p}' ;\Sigma_i)
$$
the inverse image of $\{[\Sigma_i]\}$ in 
$\mathcal M^{\text{\rm main}}_{k+1;2}(\beta;\text{\bf p},\text{\bf p}')$.
\par
Let $h_j \in H^1(L(\text{\bf u});\C)$
($j=1,\dots, k$). (Note we identify the cohomology group 
with the set of $T^n$ invariant forms.) We pull back
$h_1\times \dots\times h_k$ to 
$\mathcal M^{\text{\rm main}}_{k+1;2}(\beta;\text{\bf p},\text{\bf p}';\Sigma_i)$
by $({\rm ev}_1,\dots,{\rm ev}_k)$ 
and consider the integration along fiber by 
${\rm ev}_{0}$. We denote it by 
\begin{equation}\label{fixq}
\mathrm{Corr}(h_1\times\dots\times h_k;
\mathcal M^{\text{\rm main}}_{k+1;2}(\beta;\text{\bf p},\text{\bf p}';\Sigma_i)).
\nonumber\end{equation}
 More precisely we take a $T^n$ invariant multisection $\frak s$ so that 
the zero set 
$\mathcal M^{\text{\rm main}}_{k+1;2}(\beta;\text{\bf p},\text{\bf p}';\Sigma_i)^{\frak s}$
is transversal to zero. Then integration along the fiber is 
well-defined. This is because ${\rm ev}_{0}$ on 
$\mathcal M^{\text{\rm main}}_{k+1;2}(\beta;\text{\bf p},\text{\bf p}';\Sigma_i)^{\frak s}$
must become a submersion by the $T^n$ equivariance.
\par
We put
$$
\aligned
&\mathrm{Corr}(h_1\times\dots\times h_k;
\mathcal M^{\text{\rm main}}_{k+1;2}(\text{\bf p},\text{\bf p}';\Sigma_i)) \\
&= \sum_{\beta} T^{(\beta\cap \omega)/2\pi}
\mathrm{Corr}(h_1\times\dots\times h_k;
\mathcal M^{\text{\rm main}}_{k+1;2}(\beta;\text{\bf p},\text{\bf p}';\Sigma_i))
\endaligned
$$
and extend $\mathrm{Corr}(\cdots;
\mathcal M^{\text{\rm main}}_{k+1;2}(\text{\bf p},\text{\bf p}';\Sigma_i))
$ to 
$$
H^1(L(\text{\bf u});\Lambda_0)^{\otimes k} \to \Lambda_0.
$$
We then can prove
the following two formulas:
\begin{equation}\label{fibersigma1}
\mathrm{Corr}(
\underbrace{b,\dots,b}_{k};
\mathcal M^{\text{\rm main}}_{k+1;2}(\text{\bf p},\text{\bf p}';\Sigma_1))
= 
\frak q_{1;k}(\text{\bf p}\cup^Q\text{\bf p}';b^{k}).
\end{equation}
\begin{equation}\label{fibersigma2}
\mathrm{Corr}(
\underbrace{b,\dots,b}_{k};
\mathcal M^{\text{\rm main}}_{k+1;2}(\text{\bf p},\text{\bf p}';\Sigma_2))
=
\sum_{k_1+k_2=k}
\frak q_{1;k_1}(\text{\bf p};b^{k_1})
\frak q_{1;k_2}(\text{\bf p}';b^{k_2}).
\end{equation}
Note the sum over $k$ of the right hand sides of (\ref{fibersigma1}) 
and (\ref{fibersigma2}) are 
$$
\frak{ks}_{\text{\bf  0}}(\text{\bf p}\cup^Q\text{\bf p}')
\qquad \text{and} 
\qquad
\frak{ks}_{\text{\bf  0}}(\text{\bf p})
\frak{ks}_{\text{\bf  0}}(\text{\bf p}')
$$
respectively.
(Note we are studying the case $\frak b = \text{\bf  0}$.) 
\par
We finally use cobordism argument to show that the 
left hand side of (\ref{fibersigma1}) 
coincides with the left hand side of (\ref{fibersigma2})
modulo elements in the Jacobian ideal.
This is an outline of the proof of Theorem 
\ref{multiplicative}.
See \cite{toric3} section 9 for detail.
\par\medskip
The outline of the rest of the proof of 
Theorem \ref{Jacqcmain1} is as follows.
\par
We first prove the surjectivity of $\frak{ks}_{\frak b}$.
For this purpose we consider the map obtained by 
reducing the coefficient to $\C = \Lambda_0/\Lambda_+$.
Then the quantum cohomology of the domain becomes ordinary 
cohomology. We can calculate the $\C = \Lambda_0/\Lambda_+$ 
reduction of the Jacobian ring using Cho-Oh's result
(namely by studying the leading order term 
$z_1 + \dots +z_m$. See Theorem \ref{calcPObulk}.)
Then the $\C$-reduction of $\frak{ks}_{\frak b}$ 
is an isomorphism by a classical result of Stanley 
which calculates the cohomology ring of toric manifold.
(See for example \cite{fulton}.)
It implies that $\frak{ks}_{\frak b}$  
itself is surjective.
\par
We remark that the fact that $\C$-reduction of $\frak{ks}_{\frak b}$
is an isomorphism does not imply that 
$\frak{ks}_{\frak b}$ is isomorphism.
In fact we need to eliminate the possibility 
that $\text{\rm Jac}(\frak{PO}^{\frak b})$ has 
a component such as $\Lambda_0/(T^{\lambda})$.
Note that 
the (quantum) cohomology $H(X;\Lambda_0)$ 
is a free $\Lambda_0$ module.
Therefore to prove the injectivity of $\frak{ks}_{\frak b}$ 
and complete the proof of Theorem \ref{Jacqcmain1} 
it suffices to prove the following inequality.
\begin{equation}\label{mainineqality}
\text{\rm rank}_{\Lambda} (\text{\rm Jac}(\frak{PO}^{\frak b}) \otimes
_{\Lambda_0}\Lambda)
\ge 
\text{\rm rank}_{\Q} H(X;\Q).
\end{equation}
We remark that in many explicit examples we can prove 
the equality (\ref{mainineqality}) directly by finding critical points of 
$\frak{PO}^{\frak b}$, for example by solving leading term 
equation.
However the proof of (\ref{mainineqality}) is in general more involved, which 
we briefly describe now. 
We consider the case $\frak b = \text{\bf 0}$, for simplicity.
\par
We prove (\ref{mainineqality}) in two steps. 
We first use a result of McDuff-Tolman
\cite{mc-tol} (which is based on Seidel's work 
\cite{seidel:auto}), to find elements 
$z'_1,\dots,z'_m \in HQ(X;\Lambda_0)$
with the following properties.
\begin{enumerate}
\item
$z'_1,\dots,z'_m$ satisfies quantum Stanley-Reisner relation.
\item
There exists $\frak P_i(Z_1,\dots,Z_m) = 
\sum_{j=1}^m v_{j,i} Z_i + \sum_{k=1}^{\infty}T^{\rho_k}
P_{i,k}(Z_1,\dots,Z_m)$
such that
\begin{equation}\label{relationforQh}
\frak P_i(z'_1,\dots,z'_m) = 0
\end{equation}
and $\rho_k \to \infty$, $\rho_k >0$, $P_{i,k} \in \C[Z_1,\dots,Z_m]$.
(We recall $d\ell_j = ( v_{j,1},\dots, v_{j,n}) \in \Z^n$.)
\item
The relations in the above (1),(2) are all the relations among 
$z'_i$. Moreover $z'_i$ generates  $HQ(X;\Lambda_0)$.
\end{enumerate}
Let us explain the above statement briefly.
By putting 
$Z_i = T^{\lambda_i}y_1^{v_{i,1}}\cdots y_n^{v_{i,n}}$ 
we obtain a surjective  ring homomorphism
$$
\Lambda[Z_1,\dots,Z_m] 
\to \Lambda[y_1,y_1^{-1},\dots,y_n,y_n^{-1}].
$$
The quantum Stanley-Reisner relations are the generators of the 
kernel of this homomorphism.
(See \cite{toric1} Definition 6.4.)
The quantum Stanley-Reisner relation appeared in the Batyrev's work 
on quantum cohomology of toric manifold and is given explicitly 
by using moment polytope $P$.
\par
We put $z_i  = T^{\lambda_i}y_1^{v_{i,1}}\cdots y_n^{v_{i,n}} 
\in \text{\rm Jac}(\frak{PO}^{\text{\bf 0}})$.
Then (\ref{POmainformulabulk}) implies that 
it satisfies the formula
\begin{equation}\label{relationforJac}
\sum_{j=1}^m v_{j,i} z_i + \sum_{k=1}^{\infty}T^{\rho_k}
\frac{\partial P_{k}}{\partial x_i}(z_1,\dots,z_m) =0.
\end{equation}
(Note we put $y_i = e^{x_i}$ so (\ref{relationforJac}) is 
$\frac{\partial P_{k}}{\partial x_i} = y_i\frac{\partial P_{k}}{\partial y_i}$.)
We remark that the first term of the left hand side of (\ref{relationforJac}) is
$$
\frac{\partial \frak{PO}^{\text{\bf 0}}}{\partial x_i}(z_1,\dots,z_m).
$$
We also remark that the left hand side of 
(\ref{relationforQh}) is similar to (\ref{relationforJac}). 
Namely their leading order terms coincide. 
\par
The element $z'_i$ is the invariant of \cite{seidel:auto} associated 
to the Hamiltonian $S^1$ action. Here $S^1$ is the component of $T^n$ which 
fixes $D_i$.
The fact that they satisfy the quantum Stanley-Reisner relation 
is proved in \cite{mc-tol} using the relation between those $S^1$ actions 
and basic properties of Seidel invariant.
The property (2) can be proved using the fact 
$z'_i \equiv [D_i] \mod \Lambda_+$.
\par
Let $(QSR) \subset \Lambda_0\langle\!\langle Z_1,\dots,Z_m
\rangle\!\rangle$ be the ideal generated by the 
quantum Stanley-Reisner relations.
Then (1), (2) above imply the existence of homomorphism
\begin{equation}\label{PandQH}
\frac{\Lambda_0\langle\!\langle Z_1,\dots,Z_m
\rangle\!\rangle }{\text{Clos}((QSR) \cup \{\frak P_i : i=1,\dots,m\})}
\to 
QH(X;\Lambda_0).
\end{equation} 
here $\text{Clos}$ means a closure with respect to an appropriate topology.
By reduction to $\C = \Lambda_0/\Lambda_+$ we can show that 
(\ref{PandQH}) is an isomorphism.
(We use the fact that $QH(X;\Lambda_0)$ is torsion free here.)
\par
Now the proof of (\ref{mainineqality}) goes as follows.
For $\frak s \in \Lambda$ we put
$$
\frak P_i^{\frak s} =
\frak s \frac{\partial \frak{PO}^{\text{\bf 0}}}{\partial x_i}
+ (1-\frak s)\frak P_i.
$$
We remark $\frak P_i^{\frak s} $ has the form 
$$
\frak P^{\frak s}_i(Z_1,\dots,Z_m) = 
\sum_{j=1}^m v_{j,i} Z_i + \sum_{k=1}^{\infty}T^{\rho_k}
P^{\frak s}_{i,k}(Z_1,\dots,Z_m).
$$
We define the ring $\frak R_{\frak s}$ by
$$
\frak R_{\frak s} = 
\frac{\Lambda_0\langle\!\langle Z_1,\dots,Z_m
\rangle\!\rangle }{\text{Clos}((QSR) \cup \{\frak P^{\frak s}_i : i=1,\dots,m\})}
\otimes_{\Lambda_0}\Lambda.
$$
We have
$$
\frak R_{0} \cong QH(X;\Lambda)
$$
since (\ref{PandQH}) is an isomorphism.
On the other hand
$$
\frak R_{1} \cong \text{\rm Jac}(\frak{PO}^{\text{\bf 0}})\otimes_{\Lambda_0}\Lambda.
$$
Thus it suffices to show that $\dim_{\Lambda}\frak R_{\frak s}$ is 
independent of $\frak s$. 
We regard $\cup_{\frak s \in \Lambda}Spec(\frak R_{\frak s})$ as a 
family of affine schemes parametrized by $\frak s \in \Lambda$.
If we can prove that this family is flat and proper then 
the independence of  $\dim_{\Lambda}\frak R_{\frak s}$ is a 
standard result of algebraic geometry.
\par
We prove the properness using the fact that 
the valuation of the solution of the equation 
$\frak P^{\frak s}_1= \dots =\frak P^{\frak s}_m=0$ can not escape from 
moment polytope.
The flatness is a consequence of the fact that 
our scheme is a local complete intersection and 
also of standard facts about the regular sequence of 
Cohen-Macauley ring.
\par
In general $\frak P^{\frak s}_i$ is an infinite series rather 
than a polynomial.  So we first need to change the coordinate $y_i$ so 
that $\frak P^{\frak s}_i$ becomes a polynomial.
Such a process is known in algebraic geometry as a 
algebraization of singularity. See \cite{toric3}  section 12.
\par
This is an outline of the proof of (\ref{mainineqality}).
See \cite{toric3} especially its section 14 for details.
\qed
\begin{Remark}
We regard
\begin{equation}\label{familyoverb}
\bigcup_{\frak b \in H(X;\Lambda_0)} 
Spec(\frak{PO}^{\frak b})
\end{equation}
as a $H(X;\Lambda_0)$ parametrized `family of schemes'
\footnote{It is proved in \cite{toric3} that each of $\frak{PO}^{\frak b}$ can 
be transformed to a Laurent polynomial by change of variables.
So we can define  its $Spec$. It is not verified that the whole family can be 
regarded to be a scheme. So we put quotation mark.}
\par
The same argument to show the flatness and properness 
of the family $\cup_{\frak s \in \Lambda}Spec(\frak R_{\frak s})$ 
seems to be applicable to show that the family (\ref{familyoverb}) 
is also flat and proper.
\par
In the study of K. Saito theory of Laurent polynomials 
(such as one described in \cite{saba}), 
the properness of the family of the critical point sets
is an important issue. 
When one works over $\C$ the properness is not 
necessarily satisfied. 
When we work over a Novikov ring in place of $\C$, 
properness of the family of the critical point sets
(that is the geometric points of  $Spec(\frak{PO}^{\frak b})$) 
is always satisfied at least for the potential function 
appearing as the mirror of a toric manifold.
The authors believe that this is an important advantage of working
with Novikov ring over working with $\C$.
\end{Remark}
\begin{Remark}
Let us consider the family (\ref{familyoverb}).
For the $H^2(X;\Lambda_0)$ part of $\frak b$ it is 
natural to replace the coordinate $w_i$ by its exponential 
$\frak w_i = e^{w_i}$.
Then we may extend the domain 
$ \{\frak w_i \mid \frak v_T(\frak w_i) = 0\}$
to $\frak w_i \in \Lambda$.
Note in $\frak{PO}^{\frak b}$ 
the leading order term is 
$\sum \frak w_i z_i$. 
So if we extend $\frak w_i$ and allow 
for example $\frak w_i = T^{c}$, we have a term 
$T^cz_i$. We may regard this insertion $\frak w_i = T^{c}$ as  
changing the moment polytope. Namely 
appearance of the term $T^cz_i$  is  equivalent 
to moving $\partial_iP = \{\text{\bf u} 
\mid \ell_i(\text{\bf u}) = 0\}$ 
to $\{\text{\bf u} 
\mid \ell_i(\text{\bf u}) = -c\}$\footnote{In other 
words the parameter $\frak v_T(\frak w_i)$ corresponds to the 
K\"ahler cone of our toric manifold $X$. 
This is similar to the fact that the valuation of $\frak y_i$ corresponds to the 
parameter $\text{\bf u}$ of the Lagrangian submanifold $L(\text{\bf u})$}.
\par
Thus for this extended family the flatness and 
properness still hold as far as the 
corresponding moment polytope is combinatorially equivalent to the 
original one.
\par
There is some flavor of this kind of arguments in 
\cite{toric3} subsection 14.2. 
\end{Remark}
\par\bigskip
\section{Poincar\'e duality and Residue pairing}
\label{PD}

In this section we explain that the isomorphism 
in Theorem \ref{Jacqcmain1} can be enhanced to 
give an isomorphism between two Frobenius manifold structures.
\par
\subsection{Big quantum cohomology and Frobenius manifold}
\label{subseqQC}

\begin{Definition}\label{deffrebenius}
A {\it Frobenius manifold structure} on a manifold $M$ 
is a quintet $(\langle \cdot \rangle,\nabla,\circ,e,\Phi)$ 
with the following properties.
\begin{enumerate} 
\item $\langle \cdot \rangle$ is a non-degenerate inner product 
on the tangent bundle $TM$.
\item
$\nabla$ is a connection of $TM$.
\item
$\nabla$ is a metric connection. Namely :
$$
X\langle Y,Z \rangle = \langle \nabla_XY,Z \rangle + 
 \langle Y,\nabla_XZ \rangle.
$$
\item
$\nabla$ is flat and torsion free. Namely :
$$
\aligned
\nabla_X \nabla_Y - \nabla_Y \nabla_X - \nabla_{[X,Y]} &= 0, \\
\nabla_X Y - \nabla_Y X - [X,Y] &= 0.
\endaligned$$
\item $\circ$ defines a ring structure on $T_pM$ which depends 
smoothly on $p$ and satisfies
\begin{equation}\label{frobeniusalg}
\langle X\circ Y,Z \rangle =
\langle X,Y\circ Z \rangle.
\end{equation}
An associative algebra with unit which satisfies (\ref{frobeniusalg})
is called a {\it Frobenius algebra}.
\item $e$ is a section of $TM$ such that $e(p)$ is the unit of 
$(T_pM,\circ,+)$ for each $p$. Moreover
$$
\nabla e = 0.
$$
\item
$\Phi$ is a function on $M$ such that
\begin{equation}\label{potentiality}
\left\langle \frac{\partial}{\partial x_i} \circ \frac{\partial}{\partial x_j},
\frac{\partial}{\partial x_k} \right\rangle
= \frac{\partial^3\Phi}{\partial x_i\partial x_j\partial x_k}.
\end{equation}
Here $x_i$ $(i=1,\dots,\dim M)$ is a local coordinate  of $M$ 
such that $\nabla_{\frac{\partial}{\partial x_i}}(\frac{\partial}{\partial x_j}) = 0$.
We call $\Phi$ the {\it potential}.
\end{enumerate} 
\par
In some case we have a vector field 
$\frak E$  on $M$ that satisfies the following
\begin{equation}\label{euler}
\aligned
\frak E\langle X,Y\rangle - \langle [\frak E,X],Y\rangle
- \langle X, [\frak E,Y]\rangle &= d_1\langle X,Y\rangle, \\
 [\frak E, X \circ Y] -  [\frak E, X] \circ Y -  X \circ [\frak E,Y]  &= d_2X \circ Y,\\
 [\frak E, e]  &= d_3e,
\endaligned
\end{equation}
where $d_1,d_2,d_3 \in \Q$.
We call $\frak E$ the {\it Euler vector field}.
\end{Definition}
\begin{Remark}
In various situations where a Frobenius manifold arises 
the tangent space $T_pM$ appears as either a $\C$ vector space 
or a $\Lambda$ vector space.
In that case the inner product $\langle \cdot \rangle$ 
is bilinear over $\C$ or $\Lambda$. (In this case $\langle \cdot\rangle$ is 
required to be complex symmetric {\it not hermitian}.)
Moreover $\Phi$ is a $\C$ or $\Lambda$ valued function.
\par
We do not try to define what connection, funciton, coordinate etc.
mean in case $TM$ is a $\Lambda$ vector space.
At the present stage of development, we do not meet the 
situation where we need to seriously study it. 
In the main example of our consideration, $M$ is  a $\Lambda_0$ affine  
space, hence we can easily make sense out of them.
\end{Remark}
This structure first appeared in K. Saito's work \cite{Sai83} 
(see the next subsection). 
Dubrovin \cite{dub} discovered this structure in Gromov-Witten theory, which we recall 
below.
\par
Let $X$ be a symplectic manifold. We take $M = H^{evev}(X;\Lambda_0)$
the even degree cohomology group of $X$ with $\Lambda_0$ coefficient.
(We may include odd degree part by regarding $X$ as a supermanifold.
Since in the case of our main interest (toric manifold), there is no 
cohomology class of odd degree, we do not discuss odd degree part.)
\par
In subsection \ref{bigquatum} we associate a deformed quantum 
cup product $\cup^{\frak b}$ on $H(X;\Lambda)$ for each 
$\frak b \in H^{even}(X;\Lambda_0)$.
We regard $T_{\frak b} M = H(X;\Lambda)$
and put $\circ= \cup^{\frak b}$ there.
It is associative.
\begin{Remark}
Note $ H^{even}(X;\Lambda_0)$ is not an open set of $ H^{even}(X;\Lambda)$.
So  $T_{\frak b} H^{even}(X;\Lambda_0) = H(X;\Lambda)$ 
does not make sense in a usual sense of manifold.
This is regarded only as a convention here.
\end{Remark}
We have Poincar\'e duality pairing 
$$
H^{d}(X;\Lambda) \otimes_{\Lambda}
 H^{2n-d}(X;\Lambda)  \to \Lambda.
$$
The inner product $\langle\cdot\rangle$ is the 
Poincar\'e duality pairing.
We remark that then the Levi-Civita connection, 
that is the connection which is a torsion free metric connection  
of the metric $\langle\cdot\rangle$, is the standard affine 
connection of the vector space $H^{even}(X;\Lambda_0)$.
It is obviously flat.
\par
(\ref{frobeniusalg}) follows from
$$
\langle \frak c \cup^{\frak b} \frak d, \frak e\rangle_{\text{\rm PD}_X}
= \sum_{\ell=0}^{\infty} \frac{1}{\ell!}GW_{\ell+3}(\frak c,\frak d,\frak e,\frak b,\dots,\frak b).
$$
(See (\ref{defcup}).) and the fact that $GW_{\ell}(Q_1,\dots,Q_{\ell})$ is independent of the 
permutation of $Q_i$.
\par
The element $e$ is the unit of the cohomology group that is the 
Poincar\'e dual to the fundamental homology class $[X]$.
\par
The potential $\Phi$ is defined by
\begin{equation}\label{GWpotential}
\Phi(\frak b)
= \sum_{\ell=0}^{\infty} \frac{1}{\ell!}GW_{\ell}(\frak b,\dots,\frak b)
\end{equation}
for which the formula (\ref{potentiality}) can be easily checked.
The potential $\Phi$ in (\ref{GWpotential}) is called the {\it Gromov-Witten potential}.
\par
The Euler  vector field $\frak E$ is defined by the vector field:
\begin{equation}\label{eulerquatumcoh}
\frak E
= \frac{\partial}{\partial w_0} 
+ \sum_{i=1}^m  r_i\frac{\partial}{\partial \frak w_i} + \sum_{i=m+1}^B 
\left(1 - \frac{\text{\rm deg}\,\text{\bf p}_i}{2}\right)
w_i\frac{\partial}{\partial w_i},
\end{equation}
where $c_1(X) = \sum_{i=1}^m r_i\text{\bf p}_i$.
We remark that $\text{\bf p}_i$, $i=0,\dots,B$ are basis of 
$H^{even}(X;\Q)$ such that $\text{\rm deg}\,\text{\bf p}_0 = 0$,
$\text{\rm deg}\,\text{\bf p}_i = 2$ for $i=1,\dots,m$ and 
$\text{\rm deg}\,\text{\bf p}_i > 2$ for $i>m$.
\par
By using the dimension formula 
$$
\dim_{\C} \mathcal M_{\ell}(\alpha) = n + \ell - 3 + c_1(X) \cap \alpha
$$
of the moduli space 
$\mathcal M_{\ell}(\alpha)$ 
of pseudo-holomorphic sphere with $\ell$ interior marked points and of 
homology class $\alpha$, we can prove (\ref{euler}), where 
$d_1 = 2-n$, $d_2 = 1$, $d_3 = 0$.
Thus we have:
\begin{Theorem}{\rm (Dubrovin)}\label{FrobDub}
$(\langle\cdot\rangle,\nabla,\cup^{\frak b},\Phi,e)$ is 
a structure of Frobenius manifold on $H(X;\Lambda_0)$.
$(\ref{eulerquatumcoh})$ is its Euler vector field.
\end{Theorem}
\par\bigskip
\subsection{A fragment of K. Saito theory.}
\label{subseqSaito}

Let
\begin{equation}\label{germfunc}
F(x_1,\dots,x_n;w_0,w_1,\dots,w_B)
:  U \times V \to \C
\end{equation}
be a holomorphic function on $U\times V \subset \C^n \times \C^{B+1}$.
Here $U$ and $V$ are small neighborhoods of origin in  
$\C^n$ and $\C^{B+1}$, respectively.
\par
We assume $F$ is of the form
$$
F(x_1,\dots,x_n;w_0,w_1,\dots,w_B) = w_0 + F(x_1,\dots,x_n;0,w_1,\dots,w_B). 
$$
We put
$$
F^{\vec w}(x_1,\dots,x_n) = F(x_1,\dots,x_n;w_0,w_1,\dots,w_B),
$$
for $\vec w = (w_0,\dots,w_B)$.
We assume that $F^{\vec 0}(x_1,\dots,x_n)$ has $\vec x = \vec 0$ as an isolated critical 
point. Namely 
$
(d F^{\vec 0})(0,\dots,0)  = 0, 
$
and 
$(d F^{\vec 0})(\vec x)  \ne 0$ for $\vec x \in U \setminus \{\vec 0\}$.
\begin{Definition}
We define the {\it Jacobian ring} $\text{\rm Jac}(F^{\vec w})$ by
\begin{equation}\label{Jacgerm}
\text{\rm Jac}(F^{\vec w}) 
= \frac{\mathcal O(U)}{\left( \frac{\partial F^{\vec w}}{\partial x_i} ; 
i=1,\dots,n\right)}
\end{equation}
Here $\mathcal O(U)$ is the ring of holomorphic functions on $U$ 
and the denominator is its ideal generated by 
$\frac{\partial F^{\vec w}}{\partial x_i}$, $i=1,\dots,n$.
\par
We define the {\it Kodaira-Spencer map} $\frak{ks}_{\vec w} : T_{\vec w}V \to \text{\rm Jac}(F^{\vec w}) $
by
\begin{equation}\label{KSgerm}
\frak{ks}_{\vec w}\left( \frac{\partial}{\partial w_i}\right)
\equiv \frac{\partial F}{\partial w_i}(x_1,\dots,x_n;{\vec w}) 
\in \text{\rm Jac}(F^{\vec w}).
\end{equation}
\par
$F$ is called a {\it universal unfolding} of $F^{\vec 0}$ if 
$\frak{ks}_{\vec 0} : T_{\vec 0}V \to \text{\rm Jac}(F^{\vec 0})$ is an isomorphism.
\end{Definition}
We remark that if $F$ is a universal unfolding of $F^{\vec 0}$ 
then by shrinking $V$ if necessary we may assume that $\frak{ks}_{\vec w}$ 
is an isomorphism for any $\vec w \in V$.
We assume it in the rest of this subsection.
\par
We remark that $\text{\rm Jac}(F^{\vec w}) $ is a ring. 
On the other hand $T_{\vec w}V$ do not have a ring structure a priori.
We {\it define}
\begin{equation}
X \circ Y = (\frak{ks}_{\vec w})^{-1}(\frak{ks}_{\vec w}(X)\frak{ks}_{\vec w}(Y)),
\end{equation}
for $X,Y \in T_{\vec w}V$.
Thus $(T_{\vec w}V,\circ,+)$ forms a ring.
Note $\partial/\partial w_0 \in T_{\vec w}V$ is sent to $[1] \in \text{\rm Jac}(F^{\vec w}) $.
Therefore 
$$
e(\vec w) = \partial/\partial w_0 \in T_{\vec w}V
$$
is a unit.
\par
\begin{Theorem}\text{\rm (K.Saito-M.Saito)}\label{existfrob}
There exists a $\C$ valued metric $\langle \cdot \rangle$ on $TV$, 
its Levi-Civita connection $\nabla$ 
and a holomorphic function $\Phi : V \to \C$ such that
$(\langle \cdot \rangle,\circ,e,\nabla,\Phi)$ is a Frobenius manifold.
\end{Theorem}
K. Saito \cite{Sai83} constructed a Frobenius manifold structure 
assuming the existence of a primitive form.
We do not explain the notion of primitive form here. 
(See \cite{SaiTaka} for its description 
in a way closely related to the discussion here.)
Existence of primitive form for a universal unfolding of a 
germ of isolated singularity is established in 
\cite{Msaito}.
We remark that Theorem \ref{existfrob} had been proved before 
Gromov-Witten theory started. 
\par
The metric $\langle \cdot \rangle$ is called a {\it residue paring}.
Since $\nabla$ is flat there exists a local coordinate $t_0,t_1,\dots,t_B$ of $V$ so that 
$\nabla_{\partial/\partial t_i}(\partial/\partial t_j)  = 0$.
Such a coordinate $(t_0,t_1,\dots,t_B)$ is called a {\it flat coordinate}.
($t_0 =w_0$.)
\par
For some $F$ associated to an ADE singularity, the primitive form takes a simple form
$dx_1\wedge dx_2 \wedge dx_3$. In such a case 
we have the following description of the residue pairing.
\par
We put
$$
\text{\rm Crit}(F^{\vec w})
= \{\frak y \in U \mid dF^{\vec w}(\frak y) = 0\}.
$$
Let $\mathcal O_{\frak y}$ be the ring of germs of holomorphic functions 
at $\frak y \in U$. We put
\begin{equation}\label{Jaclocy}
\text{\rm Jac}(F^{\vec w};\frak y) 
= \frac{\mathcal O_{\frak y}}{\left( \frac{\partial F^{\vec w}}{\partial x_i} ; 
i=1,\dots,n\right)}
\end{equation}
The following fact is standard:
\begin{Proposition}\label{localizationC}
We have
$$
\text{\rm Jac}(F^{\vec w})  
\cong \prod_{\frak y \in \text{\rm Crit}(F^{\vec w})}\text{\rm Jac}(F^{\vec w};\frak y). 
$$
$\text{\rm Jac}(F^{\vec w};\frak y)$ is one dimensional if and only 
if the critical point $\frak y$ is non-degenerate.
\end{Proposition}
Let $\vec w$ be a vector such that $F^{\vec w}$ is a Morse function.
Let $1_{\frak y} \in \text{\rm Jac}(F^{\vec w};\frak y)$ be the unit.
Then Proposition \ref{localizationC} implies that 
$\{ 1_{\frak y} \mid \frak y \in \text{\rm Crit}(F^{\vec w})\}$ 
forms a $\C$ basis of the vector space $\text{\rm Jac}(F^{\vec w})$.
If $\frak y \ne \frak y'$ we obtain
$$
\langle  1_{\frak y},1_{\frak y'}\rangle
= \langle  1_{\frak y},1_{\frak y'}\circ 1\rangle
= \langle  1_{\frak y}\circ 1_{\frak y'},1\rangle = 0,
$$
from the equation $1_{\frak y}\circ 1_{\frak y'} = 0$ and (\ref{frobeniusalg}).
Namely $\{ 1_{\frak y} \mid \frak y \in \text{\rm Crit}(F^{\vec w})\}$ is 
an orthogonal basis with respect to the residue pairing.
\par
\begin{Lemma}\label{Hessdetres}
If the primitive form is $dx_1\wedge \dots\wedge dx_n$ and 
$F^{\vec w}$ is a Morse function then we have
$$
\langle  1_{\frak y},1_{\frak y}\rangle
= 
\left(\det 
\left[ \frac{\partial^2 F^{\vec w}}{\partial x_i\partial x_j}\right]_{i=1,j=1}^{i=n,j=n}
(\frak y)\right)^{-1}.
$$
\end{Lemma}
This lemma follows from the definition.
We remark that in general the primitive form is not necessarily 
equal to $dx_1\wedge \dots\wedge dx_n$.
\par\bigskip
\subsection{Residue pairing on $\text{\rm Jac}(\frak{PO}^{\frak b})$}
\label{subseqRP}

We now consider the case $F(x_1,\dots,x_n,\vec w)= \frak{PO}^{\frak b}(y_1,\dots,y_n)$
where $\frak b = \sum w_i \text{\bf p}_i$ and $e^{x_i} = y_i$.
\par
We however remark that our situation is different from that of 
subsection \ref{subseqSaito} in the following two points.
\begin{enumerate}
\item The tangent space $T_{\frak b}(H(X;\Lambda_0))$ is a $\Lambda$ vector space and is not
a $\C$ vector space.
\item The `open set' on which $\frak{PO}^{\frak b}$ is defined  is the set $\frak A(\overset{\circ}P)$ which is not 
a `small' neighborhood of a point.
\end{enumerate}
However, many parts of the story are directly translated to the case 
$\frak{PO}^{\frak b}$.
(See however Remark \ref{1224}.)
Note $V$ in subsection \ref{subseqSaito} corresponds to 
$H(X;\Lambda_0)$.
\par
In this subsection we describe a pairing on $\text{\rm Jac}(\frak{PO}^{\frak b})$ 
which we expect to be the version of residue pairing in our situation.
\par 
\begin{Definition}
\label{Def:Frobeniusalg}
Let $C$ be a $\Z_2$ graded finitely generated free $\Lambda$ module.
A structure of {\it unital Frobenius algebra} of dimension $n$ is
$\langle \cdot,\cdot \rangle : C^k \otimes C^{n-k} \to \Lambda$,
$\cup :  C^k \otimes C^{\ell} \to  C^{k+\ell}$,
$1 \in C^0$, such that:
\begin{enumerate}
\item $\langle \cdot,\cdot\rangle$ is a graded symmetric bilinear form which induces an
isomorphism $x \mapsto (y \mapsto \langle x,y\rangle)$,
$C^k \to Hom_{\Lambda_0}(C^{n-k},\Lambda)$.
\item $\cup$ is an associative product on $C$. $1$ is its unit.
\item  $\langle x\cup y,z\rangle = \langle x,y\cup z\rangle$.
\end{enumerate}
\end{Definition}
The cohomology group of an oriented closed manifold becomes a unital Frobenius algebra in an obvious way.
\begin{defn}\label{invariantZ}
Let $(C,\langle \cdot,\cdot\rangle,\cup,1)$ be a unital Frobenius algebra.
We take a basis $\text{\bf e}_I$, $I \in \frak I$ of $C$ such that $\text{\bf e}_0$ is the unit.
Let $g_{IJ} = \langle \text{\bf e}_I,\text{\bf e}_J\rangle$ and let
$g^{IJ}$ be its inverse matrix. We define an invariant of $C$ by
\begin{equation}\label{defnformulaZ}
\aligned
Z(C) = \sum_{I_1,I_2,I_3 \in \frak I}\sum_{J_1,J_2,J_3 \in \frak I}
&(-1)^*g^{I_1J_1}g^{I_2J_2}g^{I_30}g^{J_30} \\
&\langle \text{\bf e}_{I_1} \cup \text{\bf e}_{I_2},\text{\bf e}_{I_3}\rangle
\langle \text{\bf e}_{J_1} \cup \text{\bf e}_{J_2},\text{\bf e}_{J_3}\rangle
\endaligned
\end{equation}
where $* =  \deg \text{\bf e}_{I_1}
\deg \text{\bf e}_{J_2}
+ \frac{n(n-1)}{2}$.
We call $Z(C)$ the \emph{trace} of unital Frobenius algebra $C$.
\end{defn}
It is straightforward to check that $Z(C)$ is independent of the choice of the basis.
This invariant is an example of 1-loop partition function and
can be described by the following Feynman diagram.
\par\medskip
\hskip3.5cm
\epsfbox{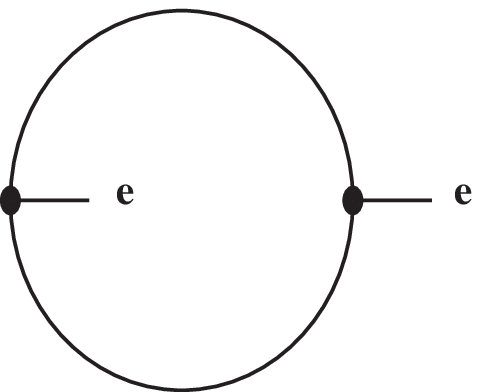}
\par\smallskip
\centerline{\bf Figure 12.1}
\par\bigskip
Let us consier $\text{\bf u} \in \text{\rm Int}\, P$ and 
$b \in H^1(L(\text{\bf u});\Lambda_0)$ such that 
the Floer cohomology 
$HF((L(\text{\bf u}),(\frak b,b)),(L(\text{\bf u}),(\frak b,b));\Lambda)$ is isomorphic 
to  $H(T^n;\Lambda)$.
\par
We have a binary operator $\frak m_2^{\frak c,\frak b,b}$ on it. The Poincar\'e duality 
induces  a $\Lambda$ valued non-degenerate inner product 
$\langle \cdot \rangle_{\text{\rm PD}_{L(\text{\bf u})}}$ of it.
\par
We define
\begin{eqnarray}
x \cup^{\frak c,\frak b,b} y &=& (-1)^{\deg x(\deg y + 1)} \frak m^{\frak c,\frak b,b}_2(x,y), \label{deformcup}\\
\langle x, y\rangle_{\text{cyc}} &=& (-1)^{\deg x(\deg y + 1)}\langle x, y\rangle_{\text{\rm PD}_{L(u)}}.
\label{PDandCYC}
\end{eqnarray}
Then 
$(H(L(\text{\bf u});\Lambda),\langle \cdot, \cdot\rangle_{{\rm cyc}},
\cup^{\frak c,\frak b,b},\text{\rm PD}[L(\text{\bf u})])$ becomes a unital Frobenius algebra.
\begin{Remark}
We remark that the operation $\frak m_2^{\frak c,\frak b,b}$ 
is slightly different from the operation $\frak m_2^{\frak b,b}$ which is 
obtained from the operation $\frak q_{\ell,k}$ 
by (\ref{mbulkdeformed}).
In fact  $\frak q_{\ell,k}$ may not satisfy the cyclic symmetry:
\begin{equation}\label{formulacyclic}
\aligned
&\langle\frak q_{\ell;k}(\text{\bf y};h_1,\dots,h_k),h_0\rangle_{\rm{cyc}}
\\
&=
(-1)^{\deg' h_0(\deg' h_1 + \dots + \deg' h_k)}
\langle\frak q_{\ell;k}(\text{\bf y};h_0,h_1,\dots,h_{k-1}),h_k\rangle_{\rm{cyc}}.
\endaligned
\end{equation}
This is because the way how we perturb the moduli space $\mathcal M_{k+1;\ell}^{\text{\rm main}}(\beta)$, 
which we described in sections \ref{Floertheory} and 
\ref{operatorq}, breaks cyclic symmetry.
\par
However we can modify the construction of $\frak q_{\ell;k}$ to obtain 
$\frak q_{\ell;k}^{\frak c}$ for which (\ref{formulacyclic}) is satisfied.
Using it in place of $\frak q_{\ell;k}$ we define 
$\frak m^{\frak c,\frak b,b}_2$, which appears in (\ref{deformcup}).
Then Definition \ref{Def:Frobeniusalg} 3) is satisfied for $\cup^{\frak b,b}$.
\par
This point is quite technical and delicate. 
So we do not discuss its detail in this survey and refer readers to
\cite{toric3} sections 18-19. 
However it is inevitable and essential, especially in the 
non-Fano case. It might be related to the fact that primitive form 
may be different from $dx_1\wedge \dots \wedge dx_n$ in general. 
\end{Remark}
We put
\begin{equation}\label{Zbbeq}
Z(\frak b, b) =
Z((H(L(u);\Lambda),\langle \cdot, \cdot\rangle_{\text{cyc}},
\cup^{\frak c,\frak b,b},\text{\rm PD}([L(u)])).
\end{equation}
\begin{Definition}\label{res2}
Assume that $\frak{PO}^{\frak b}$ is a Morse function. We then
define a {\it residue pairing}
$$
\langle \cdot,\cdot \rangle_{\text{\rm res}}:
(\text{\rm Jac}(\frak{PO}^{\frak b})
\otimes_{\Lambda_0}\Lambda) \otimes (\text{\rm Jac}(\frak{PO}^{\frak b})
\otimes_{\Lambda_0}\Lambda)
\to \Lambda
$$
by
\begin{equation}\label{rpdefine}
\langle 1_{\frak y},1_{\frak y'} \rangle_{\text{\rm res}}
=
\begin{cases} 0  &\text{if $\frak y \ne \frak y'$,} \\
\left(Z(\frak b, b)\right)^{-1} &\text{if $\frak y = \frak y'$.}
\end{cases}
\end{equation}
\end{Definition}
We remark that we use the decomposition (\ref{localization}) 
and $1_{\frak y}$ is the unit of $\text{\rm Jac}(\frak{PO}^{\frak b};\frak y)$.
$\text{\bf u} = (u_1,\dots,u_n)$ is defined by the valuation of $\frak y
= (\frak y_1,\dots,\frak y_n)$. 
Namely  $u_i = \frak v_T(\frak y_i)$.
$b\in H^1(L(\text{\bf u});\Lambda_0)$ is defined from $\frak y_i$ by
$b = \sum_{i=1}^n \frak x_i\text{\bf e}_i$, $T^{u_i}e^{\frak x_i} = \frak y_i$.
\par
The name `residue pairing' is justified by the following Theorem \ref{cliffordZ}
and Lemma \ref{Hessdetres}.
\begin{Theorem}\label{cliffordZ}
\begin{enumerate}
\item 
Assume that  $\frak y$ is a nondegenerate critical point of $\frak {PO}^{\frak b}$. 
Suppose $b = \sum_{i=1}^n \frak x_i\text{\bf e}_i$, $T^{u_i}e^{\frak x_i} = \frak y_i$ 
as above.
Then
\begin{equation}\label{res0reseq}
\aligned
Z(\frak b, b)
\equiv \det \left[y_i y_j \frac{\partial^2\frak{PO}^{\frak b}}{\partial y_i\partial y_j}
\right]_{i,j=1}^{i,j=n} (\frak y) 
\mod T^{\lambda}\Lambda_+.
\endaligned
\end{equation}
Here $\lambda = v_T(Z(\frak b, b))$ and
$\frak y = (e^{x_1},\dots,e^{x_n})$.
\item If $\dim_{\C} X = 2$, then
we have
\begin{equation}\label{res0res}
Z(\frak b, b) = \det \left[y_iy_j\frac{\partial^2\frak{PO}^{\frak c,\frak b}}{\partial y_i\partial y_j}
\right]_{i,j=1}^{i,j=n} (\frak y).
\end{equation}
\item If $X$ is nef and $\text{\rm deg}\frak b =2$, then
we have
\begin{equation}\label{res1res}
Z(\frak b, b) = \det \left[y_i  y_j\frac{\partial^2\frak{PO}^{\frak b}}{\partial y_i\partial y_j}
\right]_{i,j=1}^{i,j=n} (\frak y).
\end{equation}
\end{enumerate}
\end{Theorem}
\begin{Remark}\label{cyclicsymremark}
We use $\frak m_{k}^{\frak c,\frak b}$ in place of $\frak m_k^{\frak b}$ to define 
$\frak{PO}^{\frak c,\frak b}$ by
$$
\frak{PO}^{\frak c,\frak b}(b)
= \sum_{k=0}^{\infty} \int_{L(\text{\bf u})}\frak m_{k}^{\frak c,\frak b}(\underbrace{b,\dots,b}_k).
$$
$\frak{PO}^{\frak c,\frak b}$ appears in (\ref{res0res}).
\end{Remark}
Theorem \ref{cliffordZ} is Theorem 2.24 \cite{toric3}.
\begin{proof}[Sketch of the proof]
We discuss only the case $X$ is nef and $\frak b =\text{\bf 0}$.
We will prove that the algebra 
$(H(L(u);\Lambda),
\cup^{b})$ is a Clifford algebra, modifying the proof of a related result by Cho 
\cite{Cho05II}.
More precisely we prove the following Proposition \ref{JacisClifford}.
\par
Let $\text{\bf e}'_1,\dots,\text{\bf e}'_n$ be formal variables and $d_i \in \Lambda \setminus \{0\}$
($i=1,\dots,n$). We consider relations
\begin{equation}\label{clifrel}
\left\{
\aligned
\text{\bf e}'_i\text{\bf e}'_j+\text{\bf e}'_j\text{\bf e}'_i&=0,\qquad \qquad i \ne j\\
\text{\bf e}'_i \text{\bf e}'_i&=d_i 1.
\endaligned
\right.
\end{equation}
We take a free (non-commutative) $\Lambda$ algebra generated by
$\text{\bf e}'_1,\dots,\text{\bf e}'_n$ and divide it by the two-sided ideal generated
by (\ref{clifrel}). We denote it by
$\text{\rm Cliff}_\Lambda(n;\vec d)$, where we set $\vec d = (d_1,\dots,d_n)$.
\par
Let $I = (i_1,\dots,i_k)$, $1 \le i_1<\dots<i_k \le n$.
We write the set of such $I$'s by $2^{\{1,\dots,n\}}$.
We put
$$
\text{\bf e}'_I = \text{\bf e}'_{i_1} \text{\bf e}'_{i_2} \cdots \text{\bf e}'_{i_{k-1}} 
\text{\bf e}'_{i_k} \in
\text{\rm Cliff}_\Lambda(n;\vec d).
$$
It is well-known and can be easily checked that
$\{\text{\bf e}'_I \mid I \in 2^{\{1,\dots,n\}}\}$
forms a basis of $\text{\rm Cliff}_\Lambda(n;\vec d)$ as a
$\Lambda$ vector space.
\par
Assume moreover that there exists a $\Lambda$ valued non-degenerate 
inner product $\langle \cdot \rangle$ on $\text{\rm Cliff}_\Lambda(n;\vec d)$
such that $\text{\rm Cliff}_\Lambda(n;\vec d)$ becomes a Frobenius algebra.
We say that $\text{\bf e}'_i$ forms a {\it cyclic Clifford basis} if 
\begin{equation}\label{cyclicclifrel}
\langle \text{\bf e}'_I, \text{\bf e}'_J \rangle
=
\begin{cases}
(-1)^{*(I)}& J = I^c,\\
0& \text{\rm otherwise}.
\end{cases}
\end{equation}
Here $I^c = \{1,\dots,n\} \setminus I$ and
$*(I) = \#\{(i,j) \mid i\in I, j\in I^c, j< i\}$. 
\begin{Proposition}\label{JacisClifford}
Suppose $X$ is nef and $\deg\frak b = 2$. We also assume that $L(\text{\bf u})$ and
$b \in H^1(L(\text{\bf u});\Lambda_0)$ 
satisfy 
$
HF((L(\text{\bf u}),(\frak b,b)),(L(\text{\bf u}),(\frak b,b));\Lambda) \cong H(T^n;\Lambda).
$
\par
Then there exists a basis $(\text{\bf e}'_1,\dots,\text{\bf e}'_n)$ 
of $H^1(L(\text{\bf u});\Lambda)$  such that  the algebra
$((H(L(u);\Lambda),\cup^{\frak c,\frak b,b})$ is isomorphic to the 
Clifford algebra $\text{\rm Cliff}_\Lambda(n;\vec d)$
where $(d_1,\dots,d_n)$ are the set of eigenvalues (counted with 
multiplicity) of the Hessian matrix 
$$
\text{\rm Hess}_{\frak y}(\frak{PO}^{\frak b}) 
= \left[y_i  y_j\frac{\partial^2\frak{PO}^{\frak b}}{\partial y_i\partial y_j}
\right]_{i,j=1}^{i,j=n} (\frak y).
$$
Moreover $(\text{\bf e}'_1,\dots,\text{\bf e}'_n)$ is a cyclic Clifford basis.
\par
Furthermore
$$
\int_{L(\text{\bf u})}\text{\bf e}'_1 \cup^{\frak b,b}\dots\cup^{\frak b,b}\text{\bf e}'_n
= 1.
$$
\end{Proposition}
This is \cite{toric3} Theorem 22.2.
Once Proposition \ref{JacisClifford} is established we can prove
Theorem \ref{cliffordZ} by a direct calculation.
(See \cite{toric3} section 23.)
\end{proof}
\begin{proof}[Sketch of the proof of Proposition \ref{JacisClifford}]
Note
$$
\frak{PO}^{\frak b}(b) =  
 \sum_{k=0}^{\infty} \int_{L(\text{\bf u})}\frak m_{k}^{\frak b}(\underbrace{b,\dots,b}_k).
$$
Its first derivative at $\frak y$ is zero since $\frak y$ is a critical point.
We calculate its second derivative $\partial^2\frak{PO}^{\frak b}/\partial x_i\partial x_j 
= y_iy_j\partial^2\frak{PO}^{\frak b}/\partial y_i\partial y_j$.
Then we have
\begin{equation}\label{m2andsecondder}
\frak m_2^{\frak b,b}(\text{\bf e}_i,\text{\bf e}_j) 
+ \frak m_2^{\frak b,b}(\text{\bf e}_j,\text{\bf e}_i) 
= 
\left(\left(y_i  y_j\frac{\partial^2\frak{PO}^{\frak b}}{\partial y_i\partial y_j}
\right)(\frak y)\right) 1
\end{equation}
Here $1\in H^0(L(\text{\bf u});\Q)$ is the unit
and $\{\text{\bf e}_i\}$ is the basis of $H^1(L(\text{\bf u});\Q)$ which 
we fixed before. (Note $b = \sum x_i\text{\bf e}_i$.)
\par
We take basis $(\text{\bf e}'_1,\dots,\text{\bf e}'_n)$ of $H^1(L(\text{\bf u});\Lambda)$ 
so that the Hessian matrix becomes the diagonal matrix and
$
\int_{L(\text{\bf u})}\text{\bf e}'_1 \cup \dots\cup\text{\bf e}'_n
= 1.
$
Then (\ref{m2andsecondder}) implies that $(\text{\bf e}'_1,\dots,\text{\bf e}'_n)$ 
satisfies the Clifford relation (\ref{clifrel}).
Using this fact we can prove that $((H(L(u);\Lambda),\cup^{\frak c,\frak b,b})$ 
is a Clifford algebra.
(We do not use the assumption $X$ is nef and $\frak b$ is degree two,
up to this point.) 
\par
The proof of (\ref{cyclicclifrel}) is as follow.
We use the assumption that $X$ is nef and $\frak b$ is degree two 
to show  
\begin{equation}\label{degtari}
\frak a \cup^{\frak c,\frak b,b}  \frak a'  - \frak a \cup  \frak a' \in \bigoplus_{k < \deg \frak a+ \deg\frak a'}H^k(T^n;\Lambda)
\end{equation}
for $\frak a, \frak a' \in HF((L(\text{\bf u}),(\frak b,b)),(L(\text{\bf u}),(\frak b,b));\Lambda) \cong H(T^n;\Lambda)$.
Here the second term is the usual cup product.
We use cyclic symmetry to show
$$
\langle \text{\bf e}'_I,\text{\bf e}'_J\rangle_{\text{\rm PD}_{L(\text{\bf u})}} 
= \langle \text{\bf e}'_I \cup^{\frak c,\frak b,b}\text{\bf e}'_J,1\rangle_{\text{\rm PD}_{L(\text{\bf u})}}
=  \int_{L(\text{\bf u})}\text{\bf e}'_I \cup^{\frak c,\frak b,b}\text{\bf e}'_J.
$$
Using (\ref{degtari}) and Clifford relation, we can see that 
 $ \text{\bf e}'_I \cup^{\frak c,\frak b,b}\text{\bf e}'_J$ has no $H^n(L(\text{\bf u});\Lambda)$
 component unless $I^c = J$. This implies Proposition \ref{JacisClifford}.
\end{proof}

\par\bigskip

\subsection{Residue pairing is Poincar\'e duality.}
\label{subseqPD}
\begin{Theorem}\label{PDresidue}
Let $X$ be a compact toric manifold and $\frak b \in \mathcal A(\Lambda_0)$.
Suppose $\frak{PO}^{\frak b}$ is a Morse function.
Then for each $\frak a_1,\frak a_2 \in H(X;\Lambda)$ we have
\begin{equation}
\langle\frak a_1,\frak a_2 \rangle_{\text{\rm PD}_X}
= 
\langle\frak{ks}_{\frak b}\frak a_1,\frak{ks}_{\frak b}\frak a_2\rangle_{\rm res}.
\end{equation}
Here the pairing in the right hand side is defined in 
Definition \ref{res2} and the map $\frak{ks}_{\frak b}$ 
is the isomorphism in Theorem \ref{indepenceKS}.
The pairing in the left hand side is the Poincar\'e duality.
\end{Theorem}
Theorem \ref{PDresidue} is \cite{toric3} Theorem 1.1 (2) and 
is proved in \cite{toric3} sections 17-21.
Before explaining an outline of its proof, 
we mention some of its consequences.
\begin{Corollary}\label{novikovLGForb}
\begin{enumerate}
\item The inner product
$\langle\cdot\rangle_{\rm res}$, 
whose definition was given
only in case 
$\frak{PO}^{\frak b}$  is a Morse function
(in Definition \ref{res2}), extends to arbitrary $\frak b$'s.
\item The Levi-Civita connection $\nabla$ of this extended 
$\langle\cdot\rangle_{\rm res}$  is flat.
\item $(H(X;\Lambda_0),\langle\cdot\rangle_{\rm res},\nabla,\circ, \Phi,1)$ 
is a Frobenius manifold.
\item
The Frobenius manifold structure of Item $3)$ above is 
equal to one in Theorem \ref{FrobDub}.
\end{enumerate}
\end{Corollary}
\begin{proof}
1) is an immediate consequence of Theorem \ref{PDresidue} 
and the fact that the Poincar\'e duality pairing 
is independent of $\frak b$ and is obviously extended.
\par
The Levi-Civita connection of the Poincar\'e duality pairing 
is the canonical affine connection of $H(X;\Lambda_0)$ 
and hence is flat. 2) follows.
\par
3) then follows from Theorem \ref{FrobDub}.
\par
4) is obvious.
\end{proof}
\begin{Remark}
The Frobenius manifold in Corollary \ref{novikovLGForb} 3)  has an Euler vector field 
(\ref{eulerquatumcoh}) with $r_i =1$. We also have
\begin{equation}\label{EandPO}
\frak E(\frak{PO}) = \frak{PO},
\end{equation}
here $\frak{PO}$ is a function of $\frak b = \sum w_i \text{\bf p}_i$ 
and $y_i$.  The formula (\ref{EandPO}) is proved in 
\cite{toric2} Theorem 10.2.
\end{Remark}
\begin{Remark}
Corollary \ref{novikovLGForb} first appeared as a conjecture in \cite{taka},
where the case of $\C P^1$ was checked. 
It was further studied in \cite{barani}.
See the papers mentioned at the end of the introduction for some of the other related works.
\end{Remark}

The above proof of the coincidense of the two Frobenius manifold 
structures is not so satisfactory since the proof of 
Items 1), 2) uses the isomorphism of Item 4). 
It is preferable that we construct Frobenius manifold structure 
on $H(X;\Lambda_0)$ 
using the family of functions $\frak{PO}^{\frak b}$ and without going to the quantum cohomology theory side, and then 
prove Item 4) for that Frobenius manifold structure.

\begin{Problem}
Develop an analogue of K. Saito theory for our family of $\Lambda$ valued functions 
$\frak{PO}^{\frak b}$.
\par
Define the notion of primitive form for it and prove its existence.
\par
Construct the Frobenius manifold structure on $H(X;\Lambda_0)$ 
using primitive form and prove that it is isomorphic to one 
obtained in Theorem  \ref{FrobDub}.
\end{Problem}
Another corollary of Theorem \ref{PDresidue} is the following.
Let $\text{\rm Crit}(\frak{PO}^{\frak b})$ be the critical point set of 
$\frak{PO}^{\frak b}$. For 
$\frak y = (\frak y_1,\dots,\frak y_n) \in \frak A (P)$ 
we put 
\begin{equation}\label{ytobu}
\frak y_i = T^{u_i} e^{\frak x_i},
\qquad b = \sum_{i=1}^n \frak x_i\text{\bf e}_i \in H^1(L(\text{\bf u}),\Lambda_0).
\end{equation}
Here $\text{\bf u} = (u_1,\dots,u_n) \in P$ and $\frak x_i \in \Lambda_0$.
Note $u_i = \frak v_T(\frak y_i)$.
In this way we may regard $\text{\rm Crit}(\frak{PO}^{\frak b})$ 
as a set of pairs $(\text{\bf u}_c,b_c)$, $c=1,\dots,B$.
Here we put
$B= \#\text{\rm Crit}(\frak{PO}^{\frak b})$.
\begin{Corollary}\label{Corresi}
Suppose $\frak{PO}^{\frak b}$ is a Morse function. 
Then we have
\begin{equation}\label{hennasiki}
0 = \sum_{c=1}^{B}
\frac{1}{Z(\frak b,b_c)}.
\end{equation}
\end{Corollary}
\begin{proof}
Let $1_X \in H^0(X;\Lambda)$ be the unit. Then
$
\langle 1_X,1_X \rangle_{\text{\rm PD}_X} = 0.
$
By Proposition \ref{localization} we have
$
1_X = \sum_{\frak y \in \frak M(X,\frak b)} 1_{\frak y}
$
where $1_{\frak y}$ is the unit of the Jacobian ring $\text{\rm Jac}(\frak{PO}_{\frak b};\frak y)$.
Corollary \ref{Corresi} 
now follows from (\ref{rpdefine}) and Theorem \ref{PDresidue}.
\end{proof}
\par\bigskip
\subsection{Operator $\frak p$ and the Poincar\'e dual to $\frak{ks}_{\frak b}$.}
\label{operatorp}

In this and the next subsections we sketch a proof of 
Theorem \ref{PDresidue}.
We assume $\frak{PO}^{\frak b}$ is a Morse function in this and 
next subsections.
Let $\frak y \in \text{\rm Crit}(\frak{PO}^{\frak b})$.
It defines $\text{\bf u},b$ by  (\ref{ytobu}).
We define a homomorphism
\begin{equation}
i^{\#}_{\text{\rm qm},(\frak b,b,\text{\bf u})}: H(X;\Lambda_0)
\to HF((L(\text{\bf u}),(\frak b,b));(L(\text{\bf u}),(\frak b,b));\Lambda_0).
\end{equation}
by
\begin{equation}
i^{\#}_{\text{\rm qm},(\frak b,b,\text{\bf u})}(Q)= 
\sum_{k=0}^{\infty}\sum_{\ell_1=0}^{\infty}\sum_{\ell_2=0}^{\infty}
\frak q^{\frak c}_{\ell_1+\ell_2;k}(\frak b^{\ell_1} Q\frak b^{\ell_2},b^{k}).
\end{equation}
(See (\ref{qmap3}) and \cite{fooo-book} Theorem 3.8.62.)
\par
Here $\frak q^{\frak c}_{\ell;k}$ is a cyclically symmetric version of the 
operator $\frak q_{\ell;k}$. (See Remark \ref{cyclicsymremark}.)
\par
We define 
\begin{equation}
i_{\#,\text{\rm qm},(\frak b,b,\text{\bf u})}: 
HF((L(\text{\bf u}),\frak b,b);(L(\text{\bf u}),\frak b,b);\Lambda)
\to H(X;\Lambda)
\end{equation}
by 
\begin{equation}\label{dualize}
\langle i^{\#}_{\text{\rm qm},(\frak b,b,\text{\bf u})}(Q),P
\rangle_{\text{\text{\rm PD}}_{L(\text{\bf u})}}
=
\langle Q,i_{\#,\text{\rm qm},(\frak b,b,\text{\bf u})}(P)
\rangle_{\text{\rm PD}_{X}}.
\end{equation}
The main part of the proof of Theorem \ref{PDresidue}
is the proof of Theorem \ref{annulusmain} below.
Let ${\rm vol}_{L(\text{\bf u})} \in H^n(X;\Q)$ 
be the degree $n$ cohomology class such that
$\int_{L(\text{\bf u})} {\rm vol}_{L(\text{\bf u})} = 1$.
Let $\{\text{\bf e}_I \mid I \in 2^{n}\}$ be a basis 
of 
$$
H^n(L(\text{\bf u});\Lambda)
\cong HF((L(\text{\bf u}),(\frak b,b));(L(\text{\bf u}),(\frak b,b));\Lambda).
$$
We put
$g_{IJ} = \langle \text{\bf e}_I,\text{\bf e}_J\rangle_{PD_X}$.
Let $g^{IJ}$ be the inverse matrix of $g_{IJ}$.
\begin{Theorem}\label{annulusmain}
We have:
\begin{equation}\label{21mainformula}
\aligned
&\langle i_{\#,\text{\rm qm},(\frak b,b,\text{\bf u})}( {\rm vol}_{L(\text{\bf u})}),
i_{\#,\text{\rm qm},(\frak b,b,\text{\bf u})}( {\rm vol}_{L(\text{\bf u})})
\rangle_{\text{\rm PD}_{X}} \\
&=
\sum_{I,J \in 2^{\{1,\dots,n\}}} (-1)^{\frac{n(n-1)}{2}} g^{IJ}
\langle \frak m_2^{\frak c,\frak b,b}(\text{\bf e}_I, {\rm vol}_{L(\text{\bf u})}),
\frak m_2^{\frak c,\frak b,b}(\text{\bf e}_J, {\rm vol}_{L(\text{\bf u})})
\rangle_{\text{\rm PD}_{L(u)}}.
\endaligned
\nonumber\end{equation}
\end{Theorem}
This is \cite{toric3} Theorem 20.1.
\begin{proof}[Theorem \ref{annulusmain} $\Rightarrow$
Theorem \ref{PDresidue}]
Let $Q_{\frak y} \in H(X;\Lambda)$ be an element such that
$\frak{ks}_{\frak b}(Q_{\frak y})=
1_{\frak y}$, where $1_{\frak y}$ is the unit of the factor $\text{\rm Jac}(\frak{PO}^{\frak b};\frak y)$ of
$\text{\rm Jac}(\frak{PO}^{\frak b})$.
Let $b,\text{\bf u}$ corresponds to $\frak y$ by (\ref{ytobu}).
\par
Then 
we have
$$
i^{\#}_{\text{\rm qm},(\frak b,b,\text{\bf u})}(Q_{\frak y'})
=
\begin{cases}
1    & \text{if $\frak y'=\frak y$} \\
0   & \text{if $\frak y'\ne \frak y$.}
\end{cases}
$$
Here $1\in H^0(L(\text{\bf u});\Lambda)$ is the unit.
This is a consequence of the definition of $\frak{ks}_{\frak b}$.
Therefore 
\begin{equation}\label{kocchikite1}
\langle Q_{\frak y},i_{\#,\text{\rm qm},(\frak b,b,\text{\bf u})}(\text{vol}_{L(\text{\bf u})})\rangle_{\text{\rm PD}_{X}}
= 1.
\end{equation}
We remark
$
\langle Q_{\frak y},Q_{\frak y'} \rangle 
= 
\langle Q_{\frak y}\cup Q_{\frak y'}, 1 \rangle
= 0
$
if $\frak y\ne \frak y'$.
Therefore
\begin{equation}\label{kocchikite2}
i_{\#,\text{\rm qm},(\frak b,b,\text{\bf u})}(\text{vol}_{L(u)})
= \frac{1}{\langle Q_{\frak y},Q_{\frak y}\rangle_{\text{\rm PD}_{X}}}Q_{\frak y}.
\end{equation}
Theorem \ref{annulusmain} implies
\begin{equation}\label{kocchikite3}
\langle i_{\#,\text{\rm qm},(\frak b,b,\text{\bf u})}(\text{\rm vol}_{L(\text{\bf u})}),
i_{\#,\text{\rm qm},(\frak b,b,\text{\bf u})}(\text{\rm vol}_{L(\text{\bf u})})
\rangle_{\text{\rm PD}_{X}}
= Z(\frak b,b).
\end{equation}
(See \cite{toric3} subsection 26.2 for sign.) 
Theorem \ref{PDresidue} follows from 
(\ref{kocchikite2}) and (\ref{kocchikite3}).
\end{proof}
\par
To prove Theorem \ref{annulusmain} we need a geometric description of the 
homomorphism $ i_{\#,\text{\rm qm},(\frak b,b,\text{\bf u})}$.
We use the operator $\frak p$ introduced in 
\cite{fooo-book} section 3.8, for this purpose.
To simplify the notation we consider only the case $\frak b = \text{\bf 0}$.
Let $C$ be a filtered $A_{\infty}$ algebra and define an
automorphism
$
\text{cyc} : B_kC[1] \to B_kC[1]
$
by
$$
\text{cyc}(x_1\otimes\cdots\otimes x_k) = (-1)^{\deg'x_k\times
(\sum_{i=1}^{k-1}\deg'x_i)}
x_k\otimes x_1\otimes\cdots\otimes x_{k-1}.
$$
It induces a $\Bbb Z_{k}$ action on $B_kC[1]$. Let
$B^{\text{cyc}}_kC[1]$ be the invariant set of the $\Bbb Z_k$ action and
$B^{\text{cyc}}C[1] = \widehat{\bigoplus}_k B^{\text{cyc}}_kC[1]$
the completed direct sum of them. 
We call $B^{\text{cyc}}_kC[1]$ the  {\it cyclic bar complex}.
\begin{Theorem} 
For a relatively spin Lagrangian submanifold $L$
there exists a sequence of operators
$$
\frak p_k : B^{\text{\rm cyc}}_kH(L;\Lambda_{0})[1]
\longrightarrow
H(X;\Lambda_{0})
$$
$(k = 0,1,2,\dots)$ of degree $n+1$
with the
following properties. 
\par
Let
$
\frak p : B^{\text{\rm cyc}}H(L;\Lambda_{0})[1] \longrightarrow
H(X;\Lambda_{0})
$
be the operator whose restriction on
$B^{\text{\rm cyc}}_kH(L;\Lambda_{0})[1]$ is $\frak p_k$. 
We denote by  $\frak m_k^{\frak c}$ the cyclically symmetric version of $\frak m_k$ 
and write $\frak m^{\frak c}$ instead of $\frak m^{\frak c}_k$.
\begin{enumerate}
\item
$$
\frak p_1 \equiv i_{!} \mod \Lambda_{+}.
$$
Here $i_{!} = H^k(L;\Lambda_0) \to H^{k+n}(X;\Lambda)$ 
is the Gysin homomorphism.
\item
\begin{equation}\label{pmainformula}
\sum_c \frak p(\text{\bf x}_c^{3;1}  \otimes \frak m^{\frak c}(\text{\bf x}_c^{3;2}) 
 \otimes \text{\bf x}_c^{3;3}) = 0
\end{equation}
for $\text{\bf x} \in B^{\text{\rm cyc}}_kH(L;\Lambda_{0})[1]$, $k>0$.
We use the notation $(\ref{Deltasymbol})$.
\item
$$
(\frak p_1\circ \frak m^{\frak c}_0)(1)
+ GW_{1}(L) = 0.
$$
Here the second term is defined by
$
\langle GW_{1}(L), Q\rangle_{\text{\rm PD}_X} = GW_{2}(L,Q),
$
where the right hand side is as in $(\ref{sumGW})$.
\end{enumerate}
\end{Theorem}
This is \cite{fooo-book} Theorem 3.8.9.
(Here we use cohomology group instead of appropriate chain complex.
The latter is used in \cite{fooo-book} Theorem 3.8.9.
We also omit the statement on the unit in \cite{fooo-book} Theorem 3.8.9.)
See also \cite{toric3} section 17-19.
\par
The operator $\frak p_k$ is constructed as follows.
We consider the moduli space $\mathcal M^{\text{\rm main}}_{k;1}(\beta)$ 
described in section 2.2. Note the number of interior marked point is $1$ 
and the number of exterior marked points is $k$.
We have an evaluation map
$$
({\rm ev}_1,\dots,{\rm ev}_k,{\rm ev}^{+}) = ({\rm ev},{\rm ev}^+) : \mathcal M^{\text{\rm main}}_{k;1}(\beta)
\to L^k \times X.
$$
Let $h_1,\dots,h_k$ be differential forms on $L$.
We consider the pull back
${\rm ev}^*(h_1\times \dots \times h_k)$, which is a
differential form on $\mathcal M^{\text{\rm main}}_{k;1}(\beta)$.
We use integration along fiber by the map ${\rm ev}^+$ to obtain a differential form 
on $X$, which we put $\frak p_{k,\beta}(h_1,\dots,h_k)$. Namely
$$
\frak p_{k,\beta}(h_1,\dots,h_k) 
= {\rm ev}^+_!({\rm
 ev}^*(h_1\times \dots \times h_k)).
$$
This is a map between differential forms. 
By an algebraic argument it induces a map between tensor 
products of the de Rham cohomology 
groups of $L$ and of $X$. Thus obtain the operator
$$
\frak p_{k} = \sum_{\beta \in H_2(X,L)} T^{(\beta\cap\omega)/2\pi}\frak p_{k,\beta}.
$$
We can prove (\ref{pmainformula}) by studying the stable map 
compactification of 
$\mathcal M^{\text{\rm main}}_{k;1}(\beta)$. In case $k=0$ the compactification 
of $\mathcal M_{0;1}(\beta)$ is slightly different from the case of $k>0$.
The second term of Item 3) appears by this reason.
In our case of toric manifold and $T^n $ orbit $L$, 
this term drops since $L$ is homologous to $0$ in $X$.
So we do not discuss it here but refer to \cite{fooo-book} subsections 3.8.3 and 
7.4.1 for more detail.
\par\smallskip
Now we go back to the case where $X$ is a toric manifold and $L = L(\text{\bf u})$
is a $T^n$ orbit.
Let $b \in H^1(L(\text{\bf u});\Lambda_0)$.
For $P \in H(L(\text{\bf u});\Lambda_0)$ we put
$$
[Pe^b] = \sum_{k_1=0}^{\infty}\sum_{k_2=0}^{\infty} 
\underbrace{b \otimes \dots \otimes b}_{k_1}  \otimes P
\otimes \underbrace{b \otimes \dots \otimes b}_{k_2}.
$$
\par
Suppose 
$H(L(\text{\bf u});\Lambda)
\cong HF((L(\text{\bf u}),(\text{\bf 0},b));(L(\text{\bf u}),(\text{\bf 0},b));\Lambda)$.
\begin{Proposition}
Let $P \in H(L(\text{\bf u});\Lambda_0)$, $Q \in H(X;\Lambda_0)$.
Then we have:
\begin{equation}\label{defisitsharp2}
i_{\#,\text{\rm qm},(\text{\bf 0},b,\text{\bf u})}(P)
= \frak p([Pe^b]).
\end{equation}
\end{Proposition}
\begin{Remark}
We remark that $[Pe^b]$ is an element of $B^{\rm cyc}H(L(\text{\bf u});\Lambda_0)$ 
if $b \equiv 0 \mod \Lambda_+$.  So $\frak p([Pe^b])$ is defined in that case.
Otherwise we write $b = b_0 + b_+$ such that $b_0 \in H^1(L(\text{\bf u});\C)$ 
and $b_+ \in H^1(L(\text{\bf u});\Lambda_+)$, 
and define
$$
\frak p([Pe^b]) = \sum_{\beta\in H_2(X,L:\Z)} T^{(\beta\cap\omega)/2\pi}
\exp(b_0 \cap \partial \beta) \frak p_{\beta}([Pe^{b_{+}}]).
$$
We omit the discussion of this point. See \cite{toric2} section 9
and \cite{toric3} section 19.
\end{Remark}
\begin{proof}[Sketch of the proof]
We remark that $i_{\#,\text{\rm qm},(0,b,\text{\bf u})}(P)$ 
is defined by $(\ref{dualize})$.  Therefore 
it suffices to prove
\begin{equation}\label{pqdual}
\sum_{k=0}^\infty\langle \frak q^{\frak c}_{1,k}(Q;b^{k}),P\rangle_{\text{\rm PD}_{L(\text{\bf u})}}
= 
\langle Q,\frak p([Pe^b]\rangle_{\text{\rm PD}_{X}}.
\end{equation}
This is \cite{toric3} Theorem 19.8.
Let us sketch its proof for the case $b=0$.
In case $b=0$, Formula (\ref{pqdual})  reduced to
\begin{equation}\label{pqdual2}
\langle \frak q^{\frak c}_{1,0}(Q;1),P\rangle_{\text{\rm PD}_{L(\text{\bf u})}}
= 
\langle Q,\frak p_1(P)\rangle_{\text{\rm PD}_{X}}.
\end{equation}
We take $\rho$ and $h$ which are closed forms on $X$ and $L(\text{\bf u})$, 
representing the cohomology class $Q$ and $P$, respectively.
Then it is easy to see that the left and the right hand sides of 
(\ref{pqdual2}) both become
\begin{equation}\label{pqdual3}
\sum_{\beta\in H_2(X,L(\text{\bf u});\Z)}T^{(\beta\cap\omega)/2\pi}\int_{\mathcal M_{1;1}(\beta)}
({\rm ev}^+)^*\rho \wedge {\rm ev}^* h.
\end{equation}
Here $({\rm ev},{\rm ev}^+) : \mathcal M_{1;1}(\beta) \to L(\text{\bf u}) \times X$  is evaluation maps 
at marked points. (\ref{pqdual2}) follows.
\end{proof}
\begin{Remark}\label{1224}
In fact, we need to perturb $\mathcal M_{1;1}(\beta)$ appropriately 
so that the integration in (\ref{pqdual3}) makes sense.
It is a nontrivial thing to prove that after perturbation 
(\ref{pqdual2}) still holds. Actually we need to consider cyclically symmetric 
version of the operator $\frak q$ for this purpose.
(See \cite{toric3} Remark 19.12.)
We omit the discussion about perturbation and refer the reader to \cite{toric3} section 19.
\end{Remark}
\par\bigskip
\subsection{Annulus argument.}
\label{annulus}
We continue the sketch of the proof of Theorem 
\ref{annulusmain}.
We assume $\frak b = \text{\bf 0}$ in this subsection for simplicity.
We consider the class $\text{\rm vol}_{L(\text{\bf u})}$. (It is the Poincar\'e dual to the 
point class.) Then the left hand side is 

\begin{equation}\label{129}
\sum_{\beta_1,\beta_2 \in H_2(X,L(\text{\bf u});\Z), \atop \beta=\beta_1+\beta_2}
T^{((\beta_1+\beta_2)\cap \omega)/2\pi}
\left\langle \frak p_{\beta_1}([\text{\rm vol}_{L(\text{\bf u})}e^b],\frak p_{\beta_2}([\text{\rm vol}_{L(\text{\bf u})}e^b]\right\rangle_{\text{\rm PD}_{X}}.
\end{equation}
We show that (\ref{129}) can be regarded as an appropriate 
integration of the differential 
form $\text{\rm vol}_{L(\text{\bf u})} \times \text{\rm vol}_{L(\text{\bf u})}$ on a moduli space of 
pseudo-holomorphic annuli, as follows.
For simplicity we assume $b=0$.
\par
We consider a pair $((\Sigma;z_1,z_2),u)$ with the following properties.
\begin{enumerate}
\item $\Sigma$ is a bordered curve of genus zero such that 
$\partial \Sigma$ is a disjoint union of two circles, which we denote by $\partial_1 \Sigma$,
$\partial_2 \Sigma$.
\item The singularity of $\Sigma$ is at worst the interior double point.
\item
$z_i \in \partial_i \Sigma$ for $i=1,2$.
\item $u : \Sigma \to X$ is a pseudo-holomorphic map.
$u(\partial \Sigma) \subset L(\text{\bf u})$.
\item 
$u_*([\Sigma]) = \beta \in H_2(X, L(\text{\bf u});\Z)$.
\item 
The set of maps $v : \Sigma \to \Sigma$ which is biholomorphic, 
$v(z_i) =z_i$ for $i=1,2$,  
and $u \circ v = u$ is finite.
\end{enumerate}
We denote by $\mathcal M_{(1,1);0}(\beta)$ the totality of such 
$((\Sigma;z_1^+,z_2^+),u)$.
There exists an evaluation map
$$
{\rm ev} = ({\rm ev}_1,{\rm ev}_2) : \mathcal M_{(1,1);0}(\beta) 
\to L(\text{\bf u})^2,
$$
which is defined by
$$
{\rm ev}((\Sigma;z_1,z_2),u) = (u(z_1),u(z_2)).
$$
\par
We consider the set of all $(\Sigma;z_1,z_2)$ which satisfies 
1), 2), 3) above and 
\par\smallskip
\hskip0.3cm 7)  The set of all biholomorphic maps $v : \Sigma \to \Sigma$
with $v(z_i) =z_i$ for $i=1,2$ is finite.
\par\smallskip
We denote it by $\mathcal M_{(1,1);0}$.
There is a forgetful map
\begin{equation}\label{forgettoRSM}
\frak{forget} : \mathcal M_{(1,1);0}(\beta)  \to \mathcal M_{(1,1);0},
\end{equation}
which is obtained by forgetting the map $u$.
\par
We can show that $\mathcal M_{(1,1);0}$ is homeomorphic to a disk 
and so is connected.
We take two points $(\Sigma^{(j)};z_1^{(j)},z_2^{(j)}) \in \mathcal M_{(1,1);0}$
($j=1,2$) which we show in the figure below.
\par\medskip
\hskip1.3cm
\epsfbox{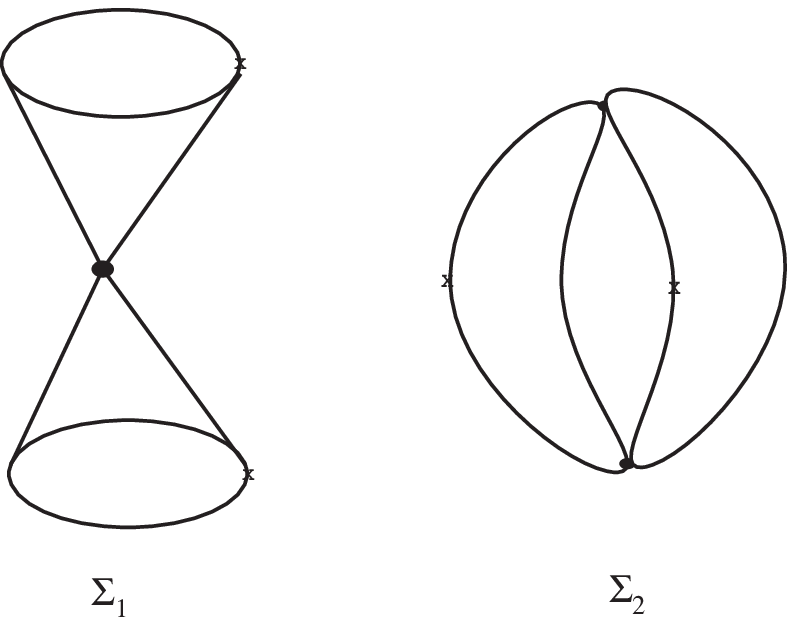}
\par\smallskip
\centerline{\bf Figure 12.2}
\par\bigskip
We denote by
$\mathcal M_{(1,1);0}(\beta;\Sigma^{(j)})$ the inverse image of $\{(\Sigma^{(j)};z_1^{(j)},z_2^{(j)})\}$ 
by the map (\ref{forgettoRSM}).
\begin{Lemma}\label{inverseimage1}
$$
\sum_{\beta_1,\beta_2 \in H_2(X,L(\text{\bf u});\Z), \atop \beta=\beta_1+\beta_2}
\langle \frak p_{1,\beta_1}(P),\frak p_{1,\beta_2}(P)\rangle_{\text{\rm PD}_{X}}
= 
\int_{\mathcal M_{(1,1);0}(\beta;\Sigma^{(1)})} {\rm ev}_1^* \text{\rm vol}_{L(\text{\bf u})}
\wedge {\rm ev}_2^* \text{\rm vol}_{L(\text{\bf u})}.
$$
\end{Lemma}
Geometric origin of this lemma is clear from Figure 12.2.
To prove the lemma rigorously we need to work out the way 
to perturb our moduli space $\mathcal M_{(1,1);0}(\beta;\Sigma^{(1)})$ 
so that the integration of the right hand side 
makes sense and the lemma holds. 
The detail is given in \cite{toric3} section 20
as the proof of  Lemma 20.8.
\par
\begin{Lemma}\label{inverseimage2}
$$
\aligned
&\sum_{I,J \in 2^{\{1,\dots,n\}}} (-1)^{\frac{n(n-1)}{2}} g^{IJ}
\langle \frak m_2^{\frak c,0,0}(\text{\bf e}_I, {\rm vol}_{L(\text{\bf u})}),
\frak m_2^{\frak c,0,0}(\text{\bf e}_J, {\rm vol}_{L(\text{\bf u})})
\rangle_{\text{\rm PD}_{L(u)}} \\
&=
\int_{\mathcal M_{(1,1);0}(\beta;\Sigma^{(2)})} {\rm ev}_1^* \text{\rm vol}_{L(\text{\bf u})}
\wedge {\rm ev}_2^* \text{\rm vol}_{L(\text{\bf u})}.
\endaligned
$$
\end{Lemma}
Geometric origin of this lemma is also clear from Figure 12.2 and the equality
\begin{equation}\label{Diagonal}
\aligned
\left[ \{(x,x) \mid x \in L(\text{\bf u})\} \right]
&=\sum_{I,J} (-1)^{\deg \text{\bf e}_I  \deg \text{\bf e}_J} g^{IJ} \text{\bf e}_I \times \text{\bf e}_J 
\\
&\in H_n( L(\text{\bf u})\times  L(\text{\bf u});\Z).
\endaligned\end{equation}
The detail is given in  \cite{toric3} section 20 as
the proof of  Lemma 20.11.
(The sign in (\ref{Diagonal}) is proved in \cite{toric3} Lemma 26.7.)
\par
Now we can use the fact that $\mathcal M_{(1,1);0}$ is connected 
to find a cobordism between $\mathcal M_{(1,1);0}(\beta;\Sigma^{(1)})$ 
and $\mathcal M_{(1,1);0}(\beta;\Sigma^{(2)})$. 
The differential form 
$ {\rm ev}_1^* \text{\rm vol}_{L(\text{\bf u})}
\wedge {\rm ev}_2^* \text{\rm vol}_{L(\text{\bf u})}$ extends to this cobordism.
Therefore Lemmas \ref{inverseimage1} and \ref{inverseimage2} 
imply Theorem \ref{annulusmain} in case $\frak b = b =0$.
The general case is similar.
\qed
\begin{Remark}
According to E. Getzler, the fact 
$\mathcal M_{(1,1);0}(\beta;\Sigma^{(1)})$ 
is cobordant to $\mathcal M_{(1,1);0}(\beta;\Sigma^{(2)})$ 
is called the Cardy relation.
\end{Remark}
\begin{Remark}
A similar trick using the annulus is used in \cite{abouz, bircor} for a similar but a slightly 
different purpose.
\end{Remark}
\par\bigskip
\section{Examples 3}
\label{exa3}
\begin{Example}
We consider the case of $\C P^n$ and $\frak b = \text{\bf 0}$.
The moment polytope $P$ is a simplex 
$
\left\{(u_1,\dots,u_n) \mid 0 \le u_i, \sum u_i \le 1\right\}
$
and the potential function is
$$
\frak{PO}^{\text{\bf 0}} = \sum_{i=1}^n  y_i + T
(y_1\cdots y_n)^{-1}.
$$
The critical points are
$\frak y^{(k)} = T^{\frac{1}{n+1}} e^{\frac{2\pi\sqrt{-1}k}{n+1}}$ $k=0,\dots,n$
which are all non-degenerate.
The isomorphism 
$
\text{\rm Jac}(\frak{PO}^{\text{\bf 0}})\otimes_{\Lambda_0} \Lambda
\cong \prod_{k=0}^{n} \Lambda 1_{\frak y^{(k)}}.
$
is induced by
$$
P \mapsto \sum_{k=0}^n P(\frak y^{(k)}) 1_{\frak y^{(k)}}.
$$
We put
$
\text{\bf f}_k = \pi^{-1}(\{(u_1,\dots,u_n) \in P \mid u_i = 0,\,
i = n-k+1, \dots, n\}).
$
We derive
$$
\frak{PO}^{w\text{\bf p}_1}(y)
= \frak{PO}^{\text{\bf 0}}(y) + (e^w - 1)T^{u_n}y_n
$$
from Proposition 4.9 \cite{toric2} and hence
\begin{equation}\label{ksf1}
\frak{ks}_{\text{\bf 0}}(\text{\bf p}_1)
= [T^{u_n}y_n] = T^{\frac{1}{n+1}} \sum_{k=0}^n e^{\frac{2\pi\sqrt{-1}k}{n+1}} 1_{\frak \frak y^{(k)}}
\end{equation}
by definition of $\frak{ks}_{\text{\bf 0}}$.
Using the fact that $\frak{ks}_{\text{\bf 0}}$ is a ring homomorphism,
we have
\begin{equation}\label{ksfm}
\frak{ks}_{\text{\bf 0}}(\text{\bf p}_\ell)
= T^{\frac{\ell}{n+1}} \sum_{k=0}^n e^{\frac{2\pi\sqrt{-1}k\ell}{n+1}} 1_{\frak y^{(k)}}.
\end{equation}
Note this holds for $\ell=0$ also since $\text{\bf f}_0$ is a unit
and $\frak{ks}_{\text{\bf 0}}$ is unital.
\par
The Hessian of $\frak{PO}^{\text{\bf 0}}$
is given by
$$
\text{\rm Hess}_{\frak y^{(k)}} \frak{PO}^{\text{\bf 0}}
= \left[ T^{\frac{1}{n+1}} \frac{\partial^2}{\partial x_i\partial x_j}
\left(e^{x_1} + \dots + e^{x_n} + e^{-(x_1 + \dots + x_n)}\right)\right]_{i,j=1}^{i,j=n}
(\frak x^{(k)})
$$
with $\frak x^{(k)}= \exp\left(\frac{2\pi\sqrt{-1}k}{n+1}\right)$. Therefore
$$
\text{\rm Hess}_{\frak y^{(k)}} \frak{PO}^{\text{\bf 0}}
=  T^{\frac{1}{n+1}}e^{\frac{2\pi\sqrt{-1}k}{n+1}} \left[
\delta_{ij} + 1\right]_{i,j=1}^{i,j=n}.
$$
It is easy to see that the determinant of the matrix
$\left[
\delta_{ij} + 1\right]_{i,j=1}^{i,j=n}$ is $n+1$. Therefore
the residue pairing is given by
\begin{equation}\label{CPnresdue}
\langle 1_{\frak y^{(k)}}, 1_{\frak y^{(k')}}\rangle_{\text{\rm res}}
= T^{-\frac{n}{n+1}}e^{-\frac{2\pi\sqrt{-1}kn}{n+1}}\frac{\delta_{kk'}}{1+n}.
\end{equation}
Combining (\ref{ksfm}) and (\ref{CPnresdue}), we obtain
\begin{equation}\label{rescalcmm}
\langle \frak{ks}_{\text{\bf 0}}(\text{\bf p}_\ell),
\frak{ks}_{\text{\bf 0}}(\text{\bf p}_{\ell'})\rangle_{\text{\rm res}}
= \frac{1}{n+1}T^{-\frac{n}{n+1}}\sum_{k=0}^n
e^{-\frac{2\pi\sqrt{-1}kn}{n+1}} T^{\frac{\ell+\ell'}{n+1}}
e^{\frac{2\pi\sqrt{-1}(\ell+\ell')k}{n+1}}.
\end{equation}
It follows that (\ref{rescalcmm}) is $0$ unless $\ell+\ell' = n$ and
$$
\langle \frak{ks}_{\text{\bf 0}}(\text{\bf p}_\ell),
\frak{ks}_{\text{\bf 0}}(\text{\bf p}_{n-\ell})\rangle_{\text{\rm res}}
= 1 =
\langle \text{\bf p}_\ell,
\text{\bf p}_{n-\ell}\rangle_{\text{PD}_{\C P^n}}.
$$
Thus Theorem \ref{PDresidue} holds in this case.
\end{Example}
\begin{Remark}
There are various works in the case of $\C P^n$.
See \cite{taka,barani,gros}.
\end{Remark}
\begin{Example}
We consider the Hirzebruch surface $F_2(\alpha)$. We use the notation of 
Example \ref{ex:Hilexa}. In this case the full 
potential function for $\frak b  = \text{\bf 0}$ is calculated in 
\cite{auroux}, \cite{toric3} section 19 and \cite{fooo10} section 5 as follows.
\begin{equation}\label{F2poten}
\frak{PO}^{\text{\bf 0}}
=
y_1+ y_2
+ T^{2}y_1^{-1}y_2^{-2} 
+ T^{1-\alpha}(1+T^{2\alpha})y_2^{-1}.
\end{equation}
The valuation of the critical points are 
$$
(\frak v_T(\frak y_1),\frak v_T(\frak y_2)) = ((1-\alpha)/2,(1+\alpha)/2) = \text{\bf u}.
$$
It is the same for $4$ critical points.  Then using the variables $\overline y_i = y_i^{\text{\bf u}}$ we have
\begin{equation}\label{F2poten2}
\frak{PO}^{\text{\bf 0}}
= 
T^{(1-\alpha)/2}(\overline y_2 + (1+T^{2\alpha})\overline y_2^{-1})  +
T^{(1+\alpha)/2}(\overline y_1 +\overline y_1^{-1}\overline y_2^{-2}).
\end{equation}
(See Example 10.1.)
(We remark $\frak v_T(\overline y_i) = 0$.)
The critical point equation is
\begin{eqnarray}
0 &=& 1 - \overline y_1^{-2}\overline y_2^{-2}.  \label{crit1}\\
0 &=& 1 -2 T^{\alpha}\overline y_1^{-1}\overline y_2^{-3} 
- (1+T^{2\alpha})\overline y_2^{-2}.\label{crit2}
\end{eqnarray}
This has 4 solution.
\par
The Hessian matrix of (\ref{F2poten2}) is 
$$
\left[
\begin{matrix}
T^{(1+\alpha)/2}(\overline y_1+\overline y_1^{-1}\overline y_2^{-2})   &   2T^{(1+\alpha)/2} \overline y_1^{-1}\overline y_2^{-2} \\
& \\
2T^{(1+\alpha)/2}\overline y_1^{-1}\overline y_2^{-2}           &
\aligned
T^{(1-\alpha)/2}(\overline y_2&+(1+T^{2\alpha})(\overline y_2^{-1}) 
\\
&+ 4T^{(1+\alpha)/2}\overline y_1^{-1}\overline y_2^{-2}
\endaligned
\end{matrix}
\right]
$$
We can easily calculate the determinants of this matrix at the four 
solutions of (\ref{crit1}), (\ref{crit2}).  The determinants are $4T,4T,-4T,-4T$.
(See \cite{toric3} section 16 for the detail of the calculation.)
\par
The Hirzebruch suface $F_2(\alpha)$ is symplectomorphic to 
$S^2(1-\alpha) \times S^2(1-\alpha)$, where $S^2(1-\alpha)$ is 
the sphere $S^2$ with total area $1-\alpha$.
This fact is proved in \cite{fooo10} Proposition 5.1.
\par 
The quantum cohomology of $S^2(1-\alpha) \times S^2(1-\alpha)$ 
is generated by $x,y$ that correspond to the 
fundamental class of the factors $S^2(1-\alpha)$ and  $S^2(1+\alpha)$ respectively.
The fundamental relations among them are 
$$
x^2 = T^{1-\alpha} 1, \qquad y^2 = T^{1+\alpha} 1, \qquad xy = yx.
$$
We put
$$
e_{\pm} = \frac{1}{2}T^{{-(1-\alpha)}/2}(T^{(1-\alpha)/2} \pm x), 
\quad 
f_{\pm} =   \frac{1}{2}T^{{-(1+\alpha)}/2}(T^{(1+\alpha)/2} \pm y).
$$
Then $e_-f_-, e_-f_+, e_+f_-, e_+f_+$
are the units of the 4 direct product factors of $QH(S^2(1-\alpha) \times S^2(1+\alpha);\Lambda)$.
We have 
$$\int_{S^2(1-\alpha) \times S^2(1+\alpha)} e_{-}f_- e_{-}f_-
= 
\frac{1}{4T}
$$
Hence 
$$
\langle e_-f_-, e_-f_-\rangle_{\text{\rm PD}_{S^2(1-\alpha) \times S^2(1+\alpha)}} = \frac{1}{4T}.
$$
We obtain $-1/4T,-1/4T,1/4T$ from $e_-f_+, e_+f_-, e_+f_+$ in the same way.
Thus, Theorem \ref{PDresidue} holds in this case also.
\end{Example}
\begin{Example}
We take the monotone toric blow up of $\C P^2$ at one point, whose 
moment polytope is 
$\{(u_1,u_2) \mid 0 \ge u_1,u_2, u_1+u_2 \le 1, u_1 \le 2/3\}$.
Its unique monotone fiber is $\text{\bf u} = (1/3,1/3)$.
We put $\overline y_1 = y^{\text{\bf u}}_1$, $\overline y_2 = y^{\text{\bf u}}_2$.
Then the potential function  (for $\frak b = \text{\bf 0}$) is:
\begin{equation}
\frak{PO}^{\text{\bf 0}}
= 
T^{1/3}(\overline y_1 + \overline y_2 + (\overline y_1\overline y_2)^{-1} + \overline y_1^{-1}).
\end{equation}
The condition for $(\overline y_1,\overline y_2)$ to be critical gives rise to the equation:
\begin{equation}\label{230}
1 - \overline y_1^{-2}\overline y_2^{-1} - \overline y_1^{-2} = 0, \quad
1 - \overline y_1\overline y_2^2 = 0.
\end{equation}
We put $\overline y_2 = z$. Then $\overline y_1 = 1/z$ and 
\begin{equation}\label{231}
z^4 + z^3 -1 = 0.
\end{equation}
By Theorem \ref{cliffordZ} (3) we have
\begin{equation}
\aligned
Z(0,(\overline y_1,\overline y_2))
= T^{2/3} \text{\rm det}
\left[
\begin{matrix}
\overline y_1 + (\overline y_1\overline y_2)^{-1} + \overline y_1^{-1}   &   (\overline y_1\overline y_2)^{-1} \\
(\overline y_1\overline y_2)^{-1}           &
\overline y_2 + (\overline y_1\overline y_2)^{-1}
\end{matrix}
\right] 
= T^{2/3}\frac{4-z^3}{z}.
\endaligned\nonumber
\end{equation}
Let $z_i$ ($i=1,2,3,4$) be the 4 solutions of (\ref{231}). 
 Then the left hand side of
(\ref{hennasiki}) becomes:
\begin{equation}\label{hennasiki2}
T^{-2/3} \sum_{i=1}^4 \frac{z_i}{4-z_i^3}.
\end{equation}
We can directly check that (\ref{hennasiki2})$=0$.
(See \cite{toric3} Example 2.35.)
Thus we checked that Corollary \ref{Corresi}  holds in this case.
\end{Example}

%
%
\end{document}